\documentclass{amsart}

\usepackage{amsmath,amscd,xypic,amssymb,combelow,color,enumitem,graphicx,float,epigraph}
\usepackage{comment}

\usepackage[normalem]{ulem}
\usepackage[dvipsnames]{xcolor}
\makeatletter
\def\squiggly{\bgroup \markoverwith{\textcolor{black}{\lower3.5\p@\hbox{\sixly \char58}}}\ULon}

\usepackage{xr-hyper}
\usepackage[
pdftex,
bookmarks=false,
colorlinks=true,
debug=true,
pdfnewwindow=true]{hyperref}

\makeatletter
  \@addtoreset{equation}{section}
\makeatother

\newtheorem{theorem}[subsection]{Theorem}  

\newtheorem{proposition}[subsection]{Proposition}
\newtheorem{lemma}[subsection]{Lemma}
\newtheorem{corollary}[subsection]{Corollary}

\newtheorem{definition}[subsection]{Definition}
\newtheorem{claim}[subsection]{Claim}
\newtheorem{example}[subsection]{Example}
\newtheorem{remark}[subsection]{Remark}

\def\fb{{\mathfrak{b}}}
\def\fg{{\mathfrak{g}}}

\def\fn{{\mathfrak{n}}}
\def\fm{{\mathfrak{m}}}

\def\fsl{{\mathfrak{sl}}}
\def\fgl{{\mathfrak{gl}}}

\def\BC{{\mathbb{C}}}

\def\BN{{\mathbb{N}}}

\def\BP{{\mathbb{P}}}

\def\BQ{{\mathbb{Q}}}
\def\BZ{{\mathbb{Z}}}

\def\woo{\widehat{\otimes}}

\def\CA{{\mathcal{A}}}
\def\CB{{\mathcal{B}}}

\def\CO{{\mathcal{O}}}
\def\CP{{\mathcal{P}}}

\def\CS{{\mathcal{S}}}

\def\CV{{\mathcal{V}}}

\def\ph{\varphi}

\def\sym{\textrm{Sym}}

\def\UU{U_q(L\fg)}
\def\UUp{U_q^+(L\fg)}
\def\UUpm{U^\pm_q(L\fg)}
\def\UUo{U_q^\circ(L\fg)}
\def\UUm{U_q^-(L\fg)}
\def\UUg{U_q^\geq(L\fg)}
\def\UUl{U_q^\leq(L\fg)}

\def\UUi{U_{q_1,q_2}(\ddot{\fgl}_1)}
\def\UUip{U_{q_1,q_2}^+(\ddot{\fgl}_1)}

\def\UUim{U_{q_1,q_2}^-(\ddot{\fgl}_1)}
\def\UUig{U_{q_1,q_2}^\geq(\ddot{\fgl}_1)}
\def\UUil{U_{q_1,q_2}^\leq(\ddot{\fgl}_1)}

\def\tUU{\widetilde{U}_q(L\fg)}
\def\tUUp{\widetilde{U}_q^+(L\fg)}
\def\tUUpm{\widetilde{U}^\pm_q(L\fg)}

\def\tUUm{\widetilde{U}_q^-(L\fg)}
\def\tUUg{\widetilde{U}_q^\geq(L\fg)}
\def\tUUl{\widetilde{U}_q^\leq(L\fg)}

\def\br{{\mathbf{r}}}

\def\bs{\boldsymbol{\varsigma}}

\def\cc{{\mathbb{C}^I}}

\def\nn{{\mathbb{N}^I}}
\def\zz{{\mathbb{Z}^I}}

\def\balpha{{\boldsymbol{\alpha}}}

\def\bom{{\boldsymbol{\la}}}

\def\bpsi{{\boldsymbol{\psi}}}
\def\btau{{\boldsymbol{\tau}}}
\def\bom{{\boldsymbol{\omega}}}

\def\bm{{\boldsymbol{m}}}
\def\bn{{\boldsymbol{n}}}
\def\bp{{\boldsymbol{p}}}

\def\bx{{\boldsymbol{x}}}
\def\by{{\boldsymbol{y}}}

\def\b0{{\boldsymbol{0}}}

\def\loccit{\emph{loc.~cit.~}}

\def\hdeg{\text{hdeg }}
\def\vdeg{\text{vdeg }}

\def\wI{\widehat{I}}

\def\UUaff{U_q(\widehat{\fg})_{c=1}}

\def\UUaffg{U_q(\widehat{\fb}^+)_{c=1}}
\def\UUaffl{U_q(\widehat{\fb}^-)_{c=1}}
\def\UUaffgl{U_q(\widehat{\fb}^\pm)_{c=1}}

\def\wI{\widehat{I}}
\def\wfg{\widehat{\fg}}

\def\bx{\boldsymbol{x}}

\def\vac{|\varnothing\rangle}
\def\lead{\text{lead}}

\def\ochi{\mathring{\chi}}
\def\oJ{\mathring{J}}
\def\oCS{\mathring{\CS}}
\def\omu{\mathring{\mu}}
\def\onu{\mathring{\nu}}

\def\oL{\mathring{L}}

\def\ord{\textbf{ord }}
\def\eord{\textbf{\emph{ord }}}

\def\binfty{\boldsymbol{\infty}}

\begin{document}

\title[Category $\CO$ for quantum loop algebras]{\Large{\textbf{Category $\CO$ for quantum loop algebras}}} 

\author[Andrei Negu\cb t]{Andrei Negu\cb t}

 \address{École Polytechnique Fédérale de Lausanne (EPFL), Lausanne, Switzerland \newline \text{ } \ \ Simion Stoilow Institute of Mathematics (IMAR), Bucharest, Romania} 

\email{andrei.negut@gmail.com}

\maketitle

\begin{abstract} 
	
We generalize the Hernandez-Jimbo category $\CO$ of representations of Borel subalgebras of quantum affine algebras to the case of quantum loop algebras for arbitrary Kac-Moody $\fg$ (as well as related algebras, such as quantum toroidal $\fgl_1$). Moreover, we give explicit realizations of all simple modules, and devise tools for the computation of $q$-characters that are new even for $\fg$ of finite type. Our techniques allow us to generalize classic results of Frenkel-Hernandez, Frenkel-Mukhin, Hernandez-Jimbo and Hernandez-Leclerc, as well as prove conjectures of Feigin-Jimbo-Miwa-Mukhin and Mukhin-Young.	
	
\end{abstract}

\bigskip

\epigraph{\emph{\`A Ansel, avec affection, amour, anticipation}}

\bigskip

\section{Introduction}
\label{sec:intro}

\medskip

\subsection{Category $\CO$} 
\label{sub:intro quantum}

Let $\fg$ be a simple finite-dimensional complex Lie algebra, and fix $q \in \BC^*$ which is not a root of unity. Chari-Pressley (\cite{CP}) classified finite-dimensional representations of the quantum affine algebra $\UUaff$ in terms of so-called $\ell$-weights (we let $I$ denote a set of simple roots of $\fg$)
\begin{equation}
	\label{eqn:intro ell weight}
	\bpsi = \left( \psi_i(z)  \right)_{i \in I} \in \left( \BC[[z^{-1}]]^* \right)^I
\end{equation}
 Hernandez-Jimbo (\cite{HJ}) extended this framework to their so-called category $\CO$ of representations of the Borel subalgebra of the quantum affine algebra
\begin{equation}
	\label{eqn:intro quantum borel}
	\UUaff \supset \UUaffg \curvearrowright V
\end{equation}
which decompose into finite-dimensional weight spaces (see Definition \ref{def:category o affine}). The building blocks of category $\CO$ are the simple modules $L(\bpsi)$ indexed by rational $\ell$-weights, i.e. those for which all the power series $\psi_i(z)$ in \eqref{eqn:intro ell weight} are expansions of rational functions. While category $\CO$ has been immensely influential in the theory of integrable systems, cluster algebras and categorification (see for instance \cite{Bi, FH, FHOO, GHL, HL Cluster, HL Borel, KKOP, Q}), its theory has so far been beset by two limitations.

\medskip

\begin{enumerate}[leftmargin=*]
	
\item Though it is known how to generalize $\UUaff$ to the setting of any Kac-Moody Lie algebra $\fg$ (using the so-called Drinfeld new realization $\UU$, see Definition \ref{def:quantum loop}), it was not clear how to generalize the Borel subalgebra in \eqref{eqn:intro quantum borel} (however, see an alternative approach for toroidal types in \cite{La}).

\medskip

\item The usual constructions of the simple modules $L(\bpsi)$ involve ingenious uses of tensor products, shifted quantum loop algebras, limits, analytic continuations and other techniques. However, an explicit description of the underlying vector space of $L(\bpsi)$ was not known uniformly in $\bpsi$.

\end{enumerate}

\medskip

\noindent In the present paper, we present solutions for both these issues. 

\begin{enumerate}[leftmargin=*]

\item We define an explicit subalgebra $\CA^{\geq} \subset \UU$ for any Kac-Moody Lie algebra $\fg$, which matches the Borel subalgebra in \eqref{eqn:intro quantum borel} when $\fg$ is of finite type. We extend the fundamental theory of category $\CO$ to this new level of generality.

\medskip 

\item We present any simple module $\CA^{\geq} \curvearrowright L(\bpsi)$ as an explicit subquotient of $\CA^{\geq}$. This opens the door to new computational tools for $L(\bpsi)$, which we use to prove certain theorems pertaining to characters, which were conjectural even for $\fg$ of finite type.

\end{enumerate} 

\medskip 

\noindent The explicit construction of the representation $L(\bpsi)$ is given in Corollary \ref{cor:simple module}. It has numerous applications, such as to  generalize the grading on positive prefundamental modules from \cite[Theorem 6.1]{FH}. In more detail, consider any polynomial $\ell$-weight
\begin{equation}
\label{eqn:polynomial intro}
\btau = (\tau_i(z))_{i \in I} \in \left(\BC^* + z^{-1}\BC[z^{-1}] \right)^I
\end{equation}
In Subsection \ref{sub:integral}, we will endow $L(\btau)$ with a natural grading by $(\bom - \nn) \times \BN$, where $\bom = \lead(\btau)$ is defined as in \eqref{eqn:psi to gamma} (note that in the present paper, $\BN$ contains 0). We denote this grading by $L(\btau) = \oplus_{\bn \in \nn, d \in \BN} L(\btau)_{\bom - \bn,d}$. In Proposition \ref{prop:bigrading}, we will show that the action $\CA^{\geq} \curvearrowright L(\btau)$ interacts with this grading as follows
\begin{equation}
		\label{eqn:act 1 intro}
		F \cdot L(\btau)_{\bom - \bn,d} \subseteq L(\btau)_{\bom - \bn+\text{hdeg } F, d+\text{vdeg } F}
	\end{equation}
	\begin{equation}
		\label{eqn:act 2 intro}
		\left[\frac {\ph_j^+(z)}{\tau_j(z)} \right]_{z^{-u}} \cdot L(\btau)_{\bom - \bn,d} \subseteq L(\btau)_{\bom - \bn,d+u}
	\end{equation}
	\begin{equation}
		\label{eqn:act 3 intro}
		E \cdot L(\btau)_{\bom - \bn,d} \subseteq \bigoplus_{\bullet = 0}^{\br \cdot \text{\hdeg E}} L(\btau)_{\bom - \bn+\text{hdeg }E, d+\text{vdeg } E - \bullet}
	\end{equation}
for all $F, \ph_j^+(z), E$ in the creating, diagonal, annihilating part of $\CA^{\geq}$, respectively (see \eqref{eqn:dot product}, \eqref{eqn:hdeg vdeg} and especially \eqref{eqn:slope zero plus} for our conventions in the formulas above).

\medskip

\subsection{$q$-characters} 
\label{sub:q-char intro}

Consider an arbitrary Kac-Moody Lie algebra $\fg$, associated to a symmetrizable Cartan matrix $\{d_{ij}\}_{i,j \in I}$ (see \eqref{eqn:cartan matrix}). Most interesting applications of category $\CO$ (see \cite{Bi, FH, FHOO, GHL, HL Cluster, HL Borel, KKOP, Q} and many other works) occur through their $q$-characters. These were originally defined in \cite{FR} for finite-dimensional representations, and were extended in \cite{HJ} for any representation \eqref{eqn:intro quantum borel} via the formula 
\begin{equation}
	\label{eqn:intro q-character}
	\chi_q(V) =  \sum_{\bpsi \in \left( \BC[[z^{-1}]]^* \right)^I} \dim_{\BC}(V_{\bpsi}) [\bpsi]
\end{equation}
where $[\bpsi]$ are formal symbols, and $V_{\bpsi}$ are generalized eigenspaces for the subalgebra
\begin{equation}
	\label{eqn:intro commutative family}
\BC\left[\ph_{i,0}^+, \ph_{i,1}^+,\ph_{i,2}^+,\dots\right]_{i \in I} \subset \CA^\geq \text{ of \eqref{eqn:***}}
\end{equation}
acting on $V$. The most important $q$-characters are those of the simple modules. To this end, we give in \eqref{eqn:q-character general} the following formula for all Kac-Moody $\fg$ and rational $\bpsi$
\begin{equation}
	\label{eqn:intro q-character simple}
	\chi_q(L(\bpsi)) = [\bpsi] \sum_{\bn \in \nn} \sum_{\bx \in \BC^\bn} \mu_{\bx}^{\bpsi} \left[\left( \prod_{i \in I} \prod_{a=1}^{n_i} \frac {z - x_{ia}q^{d_{ij}}}{zq^{d_{ij}}-x_{ia}} \right)_{j \in I} \right] 
\end{equation}
for certain multiplicities $\mu_{\bx}^{\bpsi} \in \BN$ that are defined as dimensions of certain explicit vector spaces in \eqref{eqn:q-character general multiplicity}. In the right-hand side, the product of symbols $[\bpsi]$ is taken component-wise in terms of $j \in I$ (see \eqref{eqn:product}). The outer sum in \eqref{eqn:intro q-character simple} goes over $\bn = (n_i)_{i \in I} \in \nn$, while the inner sum goes over
\begin{equation}
	\label{eqn:intro x}
	\bx = (x_{i1},\dots,x_{in_i})_{i \in I} \in \BC^{\bn} := \prod_{i \in I} \BC^{n_i}/S_{n_i}
\end{equation}
Our general formula for $\mu_{\bx}^{\bpsi}$ is not always easy to calculate, but we can get better formulas with the following modification $\oL(\bpsi)$ of $L(\bpsi)$ (Definition \ref{def:variant} and \eqref{eqn:q-character general variant multiplicity}):
\begin{equation}
	\label{eqn:intro q-character variant}
	\chi_q(\oL(\bpsi)) = [\bpsi] \sum_{\bn \in \nn} \sum_{\bx \in \BC^\bn} \omu_{\bx}^{\bpsi} \left[\left( \prod_{i \in I} \prod_{a=1}^{n_i} \frac {z - x_{ia}q^{d_{ij}}}{zq^{d_{ij}}-x_{ia}} \right)_{j \in I} \right] 
\end{equation}
For $\fg$ of finite type, our modification does not change anything, because for all $\bpsi$
\begin{equation}
	\label{eqn:intro are equal}
L(\bpsi) = \oL(\bpsi)
\end{equation}
(cf. \eqref{eqn:are equal}). In what follows, for any rational $\ell$-weight $\bpsi = (\psi_i(z))_{i \in I}$, let $(\ord \bpsi) \in \zz$ be the $I$-tuple of the orders of the pole at $z=0$ of the rational functions $\psi_i(z)$.

\medskip

\begin{theorem}
	\label{thm:main}
	
Consider any Kac-Moody Lie algebra $\fg$ and any rational $\ell$-weight $\bpsi$. For any $\by \in (\BC^*)^{\bm}$ and $\bn \geq \bm$, we let $\bx = (\by,\b0_{\bn-\bm}) \in \BC^{\bn}$ and we claim that
	\begin{equation}
		\label{eqn:intro factor}
		\omu_{\bx}^{\bpsi} = \omu_{\by}^{\bpsi} \onu^{\eord\bpsi}_{\bn - \bm}
	\end{equation}
	with $\onu^{\eord \bpsi}_{\bn - \bm}$ as in \eqref{eqn:multiplicity nu}. In other words, the modified $q$-character factors as
	\begin{equation}
		\label{eqn:intro factor series}
		\chi_q(\oL(\bpsi)) = \chi^{\neq 0}_q(\oL(\bpsi)) \cdot \ochi^{\eord \bpsi}
	\end{equation}
	where $\chi^{\neq 0}_q(\oL(\bpsi))$ is the RHS of \eqref{eqn:intro q-character variant} with $\sum_{\bx \in (\BC^*)^{\bn}}$ instead of $\sum_{\bx \in \BC^{\bn}}$, and
\begin{equation}
			\label{eqn:intro ochi integral}
			\ochi^{\br} = \sum_{\bn \in \nn} \onu^{\br}_{\bn} [-\bn]
		\end{equation}
	for all $\br \in \zz$, where $[-\bn]$ is interpreted as a constant $\ell$-weight, see Definition \ref{def:ell weight general}.

\end{theorem}

\medskip 

\noindent Our proof yields explicit formulas for $\omu_{\by}^{\bpsi}$ and $\onu^{\br}_{\bp}$ in \eqref{eqn:multiplicity mu} and \eqref{eqn:multiplicity nu} respectively, and thus provides a new approach for the computation of $q$-characters. For instance, we will show that $\omu_{\by}^{\bpsi} \neq 0$ for $\by = (y_{ia})^{i \in I}_{1 \leq a \leq m_i} \in (\BC^*)^{\bm}$ only if

\medskip

\begin{itemize}
	
	\item $y_{ia}$ is one of the finitely many poles of $\psi_i(z)$, or 
	
	\medskip
	
	\item $y_{ia} = y_{jb} q^{-d_{ij}}$ for some $(j,b) < (i,a)$.
	
\end{itemize}

\medskip

\noindent for all $i$,$a$, with respect to some total order on $\{(i,a)\}^{i \in I}_{1\leq a \leq m_i}$. This gives a combinatorial criterion for finding the summands that are allowed to appear in \eqref{eqn:intro q-character variant}. 

\medskip

\subsection{Computations and corollaries}
\label{sub:computational}

For any Kac-Moody Lie algebra $\fg$, our approach provides a set of tools for computing $q$-characters, which are new even for $\fg$ of finite type (complementing the myriad known combinatorial, algebraic and geometric approaches to $q$-characters, see  \cite{CM, FM, H, HL Cluster, MY, NN, Na 1, Na 2} and many other works). In particular, since polynomials have no poles in $\BC^*$, the combinatorial criterion at the end of the previous Subsection gives
\begin{equation}
\label{eqn:intro q-character polynomial variant}
\chi_q(\oL(\btau)) = [\btau]\ochi^{\ord \btau}
\end{equation}
for any polynomial $\ell$-weight $\btau$. Similarly, we show in Subsection \ref{sub:first corollary} that
\begin{equation}
	\label{eqn:intro q-character polynomial}
	\chi_q(L(\btau)) = [\btau]\chi^{\ord \btau}
\end{equation}
where $\chi^{\br}$ is defined for any $\br \in \zz$ in \eqref{eqn:chi r}; formula \eqref{eqn:intro q-character polynomial} had been known to experts by other means for $\fg$ of finite type, and we prove it for general $\fg$.

\medskip

\begin{proposition}
\label{prop:no pole}
	
For all $\br \in - \nn$, we have $\chi^{\br} =	\ochi^{\br} = 1$.

\end{proposition}

\medskip

\noindent In the notation of \cite{FR}, the $\ell$-weights which appear in the RHS of \eqref{eqn:intro q-character variant} satisfy
\begin{equation}
\label{eqn:old a}
\left( \frac {z - x_{ia}q^{d_{ij}}}{zq^{d_{ij}}-x_{ia}} \right)_{j \in I} = A_{i,x_{ia}}^{-1}
\end{equation}
Thus, if a rational $\ell$-weight is regular (i.e. has the property that all the rational functions $\psi_i(z)$ are regular at $z = 0$), then by combining \eqref{eqn:intro factor series} and Proposition \ref{prop:no pole} we obtain the following generalization of \cite[Theorem 4.1]{FM}.

\medskip

\begin{corollary}
	\label{cor:regular}
	
For any Kac-Moody $\fg$ and any regular $\ell$-weight $\bpsi$, we have
	\begin{equation}
		\label{eqn:intro regular}
\chi_q(\oL(\bpsi)) = \chi^{\neq 0}_q(\oL(\bpsi))
	\end{equation}
	
\end{corollary}

\medskip

\noindent An important consequence of Theorem \ref{thm:main} is the following result, which generalizes \cite[Conjecture 7.15]{HL Borel}. For $\fg$ of finite type, this result had been proved for so-called reachable simple modules using cluster categorification in \cite{KKOP, Q}, and for all simple modules using shifted quantum loop algebras in \cite{H Shifted}. 

\medskip

\begin{corollary}
\label{cor:main}

For any Kac-Moody $\fg$, consider a regular $\ell$-weight $\bpsi$, and a polynomial $\ell$-weight $\btau$ which is $\bpsi$-monochrome in the sense of Definition \ref{def:monochrome}. Then
\begin{equation}
\label{eqn:intro conjecture}
\chi_q(\oL(\bpsi \btau)) = \chi^{\btau}_q(\oL(\bpsi)) \cdot [\btau] \ochi^{\eord \bpsi \btau}
\end{equation}
where the truncated $q$-character $\chi^{\btau}_q(\oL(\bpsi))$ is defined in \eqref{eqn:truncated q-character finite}. 

\end{corollary}

\medskip

\subsection{The Mukhin-Young conjecture}
\label{sub:my intro}

As explained at the end of Subsection \ref{sub:intro quantum}, for any polynomial $\ell$-weight $\btau$, we endow the simple module $L(\btau)$ with a grading by $d \in \BN$. One can do the same for $\oL(\btau)$, and define the refined $q$-characters
\begin{align}
&\chi_{q}^{\text{ref}}(L(\btau)) = [\btau] \sum_{\bn \in \nn} \sum_{d=0}^{\infty} \dim_{\BC} \Big( L(\btau)_{\bom-\bn,d}  \Big) [-\bn] v^d \label{eqn:refined intro q-character polynomial} \\
&\chi_{q}^{\text{ref}}(\oL(\btau)) = [\btau] \sum_{\bn \in \nn} \sum_{d=0}^{\infty} \dim_{\BC} \Big( \oL(\btau)_{\bom-\bn,d}  \Big) [-\bn] v^d \label{eqn:refined intro q-character polynomial variant} 
\end{align} 
(setting $v=1$ recovers the usual $q$-characters of $L(\btau)$ and $\oL(\btau)$). We will soon show that the following result refines and proves a conjecture of Mukhin-Young (\cite{MY}).  

\medskip 

\begin{theorem}
\label{thm:my refined}

Let $\fg$ be of finite type with set of positive roots $\Delta^+$. Then we have
\begin{equation}
\label{eqn:my refined}
\chi_{q}^{\emph{ref}}(L(\btau)) = [\btau] \prod_{\balpha \in \Delta^+} \prod_{d=1}^{(\eord \btau) \cdot \balpha} \frac 1{1-[-\balpha]v^d}
\end{equation}
for all polynomial $\ell$-weights $\btau$ with $\eord \btau \in \BZ_{>0}^I$, where $\cdot$ is defined in \eqref{eqn:dot product}. Our proof generalizes to any Kac-Moody Lie algebra $\fg$, in the sense of Remark \ref{rem:general}.

\end{theorem}

\medskip

\noindent Formula \eqref{eqn:my refined} actually holds for all polynomial $\ell$-weights $\btau$, but we would need certain technical modifications of our approach to deal with $\ord \btau$ having some entries 0 (these will be studied in upcoming work). Setting $v = 1$ in \eqref{eqn:my refined} gives the following formula for the usual character that features in \eqref{eqn:intro q-character polynomial}, for any $\br \in \BZ_{>0}^I$
\begin{equation}
\label{eqn:my}
\chi^{\br} = \prod_{\balpha \in \Delta^+} \left(  \frac 1{1-[-\balpha]} \right)^{\br \cdot \balpha}
\end{equation}
(still for $\fg$ of finite type). The characters \eqref{eqn:my} are additive in $\br$, due to the simplicity of tensor products from \cite[Theorem 4.11]{FH}. Therefore, we conclude that
\begin{equation}
\label{eqn:my i}
\chi^{\bs^i} = \prod_{\balpha \in \Delta^+} \left(  \frac 1{1-[-\balpha]} \right)^{\text{mult}_{\balpha_i}(\balpha)}
\end{equation}
where $\{\bs^i\}_{i\in I}$ is the standard basis of $\zz$, and $\text{mult}_{\balpha_i}(\balpha)$ is the multiplicity of the $i$-th simple root in $\balpha$. Formula \eqref{eqn:my i} was conjectured in \cite{MY} (to be more precise, \loccit dealt with the more general minimal affinizations, which in the particular case of fundamental weights are known to reproduce the characters $\chi^{\bs^i}$) and was proved case-by-case except for certain $i$ in type $E$ (\cite{L, LN, Naoi}, see also \cite{JKP, W}).

\medskip

\subsection{Quantum toroidal $\fgl_1$} 
\label{sub:intro toroidal}

The techniques and results in the present paper also apply to related algebras such as quantum toroidal $\fgl_1$ (also known as the Ding-Iohara-Miki algebra, see \cite{DI, M}). The appropriate Borel subalgebra
$$
\CA^\geq \subset \UUi
$$
as well as its natural category $\CO$ of representations, were introduced in \cite{FJMM}. In the present context, simple $\CA^\geq$ modules are parameterized by highest $\ell$-weights
$$
\bpsi = (\psi(z), m) \in \BC[[z^{-1}]]^* \times \BC
$$
where the second component $m$ amounts to an overall grading shift (for quantum toroidal $\fgl_1$, the Cartan elements do not encode the natural weight spaces of representations). Our explicit constructions of simple modules with highest $\ell$-weight $\bpsi$ allow us to obtain formulas for their $q$-characters
\begin{multline}
\label{eqn:intro char toroidal}
\chi_q(L(\bpsi)) = [\bpsi] \\ \sum_{n \in \BN} \sum_{\bx \in \BC^n/S_n} \mu_{\bx}^{\bpsi} \left[ \left( \prod_{a=1}^n \frac {(z - x_a q_1)(z - x_a q_2)(z q_1 q_2 - x_a)}{(z q_1 - x_a)(z q_2 - x_a)(z - x_a q_1 q_2)}, -n \right) \right]
\end{multline}
with $\mu_{\bx}^{\bpsi}$ given in \eqref{eqn:multiplicity toroidal} for any $\bx = (x_1,\dots,x_n)$. For a polynomial $\ell$-weight 
\begin{equation}
\label{eqn:intro integral toroidal}
\btau = \Big( \tau(z) = a_0+a_1z^{-1}+\dots+a_{r-1}z^{-r+1}+a_rz^{-r}, m \Big)
\end{equation}
with $a_0,a_r \neq 0$, we may introduce a grading $L(\btau) = \oplus_{n,d = 0}^{\infty} L(\btau)_{m-n,d}$, by analogy with the discussion at the end of Subsection \ref{sub:intro quantum}. Therefore, we may consider the following refinement of the character associated to polynomial $\ell$-weights
\begin{equation}
	\label{eqn:intro char polynomial toroidal refined}
	\chi_{q}^{\text{ref}}(L(\btau)) = [\btau] \sum_{n=0}^{\infty} \sum_{d=0}^{\infty} \dim_{\BC} \Big( L(\btau)_{m-n,d}  \Big) h^n v^d 
\end{equation}
where $h$ denotes the $\ell$-weight $[(1,-1)]$. With this in mind, the following proves \cite[Conjecture 4.20]{FJMM} (note that our conventions are dual to those of \emph{loc. cit.}).

\medskip

\begin{theorem}
\label{thm:toroidal}

For a polynomial $\ell$-weight $\btau$ as in \eqref{eqn:intro integral toroidal}, we have
\begin{equation}
\label{eqn:tau q-char intro}
\chi_{q}^{\emph{ref}}(L(\btau)) = [\btau] \prod_{n=1}^{\infty} \prod_{d=1}^{rn} \frac 1{1-h^n v^d}
\end{equation}
Moreover, we give analogues of formulas \eqref{eqn:act 1 intro}-\eqref{eqn:act 3 intro}, see \eqref{eqn:act tor 1}-\eqref{eqn:act tor 3}.

\end{theorem}

\medskip

\subsection{Shuffle algebras} 
\label{sub:intro shuffle}

Let us now explain our techniques in a nutshell. While $\UUaff$ is only defined for $\fg$ of finite type, the quantum loop algebra $\UU$ is defined for any Kac-Moody Lie algebra $\fg$ (\cite{Dr}). However, to get a handle on $\UU$ for any $\fg$, we will resort to its shuffle algebra incarnation (\cite{E1, E2}, following \cite{FO})
$$
\Upsilon : \UU \xrightarrow{\sim} \CS^+ \otimes \frac {\BC \left[\ph^{\pm}_{i,0}, \ph^\pm_{i,1}, \ph^\pm_{i,2},\dots \right]_{i \in I}}{\ph_{i,0}^+ \ph_{i,0}^- - 1} \otimes \CS^-
$$
Above, $\CS^\pm$ are known as shuffle algebras: they are certain subsets of symmetric Laurent polynomial rings in arbitrarily many variables $\{z_{ia}\}_{i \in I, a \geq 1}$, which are explicitly understood for $\fg$ of finite type (\cite{E1, E2, FO, NT}) and for $\fg$ of simply laced Kac-Moody type (\cite{N Loop}), as well as in the related case of quantum toroidal $\fgl_1$. Let 
\begin{equation}
\label{eqn:***}
\CA^{\geq} = \CS_{\geq 0}^+ \otimes \BC \left[\ph^+_{i,0}, \ph^+_{i,1}, \ph^+_{i,2},\dots \right]_{i \in I} \otimes \CS^-_{<0}
\end{equation}
where $\CS_{\geq 0}^+$ and $\CS^-_{<0}$ denote certain subalgebras of $\CS^+$ and $\CS^-$ (respectively) defined in terms of the slope condition of Definition \ref{def:slope}. With this in mind, we show that
\begin{equation}
\label{eqn:intro l}
L(\bpsi) = \CS^-_{<0} \Big/ J(\bpsi)
\end{equation}
has a natural structure of a simple $\CA^{\geq}$ module for an arbitrary highest $\ell$-weight $\bpsi$, where $J(\bpsi)$ is the subset of $\CS_{<0}^-$ defined in Proposition \ref{prop:ideal}. All the results in the previous Subsections are proved by analyzing the objects in \eqref{eqn:intro l}, as well as their modifications from \eqref{eqn:ol}
\begin{equation}
	\label{eqn:intro ol}
	\oL(\bpsi) = \CS^-_{<0} \Big/ \oJ(\bpsi)
\end{equation}
The expressions \eqref{eqn:intro q-character simple} and \eqref{eqn:intro q-character variant} keep track of the graded dimensions of the vector spaces \eqref{eqn:intro l} and \eqref{eqn:intro ol}, respectively. In general, the defining property of $\oJ(\bpsi)$ is easier to work with than that of $J(\bpsi)$, but the two coincide for $\fg$ of finite type.

\medskip

\subsection{Plan of the paper}

In Section \ref{sec:quantum}, we recall the quantum affine $\UUaff$ and quantum loop $\UU$ algebras, as well as the definition of category $\CO$ for $\fg$ of finite type from \cite{HJ}.  In Section \ref{sec:shuffle}, we recall shuffle algebras and the explicit slope factorizations of $\UU$ that they provide. This allows us to define a subalgebra
$$
\CA^\geq \subset \UU
$$
for general Kac-Moody $\fg$, which is isomorphic to $\UUaffg \subset \UUaff$ for $\fg$ of finite type (see Proposition \ref{prop:slope affine}). In Section \ref{sec:representation}, we prove fundamental facts about category $\CO$ in the shuffle algebra language, and establish the claims at the end of Subsection \ref{sub:intro quantum}. In Section \ref{sec:characters}, we prove our main results on $q$-characters: Theorems \ref{thm:main} and \ref{thm:my refined}, as well as Corollaries \ref{cor:regular} and \ref{cor:main}. In Section \ref{sec:toroidal}, we present analogues of our constructions for quantum toroidal $\fgl_1$, and use them to prove Theorem \ref{thm:toroidal}.

\medskip

\subsection{Acknowledgements} I would like to thank David Hernandez for many important conversations about quantum affinizations and their representation theory, in particular for the suggestion to pursue the contents of Section \ref{sec:toroidal}. I would also like to thank Boris Feigin for many inspiring conversations on quantum toroidal and shuffle algebras over the years.

\bigskip

\section{Quantum loop and quantum affine algebras}
\label{sec:quantum}

\medskip

\subsection{Basic definitions}
\label{sub:basic}

Fix a finite set $I$ and a symmetrizable Cartan matrix
\begin{equation}
\label{eqn:cartan matrix}
\left(\frac {2d_{ij}}{d_{ii}} \in \BZ\right)_{i,j \in I}
\end{equation}
where $d_{ii}$ are even positive integers with gcd 2, and $d_{ij} = d_{ji}$ are non-positive integers. This data corresponds to a Kac-Moody Lie algebra $\fg$, which we will refer to as the ``type" of all algebras considered hereafter. The set $I$ should be interpreted as a set of simple roots of $\fg$, and $\zz$ should be interpreted as the root lattice. The Cartan matrix \eqref{eqn:cartan matrix} is called symmetric, or equivalently $\fg$ is called simply laced, if
\begin{equation}
\label{eqn:symmetric matrix}
d_{ii} = 2, \quad \ \forall i \in I
\end{equation}
Knowledge of the integers $\{d_{ij}\}_{i,j \in I}$ provides a symmetric bilinear pairing 
\begin{equation}
\label{eqn:symmetric pairing}
\cc \otimes \cc \xrightarrow{( \cdot,\cdot )} \BC, \qquad (\bs^i, \bs^j) = d_{ij}
\end{equation}
where $\bs^i = \underbrace{(0,\dots,0,1,0,\dots,0)}_{1\text{ on the }i\text{-th position}}$. For any $\bm = (m_i)_{i \in I}, \bn = (n_i)_{i \in I} \in \cc$, we let
\begin{equation}
	\label{eqn:dot product}
	\bm \cdot \bn = \sum_{i \in I} m_i n_i
\end{equation}
and abbreviate $|\bn| = (1,\dots,1)\cdot \bn$. Consider the partial order on $\cc$ given by $\bm < \bn$ if $\bn - \bm \in \nn \backslash \b0$, where we write $\b0 = (0,\dots,0)$ (recall that $\BN$ contains 0).

\medskip

\subsection{The pre-quantum loop algebra}
\label{sub:pre-quantum}

Let $q$ be any non-zero complex number which is not a root of unity (more generally, the results of Sections \ref{sec:quantum} and \ref{sec:shuffle} hold for any non-zero non-root-of-unity $q$ in any field of characteristic zero). We henceforth fix a logarithm of $q$, so that $q^\lambda \in \BC$ be well-defined for any $\lambda \in \BC$. We abbreviate
\begin{equation}
\label{eqn:qi}
q_i = q^{\frac {d_{ii}}2}, \quad \forall i \in I
\end{equation}
Consider the following formal series for all $i \in I$
\begin{equation}
\label{eqn:formal series}
  e_i(x) = \sum_{d \in \BZ} \frac {e_{i,d}}{x^d}, \qquad
  f_i(x) = \sum_{d \in \BZ} \frac {f_{i,d}}{x^d}, \qquad
  \ph^\pm_i(x) = \sum_{d = 0}^\infty \frac {\ph^\pm_{i,d}}{x^{\pm d}}
\end{equation}
and let $\delta(x) = \sum_{d \in \BZ} x^d$ be the formal delta function. For any $i,j \in I$, set
\begin{equation}
\label{eqn:zeta}
  \zeta_{ij} (x) = \frac {x - q^{-d_{ij}}}{x - 1}
\end{equation}
We now recall the definition of the quantum loop algebra of type $\fg$, i.e. our generalization of Drinfeld's so-called new presentation of quantum affine algebras. This will be done in two steps: 

\medskip

\begin{enumerate}[leftmargin=*]

\item start from the ``pre-quantum loop algebra", i.e. the algebra considered in \cite[Definition 3.1]{H Old} with the Drinfeld-Serre relations removed, and

\medskip

\item then impose additional relations (which we only know explicitly in finite or simply-laced types) in Definition \ref{def:quantum loop}. 

\end{enumerate}

\medskip

\begin{definition}
\label{def:pre quantum loop}

The pre-quantum loop algebra associated to $\fg$ is
$$
  \tUU = \BC \Big \langle e_{i,d}, f_{i,d}, \ph_{i,d'}^\pm \Big \rangle_{i \in I, d \in \BZ, d' \geq 0} \Big/
  \text{relations \eqref{eqn:rel 0 loop}-\eqref{eqn:rel 3 loop}}
$$
where we impose the following relations for all $i,j \in I$ and all $\pm, \pm' \in \{+,-\}$:
\begin{equation}
\label{eqn:rel 0 loop}
  e_i(x) e_j(y) \zeta_{ji} \left( \frac yx \right) =\, e_j(y) e_i(x) \zeta_{ij} \left(\frac xy \right)
\end{equation}
\begin{equation}
\label{eqn:rel 1 loop}
  \ph_j^\pm(y) e_i(x) \zeta_{ij} \left(\frac xy \right) = e_i(x) \ph_j^\pm(y) \zeta_{ji} \left( \frac yx \right)
\end{equation}
\begin{equation}
\label{eqn:rel 2 loop}
  \ph_{i}^{\pm}(x) \ph_{j}^{\pm'}(y) = \ph_{j}^{\pm'}(y) \ph_{i}^{\pm}(x), \quad
  \ph_{i,0}^+ \ph_{i,0}^- = 1
\end{equation}
as well as the opposite relations with $e$'s replaced by $f$'s, and finally the relation
\begin{equation}
\label{eqn:rel 3 loop}
  \left[ e_i(x), f_j(y) \right] =
  \frac {\delta_{ij}\delta \left(\frac xy \right)}{q_i - q_i^{-1}}  \Big( \ph_i^+(x) - \ph_i^-(y) \Big)
\end{equation}
In relation \eqref{eqn:rel 0 loop} we clear denominators and obtain relations by equating the coefficients of all $x^{-d} y^{-d'}$ in the left and right-hand sides, while in \eqref{eqn:rel 1 loop} we expand in non-positive powers of $y^{\pm 1}$ and then equate coefficients. 

\end{definition}

\medskip

\noindent We will sometimes replace the generators $\{\ph_{j,d}^\pm\}_{j \in I, d \geq 0}$ by $\{\kappa_j,p_{j,u}\}_{j \in I, u \in \BZ \backslash 0}$ via
\begin{equation}
\label{eqn:k and p}
\ph^\pm_j(y) = \kappa_j^{\pm 1} \exp \left( \sum_{u=1}^{\infty} \frac {p_{j,\pm u}}{uy^{\pm u}} \right)
\end{equation}
in terms of which relation \eqref{eqn:rel 1 loop} is equivalent to
\begin{equation}
\label{eqn:rel 1 loop k}
\kappa_j e_{i}(x) = e_{i}(x) \kappa_j q^{d_{ij}}
\end{equation}
\begin{equation}
\label{eqn:rel 1 loop p}
[p_{j,u}, e_i(x)] = e_{i}(x)x^u(q^{ud_{ij}} - q^{-ud_{ij}}) 
\end{equation}
for all $i,j \in I$ and $u \neq 0$. The algebra $\tUU$ is graded by $\zz \times \BZ$, with
\begin{equation}
\label{eqn:grading loop}
\deg e_{i,d} = (\bs^i,d), \qquad \deg f_{i,d} = (-\bs^i,d), \qquad \deg \ph^\pm_{i,d'} = (0,\pm d')
\end{equation}
We will encounter the following subalgebras of $\tUU$
\begin{align*}
&\tUUp \text{ generated by } \{e_{i,d}\}_{i \in I, d \in \BZ} \\
&\tUUm \text{ generated by } \{f_{i,d}\}_{i \in I, d \in \BZ} \\
&\UUo \text{ generated by } \{\ph_{i,d'}^\pm\}_{i \in I, d' \geq 0} \\
&\tUUg \text{ generated by } \{e_{i,d},\ph_{i,d'}^+ \}_{i \in I, d \in \BZ, d' \geq 0} \\
&\tUUl \text{ generated by } \{f_{i,d}, \ph_{i,d'}^- \}_{i \in I, d \in \BZ, d' \geq 0} 
\end{align*}

\medskip

\subsection{The Hopf algebra structure} 
\label{sub:hopf pre-quantum}

It is well-known that $\tUU$ admits a topological Hopf algebra structure, with coproduct determined by
\begin{equation}
\label{eqn:coproduct ph}
\Delta(\ph^\pm_i(z)) = \ph^\pm_i(z) \otimes \ph^\pm_i(z)
\end{equation}
\begin{equation}
\label{eqn:coproduct e}
\Delta(e_i(z)) = \ph_i^+(z) \otimes e_i(z) + e_i(z) \otimes 1
\end{equation}
\begin{equation}
\label{eqn:coproduct f}
\Delta(f_i(z)) = 1 \otimes f_i(z) + f_i(z) \otimes \ph_i^-(z)
\end{equation}
and antipode $S$ given by
\begin{equation}
\label{eqn:antipode phi}
S\left(\ph_i^\pm(z) \right) = \left(\ph^\pm_i(z)\right)^{-1}
\end{equation}
\begin{equation}
\label{eqn:antipode e}
S\left(e_i(z) \right) = -\left(\ph^+_i(z)\right)^{-1} e_i(z)
\end{equation}
\begin{equation}
\label{eqn:antipode f}
S\left(f_i(z) \right) = - f_i(z) \left(\ph^-_i(z)\right)^{-1}
\end{equation}
There is also a Hopf pairing
\begin{equation}
\label{eqn:pairing pre-quantum}
\tUUg \otimes \tUUl \xrightarrow{\langle \cdot, \cdot \rangle} \BC
\end{equation}
generated by the assignments
\begin{align}
&\Big \langle e_i(x), f_j(y) \Big \rangle = \frac {\delta_{ij}\delta \left(\frac xy \right)}{q_i^{-1} - q_i}  \label{eqn:pairing ef}\\
&\Big \langle \ph^+_i(x), \ph^-_j(y) \Big \rangle =  \frac {\zeta_{ij} \left(\frac xy \right)}{\displaystyle \zeta_{ji} \left(\frac yx \right)} \label{eqn:pairing ph}
\end{align}
(the right-hand side of expression \eqref{eqn:pairing ph} is expanded as $|x| \gg |y|$) and all other pairings between generators being 0. Recall that a Hopf pairing must satisfy
\begin{align}
&\Big \langle a,b_1b_2 \Big \rangle = \Big \langle \Delta(a), b_1 \otimes b_2 \Big \rangle \label{eqn:bialgebra 1} \\
&\Big \langle a_1 a_2 ,b \Big \rangle = \Big \langle a_1 \otimes a_2, \Delta^{\text{op}}(b) \Big \rangle \label{eqn:bialgebra 2} 
\end{align}
($\Delta^{\text{op}}$ is the opposite coproduct) and that moreover it is preserved by the antipode
\begin{equation}
\label{eqn:antipode pairing}
\Big \langle S(a), S(b) \Big \rangle = \Big \langle a,b \Big \rangle
\end{equation}
for all $a,a_1,a_2 \in \tUUg$ and $b,b_1,b_2 \in \tUUl$. With this in mind, we have
\begin{equation}
\label{eqn:drinfeld double}
ba = \Big \langle a_1, S(b_1) \Big \rangle a_2b_2 \Big \langle a_3, b_3 \Big \rangle
\end{equation}
for all $a \in \tUUg$ and $b \in \tUUl$, where we use Sweedler notation 
$$
\Delta^{(2)}(a) = a_1 \otimes a_2 \otimes a_3, \quad \Delta^{(2)}(b) = b_1 \otimes b_2 \otimes b_3
$$
(recall that $\Delta^{(2)} = (\Delta \otimes \text{Id} ) \circ \Delta$) to avoid writing down the implied summation signs. Formula \eqref{eqn:drinfeld double} encodes the fact that $\tUU$ is the \textbf{Drinfeld double} 
\begin{equation}
\label{eqn:pre drinfeld double}
\tUU = \tUUg \otimes \tUUl
\end{equation}
with the Hopf algebra structure induced by formulas \eqref{eqn:coproduct ph}-\eqref{eqn:pairing ph} \footnote{We tacitly identify $\kappa_i \otimes 1$ with the inverse of $1 \otimes \kappa_i^{-1}$ in the tensor product \eqref{eqn:pre drinfeld double} and hereafter.}.

\medskip

\subsection{The quantum loop algebra}

The pairing \eqref{eqn:pairing pre-quantum} has a non-trivial kernel, i.e.
\begin{align}
&I^+ \subset \tUUp, \quad \text{defined by }a \in I^+ \ \Leftrightarrow \ \Big \langle a, \tUUm \Big \rangle = 0 \label{eqn:kernel 1} \\
&I^- \subset \tUUm, \quad \text{defined by }b \in I^- \ \Leftrightarrow \ \Big \langle \tUUp, b \Big \rangle = 0 \label{eqn:kernel 2}
\end{align}
As \eqref{eqn:pairing pre-quantum} is a Hopf pairing, $I^\pm$ are ideals, so we may consider the quotient algebras 
$$
\UUpm = \tUUpm  \Big / I^\pm
$$

\medskip

\begin{definition}
\label{def:quantum loop}

The quantum loop algebra associated to $\fg$ is
\begin{equation}
\label{eqn:quantum loop}
\UU = \UUp \otimes \UUo \otimes \UUm
\end{equation}
with the topological Hopf algebra structure induced from that of $\tUU$. Explicitly, $\UU$ is defined by generators $e_{i,d}, f_{i,d}, \ph^\pm_{i,d'}$ modulo relations \eqref{eqn:rel 0 loop}-\eqref{eqn:rel 3 loop} together with additional relations that correspond to the generators of the ideals $I^\pm$. 

\end{definition}

\medskip

\noindent The topological Hopf algebra structure and Hopf pairing defined in Subsection \ref{sub:hopf pre-quantum} descend to the subalgebras 
\begin{align*}
&\UUg \text{ generated by }\UUp \text{ and }\{\ph_{i,d}^+\}_{i \in I, d \geq 0} \\
&\UUl \text{ generated by }\UUm \text{ and }\{\ph_{i,d}^-\}_{i \in I, d \geq 0}
\end{align*}
In particular, we obtain a pairing
\begin{equation}
\label{eqn:pairing quantum}
\UUg \otimes \UUl \xrightarrow{\langle \cdot, \cdot \rangle} \BC
\end{equation}
with respect to which $\UU$ is the Drinfeld double of $\UUg$ and $\UUl$. 

\medskip

\noindent Explicit formulas for the generators of the ideal $I^+$ (the ones of $I^-$ are obtained by replacing all $e$'s by $f$'s and reordering all products) are known when 

\medskip

\begin{itemize}[leftmargin=*]

\item $\fg$ is of finite type. These are Drinfeld's $q$-Serre relations of \cite{Dr} for all $i \neq j$ in $I$ (see also \cite{Da1}):
\begin{equation}
\label{eqn:serre}
  \sum_{\sigma \in S_n} \sum_{k=0}^{n} (-1)^k {n \choose k}_i e_i(x_{\sigma(1)})\dots e_i(x_{\sigma(k)}) e_j(y) e_i(x_{\sigma(k+1)}) \dots e_i(x_{\sigma(n)}) = 0
\end{equation}
where $n = 1 - \frac {2d_{ij}}{d_{ii}}$, ${n \choose k}_i = \frac {[n]_i!}{[k]_i![n-k]_i!}$ with $[k]_i! = [1]_i \dots [k]_i$ and $[k]_i = \frac {q_i^k - q_i^{-k}}{q_i - q_i^{-1}}$.

\medskip

\item $\fg$ is simply laced. Explicit formulas for the generators of $I^+$ were worked out in \cite[Subsection 1.3]{N Loop}.

\end{itemize}

\medskip

\noindent The problem of calculating explicit generators of the ideals $I^\pm$ is still open (and very interesting) for a non-simply laced Kac-Moody Lie algebra $\fg$, see \cite{N Arbitrary} for an overview. Once one has such a collection of generators, then one would obtain a complete and explicit generators-and-relations presentation of $\UU$. 

\medskip

\subsection{The quantum affine algebra} 
\label{sub:quantum affine}

When $\fg$ is of finite type, an important indication that \eqref{eqn:serre} are the ``correct" relations comes from the fact that $\UU$ thus defined provides an incarnation of the Drinfeld-Jimbo quantum group of type $\wfg$. 

\medskip

\begin{definition}
\label{def:quantum affine}

Set $\wI = I \sqcup \{0\}$ and extend the finite type Cartan matrix of $\fg$ by adding one more row and column corresponding to the affine root $0 \in \wI$. Define 
$$
\UUaff = \BC \Big \langle e_i, f_i, \kappa^{\pm 1}_i \Big \rangle_{i \in \wI} \Big/
\left(  \text{relations \eqref{eqn:rel 0 affine}-\eqref{eqn:rel 3 affine}, and }  \prod_{i \in \wI} \kappa_i^{\eta_i} = 1  \right)
$$
where $\{\eta_i\}_{i\in \wI}$ are the dual Kac labels, and we impose the relations
\begin{equation}
\label{eqn:rel 0 affine}
\sum_{k=0}^{n} (-1)^k {n \choose k}_{i}  e_i^k e_j e_i^{n-k} = \sum_{k=0}^{n} (-1)^k {n \choose k}_{i}  f_i^k f_j f_i^{n-k} = 0
\end{equation}
for all $i\neq j$ in $\wI$ (with the notation as in \eqref{eqn:serre}), as well as for all $i,j \in \wI$
\begin{equation}
\label{eqn:rel 1 affine}
\kappa_j e_i = e_i \kappa_j q^{d_{ij}}, \quad f_i \kappa_j  = \kappa_j f_i q^{d_{ij}}, \quad \kappa_i \kappa_j = \kappa_j \kappa_i
\end{equation}
\begin{equation}
\label{eqn:rel 3 affine}
[e_i,f_j] = \frac {\delta_{ij}}{q_i - q_i^{-1}}  \left( \kappa_i - \kappa_i^{-1} \right)
\end{equation}

\end{definition}

\medskip

\noindent The connection between quantum affine and quantum loop algebras is given by
\begin{equation}
\label{eqn:iso drinfeld}
\Phi : \UUaff \xrightarrow{\sim} \UU
\end{equation}
for any finite type Lie algebra $\fg$, which was claimed by Drinfeld (\cite{Dr}) and proved to be an isomorphism in complete detail by Beck (\cite{B}) and Damiani (\cite{Da2, Da3}). Explicitly, $\Phi$ sends
$$
\Phi(e_i) = e_{i,0}, \quad \Phi(f_i) = f_{i,0}, \quad \Phi(\kappa_i) = \kappa_i,
$$
for all $i \in I$, while
\begin{align}
&\Phi(e_0) = [f_{i_1,0},[f_{i_2,0},[ \dots, [f_{i_k,0}, f_{j,1}]_q\dots]_q]_q]_q \label{eqn:iterated commutator e} \\ 
&\Phi(f_0) = [e_{i_1,0},[e_{i_2,0},[ \dots, [e_{i_k,0}, e_{j,-1}]_q\dots]_q]_q]_q \label{eqn:iterated commutator f}
\end{align}
for certain $i_1,\dots,i_k,j \in I$. In the formulas above, the $q$-commutator is defined by 
\begin{equation}
\label{eqn:q commutator}
[a,b]_q = ab - q^{(\hdeg a, \hdeg b)} ba 
\end{equation}
where hdeg denotes the $\zz$ component of the grading \eqref{eqn:grading loop} and $(\cdot,\cdot)$ denotes the symmetric pairing \eqref{eqn:symmetric pairing}. The isomorphism \eqref{eqn:iso drinfeld} was the initial motivation for the introduction of quantum loop algebras, as well as a primary catalyst for the study of the representations that we will now recall.

\medskip

\subsection{Representations} 
\label{sub:representations}

For $\fg$ of finite type, the representation theory of $\UUaff$ has been studied for decades. A classification of its finite-dimensional simple modules in terms of the joint eigenvalues of the commuting operators $\ph_{i,d}^\pm$ (we tacitly invoke the isomorphism \eqref{eqn:iso drinfeld}) was given in \cite{CP}. Let us consider the subalgebra
\begin{equation}
\label{eqn:borel}
\UUaffg \subset \UUaff
\end{equation}
generated by $\{e_i, \kappa_i\}_{i \in \wI}$. The algebras $\UUaff$ and $\UUaffg$ have the same category of so-called type 1 finite-dimensional representations (up to tensoring by one-dimensional representations, see \cite[Remark 3.11]{FH}). With this in mind, Hernandez-Jimbo initiated the study of the following larger category.

\medskip

\begin{definition} 
\label{def:category o affine}

(\cite{HJ}) Consider the category $\CO$ of complex representations 
\begin{equation}
\label{eqn:rep affine}
\UUaffg \curvearrowright V
\end{equation}
which admit a decomposition
\begin{equation}
\label{eqn:weight decomposition}
V = \bigoplus_{\bom \in \cup_{s=1}^t (\bom^s - \nn) } V_{\bom}
\end{equation}
for finitely many $\bom^1,\dots,\bom^t \in \cc$, such that every weight space
\begin{equation}
\label{eqn:weight}
V_{\bom} = \Big\{v \in V \text{ s.t. } \kappa_i \cdot v = q^{(\bom, \bs^i)}v, \ \forall i\in I \Big\}
\end{equation}
is finite-dimensional.

\end{definition}

\medskip

\noindent If $v \in V_\bom$ as above, then we call $\bom$ the weight of $v$. We can talk about maximal weights with respect to the partial order $\bom \geq \bom'$ if $\bom-\bom' \in \nn$. This notion is relevant because if $v$ is a maximal weight vector in a representation \eqref{eqn:rep affine}, then 
$$
e_{i,d}\cdot v = 0
$$
for all $i \in I$, $d \geq 0$. Moreover, if the maximal weight space is spanned by $v$, then 
$$
\ph_{i,d}^+ \cdot v = \psi_{i,d}v
$$
for all $i \in I$, $d \geq 0$, where $\{\psi_{i,d}\}_{i \in I, d \geq 0}$ are certain complex numbers.

\medskip

\begin{definition} 
\label{def:l weights}

An \textbf{$\ell$-weight} is an $I$-tuple of invertible power series
$$
\bpsi = (\psi_i(z) )_{i \in I} \in \left( \BC[[z^{-1}]]^* \right)^I
$$
If every $\psi_i(z)$ is the expansion of a rational function, then $\bpsi$ is called \textbf{rational}.

\end{definition}

\medskip

\begin{theorem} 
\label{thm:simple}

(\cite{HJ}) Up to isomorphism, there is a unique simple representation 
$$
\UUaffg \curvearrowright L(\bpsi)
$$
generated by a single vector $\vac$ that satisfies the identities
$$
e_{i,d}\cdot \vac= 0
$$
$$
\ph^+_{i}(z)\cdot \vac = \psi_i(z)\vac
$$
for all $i \in I$, $d \geq 0$. This representation is in category $\CO$ if and only if $\bpsi$ is rational. 

\end{theorem}

\medskip

\subsection{$q$-characters} 
\label{sub:q-characters affine}

For any representation $\UUaffg \curvearrowright V$ in category $\CO$, the weight spaces (i.e. the joint eigenspaces of the commutative subalgebra $\{\kappa_i\}_{i \in I}$) are finite-dimensional by definition. Therefore, the bigger commutative subalgebra $\{\ph_{i,d}^+\}_{i \in I, d \geq 0}$ also has well-defined joint generalized eigenspaces, namely
\begin{equation}
\label{eqn:ell weight decomposition}
V = \bigoplus_{\bpsi \in \left( \BC[[z^{-1}]]^* \right)^I} V_{\bpsi}
\end{equation}
\begin{equation}
\label{eqn:ell weight}
V_{\bpsi} = \Big\{v \in V \text{ s.t. } (\ph_{i,d}^+ - \psi_{i,d} \cdot \text{Id}_V)^N \cdot v = 0 \text{ for }N \gg 1, \forall i\in I, d \geq 0\Big\}
\end{equation}
In particular, if $\psi_{i,0} = q^{(\bom,\bs^i)}, \forall i \in I$ for a weight $\bom \in \cc$, then $V_{\bpsi} \subseteq V_{\bom}$. In this case, we will call $\bom$ the leading weight of $\bpsi$, and denote this by $\bom = \lead(\bpsi)$.

\medskip

\begin{definition}
\label{def:q-character}

(\cite{FR}) The $q$-character of a representation $V$ in category $\CO$ is
\begin{equation}
\label{eqn:q-character definition}
\chi_q(V) = \sum_{\bpsi \in  \left( \BC[[z^{-1}]]^* \right)^I} \dim_{\BC} (V_{\bpsi}) [\bpsi]
\end{equation}
By \cite[Lemma 3.9]{HJ}, the sum above actually only goes over rational $\ell$-weights.

\end{definition}

\medskip

\noindent It is clear that $\chi_q(V)$ descends by linearity to the Grothendieck group of category $\CO$, which we will denote by $\text{Rep}(\CO)$. It is shown in \cite{HJ} that the map
$$
\chi_q : \text{Rep}(\CO) \rightarrow \prod_{\bpsi \in  \left( \BC[[z^{-1}]]^* \right)^I} \BZ [ \bpsi ]
$$
is injective. Moreover, $\chi_q$ is a ring homomorphism, with respect to the usual coproduct of quantum affine algebras in the LHS and the component-wise multiplication in the RHS \footnote{To be rigorous, in order for the multiplication \eqref{eqn:product} to be well-defined, one needs to replace $\prod_{\bpsi} \BZ  [\bpsi]$ by its subgroup $G$ of countable sums of $\bpsi$'s, whose leading weights lie in a finite union of the form $\cup_{s=1}^t (\bom^s - \nn)$, and such that only finitely many $\bpsi$'s have leading weight $\bom$ for any $\bom \in \cc$ (see \cite{HJ} for details). The $q$-character of any representation in category $\CO$ lies in $G$.}
\begin{equation}
\label{eqn:product}
\left[ \Big(\psi_i(z) \Big)_{i\in I} \right] \cdot \left[ \Big(\psi_i'(z) \Big)_{i\in I} \right] = \left[ \Big(\psi_i(z)\psi_i'(z) \Big)_{i\in I} \right]
\end{equation}
In the present paper, we will generalize the notions above to arbitrary Kac-Moody Lie algebras $\fg$, by finding a suitable generalization of the Borel subalgebra \eqref{eqn:borel}. We recall that the initial stumbling block to this is the fact that $\UUaff$ is not defined for general $\fg$, so we must instead define a subalgebra
$$
\CA^{\geq} \subset \UU
$$
which will play the role of \eqref{eqn:borel}. The key to constructing $\CA^{\geq}$, as well as providing the technical tools to prove analogues of the results in Subsections \ref{sub:representations} and \ref{sub:q-characters affine}, is to use the shuffle algebra incarnation of $\UU$. We will now review shuffle algebras.

\bigskip

\section{Shuffle algebras}
\label{sec:shuffle}

\medskip

\subsection{The big shuffle algebra}
\label{sub:big shuffle}

We now review the trigonometric degeneration (\cite{E1, E2}) of the Feigin-Odesskii shuffle algebra (\cite{FO}) associated to an arbitrary Kac-Moody Lie algebra $\fg$. Consider the following vector space of rational functions in arbitrarily many variables
\begin{equation}
\label{eqn:big shuffle}
\CV = \bigoplus_{\bn \in \nn} \CV_{\bn}, \quad \text{where} \quad \CV_{\bn} = \frac {\BC[z_{i1}^{\pm 1},\dots,z_{in_i}^{\pm 1}]^{\text{sym}}_{i \in I}}{\prod^{\text{unordered}}_{\{i \neq j\} \subset I} \prod_{1\leq a \leq n_i, 1\leq b \leq n_j} (z_{ia} - z_{jb})}
\end{equation}
Above, ``sym" refers to \textbf{color-symmetric} Laurent polynomials, meaning that they are symmetric in the variables $z_{i1},\dots,z_{in_i}$ for each $i \in I$ separately (the terminology is inspired by the fact that $i \in I$ is called the color of the variable $z_{ia}$). We make the vector space $\CV$ into a $\BC$-algebra via the following shuffle product:
\begin{equation}
\label{eqn:mult}
E( z_{i1}, \dots, z_{i n_i}) * E'(z_{i1}, \dots,z_{i n'_i}) = \frac 1{\bn!  \bn'!}\,\cdot
\end{equation}
$$
  \textrm{Sym} \left[ E(z_{i1}, \dots, z_{in_i}) E'(z_{i,n_i+1}, \dots, z_{i,n_i+n'_i})
               \prod_{i,j \in I}\prod_{1 \leq a \leq n_i, n_j < b \leq n_j+n_j'} \zeta_{ij} \left( \frac {z_{ia}}{z_{jb}} \right) \right]
$$
with $\zeta_{ij}$ as in \eqref{eqn:zeta}. In \eqref{eqn:mult}, $\sym$ denotes symmetrization with respect to the
\begin{equation*}
  (\bn+\bn')! := \prod_{i\in I} (n_i+n'_i)!
\end{equation*}
permutations of the variables $\{z_{i1}, \dots, z_{i,n_i+n'_i}\}$ for each $i$ independently. The shuffle product \eqref{eqn:mult} is easily seen to be associative and well-defined (the latter claim is not completely trivial, as it involves showing that the apparent poles at $z_{ia}=z_{ib}$ produced by the denominators of $\zeta_{ii}$ are eliminated by the symmetrization). 

\medskip

\subsection{The (small) shuffle algebra}
\label{sub:small shuffle}

The main motivation for the multiplication \eqref{eqn:mult} is to ensure the existence of an algebra homomorphism 
\begin{equation}
\label{eqn:upsilon tilde}
\widetilde{\Upsilon}^+ : \tUUp \rightarrow \CV, \qquad e_{i,d} \mapsto z_{i1}^d \in \CV_{\bs^i}, \quad \forall i \in I, d \in \BZ
\end{equation}
The following result is a particular case of \cite[Theorem 1.2]{N Arbitrary}, where the zeta function of \loccit is taken to be our $\zeta_{ij}(x)(1-x)^{\delta_{i<j}}$ with respect to any total order on $I$.

\medskip

\begin{theorem}
\label{thm:quantum to shuffle}

We have $\emph{Ker }\widetilde{\Upsilon}^+ = I^+$ of \eqref{eqn:kernel 1}, so $\widetilde{\Upsilon}^+$ induces an isomorphism
\begin{equation}
\label{eqn:upsilon plus}
\Upsilon^+ : \UUp \xrightarrow{\sim} \CS^+
\end{equation}
where the \textbf{(small) shuffle algebra} is defined as $\CS^+ = \emph{Im }\widetilde{\Upsilon}^+ \subseteq \CV$.

\end{theorem}

\medskip 

\noindent In other words, $\CS^+$ is the subalgebra generated by $z_{i1}^d$. For $\fg$ of finite type, we have
\begin{equation}
\label{eqn:e}
\CS^+ = \left\{  \frac {\rho(z_{i1},\dots,z_{in_i})}
  {\prod^{\text{unordered}}_{\{i \neq j\} \subset I} \prod_{1\leq a \leq n_i, 1\leq b \leq n_{j}} (z_{ia} - z_{jb})} \right\}
\end{equation}
where $\rho$ goes over the set of color-symmetric Laurent polynomials that satisfy the Feigin-Odesskii \textbf{wheel conditions}:
\begin{equation}
\label{eqn:wheel}
\rho(\dots, z_{ia}, \dots,z_{jb},\dots)\Big|_{(z_{i1},z_{i2}, \dots, z_{in}) \mapsto (w, w q^{d_{ii}},  \dots, w q^{(n-1)d_{ii}}),\, z_{j1} \mapsto w q^{-d_{ij}}} =  0
\end{equation}
for any $i \neq j$ in $I$, where $n = 1 - \frac {2d_{ij}}{d_{ii}}$ (the inclusion $\subseteq$ in \eqref{eqn:e} was established by \cite{E1, E2} following \cite{FO}, and the inclusion $\supseteq$ was proved in \cite{NT}). 

\medskip

\begin{remark}

For $\fg$ of simply laced Kac-Moody type, a complete description of $\CS^+$ was given in \cite{N Loop}. We do not know a complete and explicit description of $\CS^+$ for $\fg$ of arbitrary Kac-Moody type, and this is a very interesting open problem. This problem is dual to finding an explicit set of generators for the ideal \eqref{eqn:kernel 1}, see \cite{N Arbitrary}.

\end{remark}

\medskip

\subsection{The double shuffle algebra}
\label{sub:double shuffle}

We define $\CS^-$ analogously to $\CS^+$, but with respect to the opposite algebra structure on $\CV$. Therefore, we have an isomorphism 
\begin{equation}
\label{eqn:upsilon minus}
\Upsilon^- : \UUm \xrightarrow{\sim} \CS^-, \qquad f_{i,d} \mapsto z_{i1}^d \in \CV_{\bs^i}, \quad \forall i \in I, d \in \BZ
\end{equation}
Elements of either $\CS^+$ or $\CS^-$ will be referred to as \textbf{shuffle elements}. The isomorphisms $\Upsilon^+$ and $\Upsilon^-$ glue to produce an algebra isomorphism
\begin{equation}
\label{eqn:double shuffle}
\Upsilon : \UU \xrightarrow{\sim} \CS = \CS^+ \otimes \frac {\BC [\ph_{i,d}^\pm]_{i \in I, d \geq 0}}{\ph_{i,0}^+ \ph_{i,0}^- - 1} \otimes \CS^-
\end{equation}
where the \textbf{double shuffle algebra} $\CS$ is required to satisfy relations \eqref{eqn:shuffle plus commute}, \eqref{eqn:shuffle minus commute} and \eqref{eqn:shuffle plus minus commute}. Explicitly, these relations read for all $E \in \CS^+$, $F \in \CS^-$, $j \in I$
\begin{align}
&\ph_j^\pm(y) E(z_{i1},\dots,z_{in_i}) = E(z_{i1},\dots,z_{in_i}) \ph_j^\pm(y) \prod_{i \in I} \prod_{a=1}^{n_i}\frac {\zeta_{ji} \left(\frac y{z_{ia}} \right)}{\zeta_{ij} \left(\frac {z_{ia}}y \right)} \label{eqn:shuffle plus commute} \\
&F(z_{i1},\dots,z_{in_i}) \ph_j^\pm(y)  = \ph_j^\pm(y) F(z_{i1},\dots,z_{in_i}) \prod_{i \in I} \prod_{a=1}^{n_i}\frac {\zeta_{ji} \left(\frac y{z_{ia}} \right)}{\zeta_{ij} \left(\frac {z_{ia}}y \right)} \label{eqn:shuffle minus commute}
\end{align}
The right-hand sides of the expressions above are expanded in non-positive powers of $y^{\pm 1}$. In terms of the elements $\kappa_j, p_{j,u}$ of \eqref{eqn:k and p}, it is straightforward to show that the relations above are equivalent to
\begin{equation}
\label{eqn:k shuffle}
\kappa_j X = X \kappa_j q^{(\pm \bn, \bs^j)}
\end{equation}
\begin{equation}
\label{eqn:p shuffle}
\left[p_{j,u}, X \right] = \pm X \sum_{i \in I} \left(z_{i1}^u+\dots+z_{in_i}^u\right)(q^{ud_{ij}} - q^{-ud_{ij}})
\end{equation}
for any $X(z_{i1},\dots,z_{in_i}) \in \CS^\pm$, with $\bn = (n_i)_{i \in I}$. Finally, we need to prescribe how to commute elements of $\CS^+$ and $\CS^-$. Since $\CS^\pm$ are generated by $\{z_{i1}^d\}_{i\in I}^{d \in \BZ}$, one uses 
\begin{equation}
\label{eqn:shuffle plus minus commute}
\left[ (z_{i1}^d)^+, (z_{j1}^{d'})^- \right] = \frac {\delta_{ij}}{q_i - q_i^{-1}} \cdot \begin{cases} \ph_{i,d+d'}^+ &\text{if }d+d' > 0 \\ \ph_{i,0}^+ - \ph_{i,0}^- &\text{if }d+d' = 0 \\ - \ph_{i,-d-d'}^- &\text{if }d+d' < 0 \end{cases}
\end{equation}
where $(z_{i1}^d)^\pm$ refers to $z_{i1}^d \in \CV_{\bs^i}$ interpreted as an element of the algebra $\CS^\pm$. To show that \eqref{eqn:double shuffle} is an algebra homomorphism, one needs to compare formulas \eqref{eqn:shuffle plus commute} and \eqref{eqn:shuffle plus minus commute} with \eqref{eqn:rel 1 loop} and \eqref{eqn:rel 3 loop}, respectively. This is a straightforward exercise.

\medskip

\subsection{The grading}
\label{sub:grading}

The isomorphism \eqref{eqn:double shuffle} matches the grading \eqref{eqn:grading loop} with the following $\zz \times \BZ$ grading on the double shuffle algebra $\CS$:
\begin{equation}
\label{eqn:deg}
\deg X = (\pm \bn, d)
\end{equation}
for any $X(z_{i1},\dots,z_{in_i}) \in \CS^\pm$ of total homogeneous degree $d$, where $\bn = (n_i)_{i \in I}$. We will call $\pm \bn$ and $d$ the \textbf{horizontal} and \textbf{vertical} degrees of $X$, and denote them
\begin{equation}
\label{eqn:hdeg vdeg}
\hdeg X = \pm \bn \quad \text{and} \quad \vdeg X = d
\end{equation}
We will write
\begin{equation}
\label{eqn:graded pieces}
\CS^\pm = \bigoplus_{\bn \in \nn} \CS_{\pm \bn} = \bigoplus_{\bn \in \nn} \bigoplus_{d \in \BZ} \CS_{\pm \bn, d}
\end{equation}
for the graded pieces of the shuffle algebras $\CS^\pm$. With this in mind, \eqref{eqn:k shuffle} implies
\begin{equation}
\label{eqn:k shuffle bis}
\kappa_\bm X = X \kappa_\bm q^{(\hdeg X, \bm)}
\end{equation}
for all $X \in \CS$ and $\bm = (m_i)_{i \in I} \in \zz$, where we denote $\kappa_\bm = \prod_{i\in I} \kappa_i^{m_i}$.

\medskip

\subsection{The coproduct}
\label{sub:coproduct shuffle}

The contents of the present Subsection closely follow the analogous statements in \cite[Section 3]{N Toroidal} and \cite[Section 3]{N Wheel}, so we leave their proofs as exercises to the reader. There exist topological coproducts on
\begin{align}
\CS^{\geq} = \CS^+ \otimes \BC [\ph_{i,d}^+]_{i \in I, d \geq 0} \label{eqn:s geq} \\
\CS^{\leq} = \BC [\ph_{i,d}^-]_{i \in I, d \geq 0} \otimes \CS^- \label{eqn:s leq}
\end{align}
(we tacitly assume that $\ph_{i,0}^\pm$ are replaced by $\kappa_i^{\pm 1}$ in the formulas above, so they are assumed to be invertible) defined by formula \eqref{eqn:coproduct ph} and
\begin{align}
\Delta(E) = \sum_{\b0 \leq \bm \leq \bn} \frac {\prod^{j \in I}_{m_j < b \leq n_j} \ph^+_j(z_{jb}) E(z_{i1},\dots , z_{im_i} \otimes z_{i,m_i+1}, \dots, z_{in_i})}{\prod^{i \in I}_{1\leq a \leq m_i} \prod^{j \in I}_{m_j < b \leq n_j} \zeta_{ji} \left( \frac {z_{jb}}{z_{ia}} \right)} \label{eqn:coproduct shuffle plus} \\
\Delta(F) = \sum_{\b0 \leq \bm \leq \bn} \frac {F(z_{i1},\dots , z_{im_i} \otimes z_{i,m_i+1}, \dots, z_{in_i}) \prod^{j \in I}_{1 \leq b \leq m_j} \ph^-_j(z_{jb})}{\prod^{i \in I}_{1\leq a \leq m_i} \prod^{j \in I}_{m_j < b \leq n_j} \zeta_{ij} \left( \frac {z_{ia}}{z_{jb}} \right)} \label{eqn:coproduct shuffle minus} 
\end{align}
for all $E \in \CS_{\bn}$, $F \in \CS_{-\bn}$. To make sense of the right-hand side of formulas \eqref{eqn:coproduct shuffle plus} and \eqref{eqn:coproduct shuffle minus}, we expand the denominator as a power series in the range $|z_{ia}| \ll |z_{jb}|$, and place all the powers of $z_{ia}$ to the left of the $\otimes$ sign and all the powers of $z_{jb}$ to the right of the $\otimes$ sign (for all $i,j \in I$, $1 \leq a \leq m_i$, $m_j < b \leq n_j$). Thus,
\begin{align}
&\Delta(E) = \prod_{i \in I} \prod_{a=1}^{n_i} \ph^+_i(z_{ia}) \otimes E(z_{i1},\dots,z_{in_i}) + \dots \label{eqn:leading delta plus} \\
&\Delta(F) = F(z_{i1},\dots,z_{in_i}) \otimes \prod_{i \in I} \prod_{a=1}^{n_i} \ph^-_i(z_{ia}) + \dots \label{eqn:leading delta minus}
\end{align}
where the ellipsis denotes tensors in which the second (respectively first) factor has smaller (respectively larger) horizontal degree than $E$ (respectively $F$). While they seem complicated, formulas \eqref{eqn:coproduct shuffle plus}-\eqref{eqn:coproduct shuffle minus} are forced upon us by the multiplicativity of the Drinfeld coproduct.

\medskip

\subsection{The pairing}
\label{sub:pairing shuffle}

 Throughout the present paper, the notation 
$$
\int_{|z_1| \ll \dots \ll |z_n|} R(z_1,\dots,z_n) \quad \text{is short for} \ \int_{|z_1| \ll \dots \ll |z_n|} R(z_1,\dots,z_n)  \prod_{a=1}^n \frac {dz_a}{2 \pi i z_a}
$$
and refers to the contour integral over concentric circles centered at the origin of the complex plane (the notation $|z_a| \ll |z_b|$ means that the these circles are very far away from each other when compared to any constants that might appear in the formula for $R(z_1,\dots,z_n)$). 

\begin{definition}
\label{def:i ordering}

We will call $i_1,\dots,i_n \in I$ an \textbf{ordering} of $\bn \in \nn$ if
$$
\bs^{i_1} + \dots + \bs^{i_n} = \bn
$$
If this is the case, we will employ for any $X(z_{i1},\dots,z_{in_i}) \in \CS_{\pm \bn}$ the notation
\begin{equation}
\label{eqn:notation ordering}
X(z_1,\dots,z_n)
\end{equation} 
to indicate the fact that each symbol $z_a$ is plugged into a variable of the form $z_{i_a \bullet}$ of $X$, where the choice of $\bullet$ does not matter due to the color-symmetry of $X$.

\end{definition}

\medskip

\noindent Following \cite{N Wheel, NT}, we may define a Hopf pairing
\begin{equation}
\label{eqn:pairing shuffle}
\CS^\geq \otimes \CS^\leq \xrightarrow{\langle \cdot, \cdot \rangle} \BC
\end{equation}
by formula \eqref{eqn:pairing ph} together with \footnote{We will abuse notation in our formulas for the pairing by writing $e_{i,d}, f_{i,d}$ instead of $z_{i1}^d$.}
\begin{align}
&\Big \langle E, f_{i_1,d_1} * \dots * f_{i_n,d_n} \Big \rangle = \text{ct} \int_{|z_1| \ll \dots \ll |z_n|} \frac {z_1^{d_1} \dots z_n^{d_n} E(z_1,\dots,z_n)}{\prod_{1\leq a < b \leq n} \zeta_{i_ai_b} \left(\frac {z_a}{z_b} \right)} 
\label{eqn:pairing shuffle plus} \\
&\Big \langle e_{i_1,d_1} * \dots * e_{i_n,d_n}, F \Big \rangle = \text{ct} \int_{|z_1| \gg \dots \gg |z_n|} \frac {z_1^{d_1} \dots z_n^{d_n} F(z_1,\dots,z_n)}{\prod_{1\leq a < b \leq n} \zeta_{i_bi_a} \left(\frac {z_b}{z_a} \right)} \label{eqn:pairing shuffle minus}
\end{align}
for all $E \in \CS_{\bn}$, $F \in \CS_{-\bn}$, any $d_1,\dots,d_n \in \BZ$ and any ordering $i_1,\dots,i_n$ of $\bn$ (the notation in the RHS of \eqref{eqn:pairing shuffle plus}-\eqref{eqn:pairing shuffle minus} is defined in accordance with \eqref{eqn:notation ordering}). For $E$ and $F$ whose horizontal degrees do not add up to 0, we set $\langle E, F \rangle = 0$. The ``ct" in relations \eqref{eqn:pairing shuffle plus}-\eqref{eqn:pairing shuffle minus} should be interpreted as follows: in order for the formulas above to match \eqref{eqn:pairing ef}, we would need
$$
\text{ct} = \prod_{a=1}^n \left(q_{i_a}^{-1} - q_{i_a}\right)^{-1}
$$
However, to keep our formulas simple, we henceforth assume \eqref{eqn:pairing shuffle plus}-\eqref{eqn:pairing shuffle minus} to have $\text{ct} = 1$. One cannot explicitly write down the antipode in terms of shuffle elements without an in-depth discussion of completions, and since this will not be necessary, we will instead be content with the following analogue of \eqref{eqn:pairing shuffle minus}.

\medskip

\begin{lemma}
\label{lem:antipode pairing shuffle}

We have for all $F \in \CS_{-\bn}$, any ordering $i_1,\dots,i_n$ of $\bn$ and any $d_1,\dots,d_n \in \BZ$, the following formula  \footnote{We also have the analogous formula \begin{equation}
\label{eqn:antipode pairing shuffle analogy}
\Big \langle S^{-1}(E), f_{i_1,d_1} * \dots * f_{i_n,d_n} \Big \rangle =(-1)^n  \int_{|z_1| \gg  \dots \gg |z_n|} \frac {z_1^{d_1} \dots z_n^{d_n} E(z_1,\dots,z_n)}{\prod_{1\leq a < b \leq n} \zeta_{i_ai_b} \left(\frac {z_a}{z_b} \right)}
\end{equation}
which we will not need in the present paper, and thus will not prove.}
\begin{equation}
\label{eqn:antipode pairing shuffle}
\Big \langle e_{i_1,d_1} * \dots * e_{i_n,d_n}, S(F) \Big\rangle = (-1)^n \int_{|z_1| \ll \dots \ll |z_n|} \frac {z_1^{d_1} \dots z_n^{d_n} F(z_1,\dots,z_n)}{\prod_{1\leq a < b \leq n} \zeta_{i_bi_a} \left(\frac {z_b}{z_a} \right)}
\end{equation}

\end{lemma}

\medskip

\noindent Before we prove Lemma \ref{lem:antipode pairing shuffle}, let us note that formulas \eqref{eqn:pairing shuffle minus} and \eqref{eqn:antipode pairing shuffle} are sufficient to realize $\CS$ as the Drinfeld double of $\CS^\geq$ and $\CS^\leq$, i.e. for all $E \in \CS^\geq$ and $F \in \CS^\leq$
\begin{equation}
\label{eqn:drinfeld double shuffle}
FE = \Big \langle E_1, S(F_1) \Big \rangle E_2 F_2 \Big \langle E_3, F_3 \Big \rangle
\end{equation}
as well as the analogous formula
\begin{equation}
\label{eqn:drinfeld double shuffle reversed}
EF = \Big \langle E_1, F_1 \Big \rangle F_2 E_2 \Big \langle E_3, S(F_3) \Big \rangle
\end{equation}

\medskip

\begin{proof} \emph{of Lemma \ref{lem:antipode pairing shuffle}:} The antipode satisfies $S^{\pm 1}(ab) = S^{\pm 1}(b)S^{\pm 1}(a)$, hence
\begin{equation}
\label{eqn:antipode pairing shuffle temp}
\Big \langle e_{i_1,d_1} * \dots * e_{i_n,d_n}, S(F) \Big \rangle  \stackrel{\eqref{eqn:antipode pairing}}= \Big \langle S^{-1}(e_{i_n,d_n}) * \dots * S^{-1}(e_{i_1,d_1}), F \Big \rangle
\end{equation}
If we apply $S^{-1}$ to \eqref{eqn:antipode e}, we observe that for all $i \in I$ and $d \in \BZ$
$$
S^{-1}(e_i(z)) = - e_i(z)  \left(\ph^+_i(z)\right)^{-1} \quad \Rightarrow \quad S^{-1}(e_{i,d}) = - \sum_{k=0}^{\infty} e_{i,d-k} \bar{\ph}_{i,k}^+
$$
where we write $\left(\ph^+_i(z)\right)^{-1} = \sum_{k=0}^{\infty} \frac {\bar{\ph}^+_{i,k}}{z^k}$. Using formula \eqref{eqn:rel 1 loop}, we have
\begin{equation}
\label{eqn:phi bar past e}
\left(\ph_j^+(y)\right)^{-1} e_i(x) = e_i(x) \left(\ph_j^+(y)\right)^{-1} \frac {\zeta_{ij} \left(\frac xy \right)}{\zeta_{ji} \left( \frac yx \right)}  \Rightarrow  \bar{\ph}_{j,k}^+ e_{i,d} = \sum_{m=0}^k \gamma_{ij}^{(m)} e_{i,d+m} \bar{\ph}_{j,k-m}^+
\end{equation}
where the complex numbers $\gamma_{ij}^{(m)}$ are defined by
\begin{equation}
\label{eqn:expansion zeta}
\frac {\zeta_{ij} \left(\frac xy \right)}{\zeta_{ji} \left( \frac yx \right)} = \sum_{m=0}^{\infty} \gamma_{ij}^{(m)} \frac {x^m}{y^m}
\end{equation}
We can use \eqref{eqn:phi bar past e} to move $\bar{\ph}$'s to the right in \eqref{eqn:antipode pairing shuffle temp}, so the LHS of \eqref{eqn:antipode pairing shuffle} equals
$$
(-1)^n \sum_{\{m_{a,b}\geq 0\}_{1 \leq a < b \leq n}}^{k_1,\dots,k_n \geq 0} \prod_{1 \leq a < b \leq n} \gamma_{i_ai_b}^{(m_{a,b})} \left\langle \prod_{a=n}^1 e_{i_{a},d_{a} - k_{a}+m_{a,a+1} + \dots + m_{a,n}} \prod_{a=1}^n \bar{\ph}^+_{i_a,k_a-m_{1,a}-\dots-m_{a-1,a}} , F \right\rangle
$$
Because of \eqref{eqn:bialgebra 2} and the fact that $\Delta^{\text{op}}(F) = F \otimes 1$ plus terms whose second tensor factor has $\text{hdeg} < \b0$ (and thus has pairing 0 with any product of $\bar{\ph}$'s), we see that only the terms for which $k_a = m_{1,a}+\dots+m_{a-1,a}$ for all $a \in \{1,\dots,n\}$ survive in the formula above. We therefore conclude that the LHS of \eqref{eqn:antipode pairing shuffle} equals
$$
(-1)^n \sum_{\{m_{a,b} \geq 0\}_{1 \leq a < b \leq n}} \prod_{1 \leq a < b \leq n} \gamma_{i_ai_b}^{(m_{a,b})} \left \langle\prod_{a=n}^1 e_{i_a, d_a - m_{1,a} - \dots - m_{a-1,a}+ m_{a,a+1}+ \dots+m_{a,n}}  , F \right \rangle
$$
By \eqref{eqn:pairing shuffle minus}, each pairing in the sum above equals (recall that the product of $e$'s is ordered from $i_n$ to $i_1$)
$$
(-1)^n \int_{|z_n| \gg \dots \gg |z_1|} \frac {\prod_{a=1}^n z_a^{d_a - m_{1,a} - \dots - m_{a-1,a}+ m_{a,a+1}+ \dots+m_{a,n}} F(z_1,\dots,z_n)}{\prod_{1\leq a < b \leq n} \zeta_{i_ai_b} \left(\frac {z_a}{z_b} \right)}
$$
If we recall the definition of $\gamma_{i_ai_b}^{(m_{a,b})}$ in \eqref{eqn:expansion zeta}, we obtain precisely the RHS of \eqref{eqn:antipode pairing shuffle}.

\end{proof}

\subsection{Slopes of shuffle elements}
\label{sub:slope}

Recall the horizontal and vertical degrees of an element $X \in \CS$, as defined in \eqref{eqn:hdeg vdeg}. We henceforth fix
\begin{equation}
\label{eqn:the r}
\br \in \mathbb{Q}_{>0}^I
\end{equation}
(we would not lose anything by assuming $\br = (1,\dots,1)$ throughout the whole paper, except for Subsection \ref{sub:positive} where we need any $\br \in \BZ_{>0}^I$). The \textbf{naive slope} of $X$ is
\begin{equation}
\label{eqn:naive slope}
\frac {\vdeg X}{\br \cdot \hdeg X}
\end{equation}
where the dot product is defined in \eqref{eqn:dot product}. This notion is motivated by the following picture, in which a shuffle element $X$ is thought to lie at the lattice point $(\br \cdot \hdeg X, \vdeg X)$. For example, the slanted line drawn corresponds to shuffle elements of naive slope $\frac 23$.

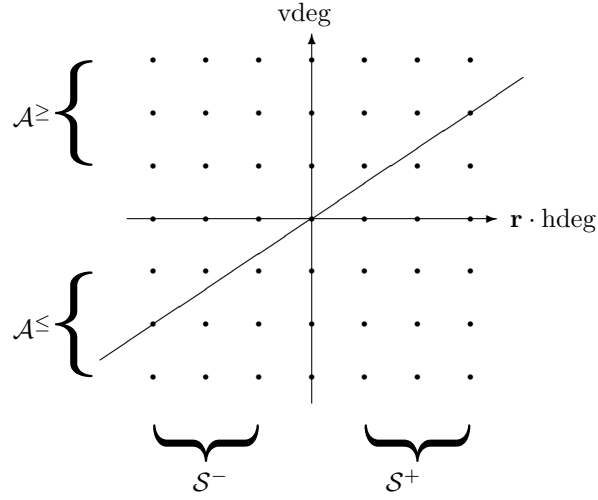
\begin{picture}(100,210)(-120,-55)
\label{figure}

\put(0,0){\circle*{2}}\put(20,0){\circle*{2}}\put(40,0){\circle*{2}}\put(60,0){\circle*{2}}\put(80,0){\circle*{2}}\put(100,0){\circle*{2}}\put(120,0){\circle*{2}}\put(0,20){\circle*{2}}\put(20,20){\circle*{2}}\put(40,20){\circle*{2}}\put(60,20){\circle*{2}}\put(80,20){\circle*{2}}\put(100,20){\circle*{2}}\put(120,20){\circle*{2}}\put(0,40){\circle*{2}}\put(20,40){\circle*{2}}\put(40,40){\circle*{2}}\put(60,40){\circle*{2}}\put(80,40){\circle*{2}}\put(100,40){\circle*{2}}\put(120,40){\circle*{2}}\put(0,60){\circle*{2}}\put(20,60){\circle*{2}}\put(40,60){\circle*{2}}\put(60,60){\circle*{2}}\put(80,60){\circle*{2}}\put(100,60){\circle*{2}}\put(120,60){\circle*{2}}\put(0,80){\circle*{2}}\put(20,80){\circle*{2}}\put(40,80){\circle*{2}}\put(60,80){\circle*{2}}\put(80,80){\circle*{2}}\put(100,80){\circle*{2}}\put(120,80){\circle*{2}}\put(0,100){\circle*{2}}\put(20,100){\circle*{2}}\put(40,100){\circle*{2}}\put(60,100){\circle*{2}}\put(80,100){\circle*{2}}\put(100,100){\circle*{2}}\put(120,100){\circle*{2}}\put(0,120){\circle*{2}}\put(20,120){\circle*{2}}\put(40,120){\circle*{2}}\put(60,120){\circle*{2}}\put(80,120){\circle*{2}}\put(100,120){\circle*{2}}\put(120,120){\circle*{2}}

\put(60,-10){\vector(0,1){140}}
\put(-10,60){\vector(1,0){140}}

\put(60,60){\line(3,2){80}}
\put(60,60){\line(-3,-2){80}}

\put(47,135){$\vdeg$}
\put(135,57){$\br \cdot \text{hdeg}$}

\put(-40,10){\scalebox{4}{$\{$}}
\put(-40,90){\scalebox{4}{$\{$}}
\put(0,-15){\scalebox{4}{\rotatebox{270}{$\}$}}}
\put(80,-15){\scalebox{4}{\rotatebox{270}{$\}$}}}

\put(-53,15){$\CA^\leq$}
\put(-53,95){$\CA^\geq$}

\put(15,-43){$\CS^-$}
\put(88,-43){$\CS^+$}

\end{picture}

\noindent The elements $\ph_{i,d}^\pm$ are thought to lie on the vertical axis (since they have horizontal degree equal to 0). The following notion is a more subtle version of naive slope. 

\medskip

\begin{definition}
\label{def:slope}

(following \cite{N Shuffle, N R-matrix}) For any $\mu \in \BQ$, we will say that

\medskip

\begin{itemize}[leftmargin=*]

\item $E \in \CS_\bn$ has \textbf{slope} $\geq \mu$ if the following limit is finite for all $\b0 < \bm \leq \bn$
\begin{equation}
\label{eqn:slope e geq}
  \lim_{\xi \rightarrow 0} \frac {E(\xi z_{i1},\dots,\xi z_{im_i},z_{i,m_{i+1}},\dots,z_{in_i})}{\xi^{\mu(\br \cdot \bm)}}
\end{equation}

\medskip

\item $E \in \CS_\bn$ has \textbf{slope} $\leq \mu$ if the following limit is finite for all $\b0 < \bm \leq \bn$
\begin{equation}
\label{eqn:slope e leq}
  \lim_{\xi \rightarrow \infty} \frac {E(\xi z_{i1},\dots,\xi z_{im_i},z_{i,m_{i+1}},\dots,z_{in_i})}{\xi^{\mu(\br \cdot \bm)}}
\end{equation}

\medskip

\item $F \in \CS_{-\bn}$ has \textbf{slope} $\leq \mu$ if the following limit is finite for all $\b0 < \bm \leq \bn$
\begin{equation}
\label{eqn:slope f leq}
  \lim_{\xi \rightarrow 0} \frac {F(\xi z_{i1},\dots,\xi z_{im_i},z_{i,m_{i+1}},\dots,z_{in_i})}{\xi^{\mu(-\br \cdot \bm)}}
\end{equation}

\medskip

\item $F \in \CS_{-\bn}$ has \textbf{slope} $\geq \mu$ if the following limit is finite for all $\b0 < \bm \leq \bn$
\begin{equation}
\label{eqn:slope f geq}
  \lim_{\xi \rightarrow \infty} \frac {F(\xi z_{i1},\dots,\xi z_{im_i},z_{i,m_{i+1}},\dots,z_{in_i})}{\xi^{\mu(-\br \cdot \bm)}}
\end{equation}

\end{itemize}

\end{definition}

\medskip

\noindent We will say that a shuffle element has slope $> \mu$ (respectively $<\mu$) if it has slope $\geq \mu+\varepsilon$ (respectively $\leq \mu-\varepsilon$) for some small enough $\varepsilon \in \BQ_{>0}$, or in other words if the respective limit in \eqref{eqn:slope e geq}-\eqref{eqn:slope f geq} is zero for all $\b0 < \bm \leq \bn$. We will write
\begin{align}
&\CS^\pm_{\geq \mu}=\bigoplus_{\bn\in \nn} \CS_{\geq \mu|\pm \bn}=\bigoplus_{\bn\in \nn} \bigoplus_{d\in \BZ} \CS_{\geq \mu|\pm \bn,d} \label{eqn:shuffle slope geq} \\
&\CS^\pm_{\leq \mu}=\bigoplus_{\bn\in \nn} \CS_{\leq \mu|\pm \bn}=\bigoplus_{\bn\in \nn} \bigoplus_{d\in \BZ} \CS_{\leq \mu|\pm \bn,d} \label{eqn:shuffle slope leq}
\end{align}
for the subsets of elements of $\CS^\pm$ of slopes $\geq \mu$ and $\leq \mu$, respectively, as well as for their graded pieces. We will similarly write $\CS^\pm_{> \mu}$ and $\CS^\pm_{< \mu}$ etc for the analogous notions. 

\medskip

\begin{example}
\label{ex:1}

When $\fg = \fsl_2$, we have $I = \{\bullet\}$ and $d_{\bullet\bullet}=2$. Because the set $I$ has a single element, we will write $n,r,\dots$ instead of $\bn,\br,\dots$ We have
$$
\CS_{\pm n} = \BC[z_1^{\pm 1}, \dots, z_n^{\pm 1}]^\emph{sym}
$$
Letting $r=1$, we have
\begin{align*}
&\CS_{\geq \mu|n} = \CS_{\leq -\mu|-n} = (z_1\dots,z_n)^{\lceil \mu \rceil} \BC[z_1,\dots,z_n]^{\emph{sym}} \\
&\CS_{\leq \mu|n} = \CS_{\geq -\mu|-n} = (z_1\dots,z_n)^{\lfloor \mu \rfloor} \BC[z^{-1}_1,\dots,z^{-1}_n]^{\emph{sym}} 
\end{align*}
In particular, $\CS_{\geq 0|n} = \BC[z_1,\dots,z_n]^{\emph{sym}}$ and $\CS_{<0|-n} = z_1\dots z_n \BC[z_1,\dots,z_n]^{\emph{sym}}$.

\end{example}

\medskip

\noindent Recall the topological coproduct \eqref{eqn:coproduct shuffle plus}-\eqref{eqn:coproduct shuffle minus}, which we will denote by $\woo$ to indicate the fact that $\Delta(E)$ and $\Delta(F)$ are infinite sums for general $E \in \CS^+$ and $F\in \CS^-$.  It is easy to see that
\begin{align}
&E \in \CS^+_{\geq \mu} \quad \Leftrightarrow \quad \Delta(E) \in  (\text{naive slope} \geq \mu) \woo (\text{anything}) \label{eqn:naive slope 1} \\
&E \in \CS^+_{\leq \mu} \quad \Leftrightarrow \quad \Delta(E) \in (\text{anything}) \woo (\text{naive slope} \leq \mu) \label{eqn:naive slope 3} \\
&F \in \CS^-_{\leq \mu} \quad \Leftrightarrow \quad \Delta(F) \in (\text{naive slope} \leq \mu) \woo (\text{anything}) \label{eqn:naive slope 2} \\
&F \in \CS^-_{\geq \mu} \quad \Leftrightarrow \quad \Delta(F) \in (\text{anything}) \woo (\text{naive slope} \geq \mu) \label{eqn:naive slope 4}
\end{align}
as well as the analogous notions with $\geq,\leq$ replaced by $>,<$. Since the shuffle product is additive in $(\br \cdot \text{hdeg}, \text{vdeg})$, the property of having naive slope $\geq \mu$ (resp. $\leq \mu$) is preserved by multiplication. Therefore, due to \eqref{eqn:naive slope 1}-\eqref{eqn:naive slope 4}, the property of having slope $\geq \mu$ (resp. $\leq \mu$) is also preserved by multiplication. In other words,  $\CS^\pm_{\geq \mu}$ and $\CS^\pm_{\leq \mu}$ are subalgebras of $\CS^\pm$ (as are $\CS^\pm_{> \mu}$ and $\CS^\pm_{< \mu}$, for the same reason). 

\medskip

\subsection{Slope subalgebras}
\label{sub:slope subalgebras}

Because of the $\bm = \bn$ case in \eqref{eqn:slope e geq}-\eqref{eqn:slope f geq}, we note that a shuffle element can simultaneously have slope $\geq \mu$ and $\leq \mu$ only if it has naive slope $\mu$. If this happens, we will say that the shuffle element in question has slope $\mu$. This property is preserved by the shuffle product, so we will write
\begin{equation}
\label{eqn:slope subalgebra}
\CB_{\mu}^\pm = \mathop{\bigoplus_{(\bn,d) \in \nn \times \BZ}}_{d = \pm \mu(\br \cdot \bn)} \CB_{\mu|\pm \bn}
\end{equation}
for the subalgebra of slope $\mu$ elements, which will be called a \textbf{slope subalgebra}. These algebras are the building blocks of shuffle algebras, in the sense that multiplication induces isomorphisms of vector spaces
\begin{equation}
\label{eqn:factorization 1}	
\bigotimes^{\rightarrow}_{\mu \in \BQ} \CB^\pm_{\mu} \xrightarrow{\sim} \CS^\pm \xleftarrow{\sim} \bigotimes^{\leftarrow}_{\mu \in \BQ} \CB^\pm_{\mu}  
\end{equation}
where $\rightarrow$ (respectively $\leftarrow$) means that we take the tensor product in increasing (respectively decreasing) order of $\mu$ \footnote{As vector spaces, $\otimes^{\rightarrow}_{\mu \in \BQ} \CB^\pm_{\mu}$ and $\otimes^{\leftarrow}_{\mu \in \BQ} \CB^\pm_{\mu}$ have bases $\otimes^{\rightarrow}_{\mu \in \BQ} b^\pm_{\mu, s_\mu}$ and $\otimes^{\leftarrow}_{\mu \in \BQ} b^\pm_{\mu, s_\mu}$ (respectively), where $b^\pm_{\mu,s_\mu}$ run over homogeneous bases of $\CB^\pm_{\mu}$ such that $b^\pm_{\mu,s_\mu} = 1$ for all but finitely many $\mu \in \BQ$.}. The first isomorphism in \eqref{eqn:factorization 1} is proved word for word as in \cite[Theorem 1.1]{N R-matrix}, and the second one is analogous; we thus leave the demonstration of the isomorphisms above as an exercise to the reader. To summarize, the isomorphisms in \eqref{eqn:factorization 1} state that any shuffle element can be uniquely written as a sum of products of shuffle elements in slope subalgebras, in either increasing or decreasing order of the slope. 

\medskip

\noindent There are natural analogues of \eqref{eqn:factorization 1} for all $\nu \in \BQ$, as follows
\begin{align}
&\text{ } \ \bigotimes^{\rightarrow }_{\mu \in [\nu, \infty)} \CB^\pm_{\mu} \xrightarrow{\sim} \CS^\pm_{\geq \nu} \xleftarrow{\sim} \bigotimes^{ \leftarrow}_{\mu \in [\nu, \infty)} \CB^\pm_{\mu} \label{eqn:factorization 3} \\
&\bigotimes^{\rightarrow}_{\mu \in (-\infty,\nu]} \CB^\pm_{\mu} \xrightarrow{\sim} \CS^\pm_{\leq \nu} \xleftarrow{\sim} \bigotimes^{\leftarrow}_{\mu \in (-\infty,\nu]} \CB_\mu^\pm \label{eqn:factorization 4}
\end{align}
as well as versions for $\CS^\pm_{> \nu}$ and $\CS^\pm_{< \nu}$, in which we replace the half-closed intervals by open intervals (see \cite[Proposition 3.14]{N R-matrix} for the proof of the first isomorphism in \eqref{eqn:factorization 4} in an closely related context; the rest are analogous). The following result is proved just like \cite[Proposition 3.12]{N R-matrix}, so we leave it as an exercise to the reader.

\medskip

\begin{proposition}
\label{prop:pair slopes}

The $\rightarrow$ factorization in \eqref{eqn:factorization 1} respects the pairing \eqref{eqn:pairing shuffle}, in the sense that for all collections of shuffle elements $\{E_\mu \in \CB^+_\mu, F_\mu \in \CB^-_\mu\}_{\mu \in \BQ}$ (almost all of which are 1) we have
\begin{equation}
\label{eqn:pair slopes}
\left \langle \prod_{\mu \in \BQ}^{\rightarrow} E_\mu, \prod_{\mu \in \BQ}^{\rightarrow} F_\mu \right \rangle = \prod_{\mu \in \BQ} \langle E_\mu, F_\mu \rangle 
\end{equation}

\end{proposition}

\medskip

\noindent Note that $\langle E,F\rangle = 0$ unless the shuffle elements $E$ and $F$ have opposite degrees in $\zz \times \BZ$. Therefore, the subalgebras $\CS^+_{\geq \nu}$ and $\CS^-_{<\nu}$ pair trivially, in the sense that
\begin{equation}
\label{eqn:orthogonal}
\Big \langle E, F \Big \rangle = \varepsilon(E) \varepsilon(F), \qquad \forall E \in \CS_{\geq \nu}^+, F \in \CS_{< \nu}^-
\end{equation}
where $\varepsilon$ denotes the counit. The analogous property holds for $\CS^+_{< \nu}$ and $\CS^-_{\geq \nu}$ etc. 

\medskip

\begin{proposition}
\label{prop:easy}

For any $\mu \in \BQ$ and $\bn \in \nn$, the vector space $\CB_{\mu|\pm \bn}$ is finite-dimensional. Moreover, if $\br \in \BZ_{>0}^I$, the assignment
\begin{equation}
\label{eqn:automorphism}
\sigma : \CS^\pm \rightarrow \CS^\pm, \qquad X(z_{i1},\dots,z_{in_i}) \mapsto X(z_{i1},\dots,z_{in_i}) \prod_{i \in I} \prod_{a=1}^{n_i} z_{ia}^{\pm r_i}
\end{equation}
is an algebra automorphism that takes $\CB_\mu^\pm$ to $\CB_{\mu + 1}^\pm$.

\end{proposition}

\medskip

\noindent The Proposition is an easy exercise: finite-dimensionality follows from the fact that having slope $\leq \mu$ and $\geq \mu$ places opposing bounds on the powers of $z_{ia}$ that can appear in the numerator of shuffle elements in $\CB_{\mu|\pm \bn}$. The fact that \eqref{eqn:automorphism} is an algebra automorphism which ``shifts" slopes by 1 is straightforward.

\medskip

\subsection{Half subalgebras}
\label{sub:half subalgebras}

The subalgebra $\CB_\mu^\pm$ consists of shuffle elements whose $(\br \cdot \text{hdeg}, \text{vdeg})$ lies on the ray of slope $\mu$ in the half-plane determined by $\pm$ the horizontal axis (see Figure \ref{figure}). It is customary to define 
\begin{equation}
\CB^\pm_{\infty} = \BC[\ph_{i,d}^\pm]_{i \in I, d \geq 0}
\end{equation}
which corresponds to the two vertical rays. By analogy with \eqref{eqn:s geq}-\eqref{eqn:s leq}, define
\begin{align}
&\CS^{\geq}_{\geq \mu} = \CS^{+}_{\geq \mu} \otimes \CB^+_{\infty} \\
&\CS^{\leq}_{\geq \mu} = \CB^-_{\infty} \otimes  \CS^{-}_{\geq \mu}
\end{align}
Because the coproduct \eqref{eqn:coproduct shuffle plus}-\eqref{eqn:coproduct shuffle minus} is coassociative (itself a straightforward exercise, which we leave to the reader), we may drop the word ``naive" from \eqref{eqn:naive slope 1}-\eqref{eqn:naive slope 4}; see \cite[Proposition 3.4]{N R-matrix} for the same statement in a very closely related context. The previous sentence translates in the following properties
\begin{align}
&\Delta(\CS^\geq_{\geq \mu}) \subset  \CS^\geq_{\geq \mu} \ \woo \ \CS^+ \label{eqn:slope 1} \\
&\Delta(\CS^+_{\leq \mu}) \subset \CS^\geq \ \woo \ \CS^+_{\leq \mu} \label{eqn:slope 3} \\
&\Delta(\CS^-_{\leq \mu}) \subset  \CS^-_{\leq \mu} \ \woo \ \CS^\leq \label{eqn:slope 2} \\
&\Delta(\CS^\leq_{\geq \mu}) \subset \CS^- \ \woo \ \CS^\leq_{\geq \mu} \label{eqn:slope 4}
\end{align}
The analogous formulas hold with the bottom indices $\geq$, $\leq$ replaced by $>$, $<$. When comparing the triangular decomposition $\CS = \CS^{\geq} \otimes \CS^{\leq}$ with the factorizations \eqref{eqn:factorization 1}, one obtains a factorization of $\CS$ as an infinite product of rays starting and ending with slope either $+\infty$ or $-\infty$. In the present Subsection, we will define a similar factorization starting and ending at any ray of slope $\nu \in \BQ$. To this end, let
\begin{align}
&\CA^{\geq}_\nu = \CS^-_{<\nu}  \otimes \CB_\infty^+ \otimes \CS^+_{\geq \nu}=  \bigotimes^{\leftarrow}_{\mu \in (-\infty,\nu)} \CB_\mu^- \otimes \bigotimes^{\leftarrow}_{\mu \in [\nu,\infty]} \CB_\mu^+ \label{eqn:a geq} \\
&\CA^{\leq}_\nu = \CS^+_{<\nu}  \otimes \CB_\infty^- \otimes \CS^-_{\geq \nu}= \bigotimes^{\leftarrow}_{\mu \in (-\infty,\nu)} \CB_\mu^+ \otimes \bigotimes^{\leftarrow}_{\mu \in [\nu,\infty]} \CB_\mu^-  \label{eqn:a leq}
\end{align}

\medskip

\begin{proposition}
\label{prop:a are algebras}

For any $\nu \in \BQ$, $\CA^{\geq}_\nu$ and $\CA^{\leq}_\nu$ are subalgebras of $\CS$.

\end{proposition}

\medskip

\begin{proof} We will only prove the statement about $\CA^{\geq}_\nu$, as the one about $\CA^{\leq}_\nu$ is analogous. As a vector space, $\CA^{\geq}_{\nu}$ is defined as the tensor product of three subalgebras:
\begin{equation}
\label{eqn:three}
\CS^-_{<\nu}  \quad \text{and} \quad \BC[\ph_{i,d}^+]_{i \in I, d\geq 0} \quad \text{and} \quad \CS^+_{\geq \nu}
\end{equation}
To show that this tensor product is itself a subalgebra of $\CS$ (i.e. closed under multiplication), we need to show that an arbitrary product of elements $x$ and $y$ from the above three subalgebras can be ``ordered", i.e. expressed as a sum of products of elements from the subalgebras \eqref{eqn:three}, in the given order. This is clear when one of $x$ and $y$ is from $\BC[\ph_{i,d}^+]$ and the other one is from $\CS^+_{\geq \nu}$ or $\CS^-_{<\nu}$, because of \eqref{eqn:k shuffle}, \eqref{eqn:p shuffle} and the following easy exercise (which we leave to the reader).

\medskip

\begin{claim}
\label{claim:preserved}

For any $\nu \in \BQ$ and $\bn \in \nn$, $\CS_{\geq \nu|\bn}$ and $\CS_{< \nu|-\bn}$ are modules over 
\begin{equation}
\label{eqn:color symmetric preserved}
\CP_{\bn} = \BC[z_{i1},\dots,z_{in_i}]^{\emph{sym}}_{i \in I}
\end{equation}
i.e. they are preserved by multiplication with any color-symmetric polynomial.

\end{claim}

\medskip

\noindent It remains to show that if $E \in \CS^+_{\geq \nu}$ and $F \in \CS^-_{<\nu}$, then $EF$ can be expressed as a sum of products of elements from the three subalgebras \eqref{eqn:three}, in the given order. However, to this end, we note that \eqref{eqn:slope 1} and \eqref{eqn:slope 2} imply that 
\begin{align}
&\Delta^{(2)} (E) \in \CS^\geq_{\geq \nu} \ \woo \ \CS^\geq \ \woo \ \CS^+ \label{eqn:delta 2 e} \\
&\Delta^{(2)} (F) \in \CS^-_{< \nu} \ \woo \ \CS^\leq \ \woo \ \CS^\leq \label{eqn:delta 2 f}
\end{align}
By \eqref{eqn:orthogonal}, any first tensor of $\Delta^{(2)}(E)$ pairs trivially with any first tensor of $\Delta^{(2)}(F)$. Thus, \eqref{eqn:drinfeld double shuffle reversed} implies
$$
EF = F_1 E_1 \Big \langle E_2, S(F_2) \Big \rangle
$$
By \eqref{eqn:slope 1} and \eqref{eqn:slope 2}, we have $E_1 \in \CS^\geq_{\geq \nu}$ and $F_1\in \CS^-_{< \nu}$ in the formula above, as we were required to prove.

\end{proof}

\medskip

\subsection{Triangular decompositions}
\label{sub:triangular decompositions}

Compare the following result with \eqref{eqn:double shuffle}.

\medskip

\begin{proposition}
\label{prop:a triangular}

For any $\nu \in \BQ$, the multiplication map induces isomorphisms
\begin{equation}
\label{eqn:a triangular}
\CA^{\geq}_\nu \otimes \CA^{\leq}_\nu \xrightarrow{\sim} \CS \xleftarrow{\sim} \CA^{\leq}_\nu \otimes \CA^{\geq}_\nu
\end{equation}

\end{proposition}

\medskip

\begin{proof} We will only prove the fact that the left-most multiplication map
$$
m : \CA^{\geq}_\nu \otimes \CA^{\leq}_\nu \rightarrow \CS
$$
is an isomorphism, as the case of the right-most map is analogous. To prove that $m$ is surjective, we must show that for any $E \in \CS^\geq$ and $F \in \CS^-$, we have
\begin{equation}
\label{eqn:permutation}
EF \in \CS^-_{< \nu} \otimes \CS^{\geq} \otimes \CS^-_{\geq \nu}
\end{equation}
We will do so by reverse induction on $|\hdeg F| \in - \BN$. By \eqref{eqn:factorization 1}-\eqref{eqn:factorization 4}, we have
\begin{equation}
\label{eqn:**}
\CS^- = \CS^-_{<\nu} \otimes \CS^-_{\geq \nu}
\end{equation}
so it suffices to prove the induction step of \eqref{eqn:permutation} for $F \in \CS^-_{<\nu}$. However, relation \eqref{eqn:drinfeld double shuffle reversed} means that we may rewrite the shuffle element $EF$ from \eqref{eqn:permutation} as
\begin{equation}
\label{eqn:the sum}
\Big \langle E_1, F_1 \Big \rangle F_2 E_2 \Big \langle E_3, S(F_3) \Big \rangle
\end{equation}
By applying \eqref{eqn:leading delta minus} twice, we see that there are two types of terms $F_2 E_2$ above:

\medskip

\begin{itemize}[leftmargin=*]

\item those corresponding to the term $1 \otimes F(z_{ia}) \otimes \prod_{i,a} \ph^-_i(z_{ia})$ from $\Delta^{(2)}(F)$. By Claim \ref{claim:preserved}, the corresponding summand in \eqref{eqn:the sum} lies in  $\CS^-_{< \nu} \otimes \CS^\geq$.

\medskip

\item those corresponding to tensors with $|\hdeg F_2| > |\hdeg F|$; in this case, formula \eqref{eqn:drinfeld double shuffle} allows us to rewrite the corresponding terms in \eqref{eqn:the sum} as a linear combination of $E'F'$ with $|\hdeg F'| \geq |\hdeg F_2|$, which lie in $ \CS^-_{< \nu} \otimes \CS^{\geq} \otimes \CS^-_{\geq \nu}$ by the induction hypothesis.

\end{itemize}

\medskip 

\noindent Now that we showed that $m$ is surjective, let us also prove that it is injective. To this end, consider any $F' \in \CS^-_{< \nu}$, $E \in \CS^{\geq}$, $F'' \in\CS^-_{\geq \nu}$. Formula \eqref{eqn:drinfeld double shuffle} gives us
$$
F'EF'' = \Big \langle E_1, S(F'_1) \Big \rangle E_2 F'_2 F''\Big \langle E_3, F'_3 \Big \rangle
$$
However, by \eqref{eqn:leading delta minus} we have $\Delta^{(2)}(F') = 1 \otimes F' \otimes \kappa_{\hdeg F'} + \dots$, where the ellipsis denotes terms with middle tensor factor having $|\text{hdeg}|$ greater than $|\hdeg F'|$, or the same hdeg but greater vdeg. In the latter case, the aforementioned terms of greater vertical degree will still be in $\CS^-_{<\nu}$ by Claim \ref{claim:preserved}. Therefore, we have
$$
F'EF'' = EF'F'' + \dots
$$
where the ellipsis denotes terms which either have greater $|\text{hdeg}|$ than $F'F''$,  or the same hdeg but greater vdeg. Since there are no non-trivial linear relations between the products $EF'F''$ (according to \eqref{eqn:**}), there are no non-trivial linear relations between the products $F'EF''$, which implies the injectivity of $m$.

\end{proof}

\medskip

\noindent For any Kac-Moody $\fg$ and $\nu \in \BQ$, we expect that there exist topological coproducts
\begin{align}
&\Delta_\nu : \CA^{\geq}_{\nu} \rightarrow \CA^{\geq}_{\nu} \  \woo \ \CA^{\geq}_{\nu} 	\label{eqn:coproduct intro plus} \\
&\Delta_\nu : \CA^{\leq}_{\nu} \rightarrow \CA^{\leq}_{\nu} \  \woo \ \CA^{\leq}_{\nu} 	\label{eqn:coproduct intro minus}
\end{align}
which generalize the Drinfeld-Jimbo coproducts on $\UUaffgl$ for finite type $\fg$ (see Proposition \ref{prop:slope affine}). Such coproducts were defined for $\fg$ of affine type $A$ (meaning that the corresponding algebra $\UU$ is a quantum toroidal algebra) in \cite{N Tale}, but they are not known in general \footnote{We refer to \cite{Da} for a glimpse of how complicated the Drinfeld-Jimbo coproduct is in terms of the Drinfeld new presentation, even when $\fg$ is of finite type.}. In simply laced types, it is expected that the sought-for coproducts \eqref{eqn:coproduct intro plus}-\eqref{eqn:coproduct intro minus} correspond to the geometric coproducts defined using $K$-theoretic stable envelopes in \cite{OS}, just as our slope subalgebras $\CB_{\mu}$ are expected to match the analogous constructions of \loccit With this extra structure, the decompositions \eqref{eqn:a triangular} are expected to be Drinfeld doubles.

\medskip 

\subsection{The finite type case}
\label{sub:finite type}

An important role in the present paper will be played by the case when $\nu = 0$. We remark that the subalgebras
\begin{equation}
\label{eqn:do not depend}
\CS_{\geq 0}^\pm, \CS_{\leq 0}^\pm, \CS_{> 0}^\pm, \CS_{< 0}^\pm \text{ and } \CB_0^\pm 
\end{equation}
do not depend on the choice of $\br \in \BQ_{>0}^I$ in \eqref{eqn:the r}. Thus, neither do the subalgebras
\begin{align}
&\CA^\geq = \CA^\geq_0 = \CS^-_{<0}  \otimes \CB_\infty^+ \otimes \CS^+_{\geq 0} \label{eqn:slope zero plus} \\
&\CA^\leq = \CA^\leq_0 = \CS^+_{<0}  \otimes \CB_\infty^- \otimes \CS^-_{\geq 0} \label{eqn:slope zero minus}
\end{align}
of $\CS$. The motivation for the study of these subalgebras stems from the following. 

\medskip

\begin{proposition}
\label{prop:slope affine}

If $\fg$ is of finite type, then the isomorphism
\begin{equation}
\label{eqn:xi}
\Xi : \UUaff \xrightarrow{\Phi} \UU \xrightarrow{\Upsilon} \CS
\end{equation}
sends the subalgebra $\UUaffg$ onto $\CA^\geq$ and the subalgebra $\UUaffl$ onto $\CA^\leq$.

\end{proposition}

\medskip

\begin{proof} It is clear that 
\begin{align}
&\Xi(e_i) = (z_{i1}^0)^+ \in \CB^+_0 \subset \CA^\geq \\
&\Xi(\kappa_i) = \kappa_i \in \CB_{\infty}^+ \subset \CA^\geq
\end{align}
for all $i \in I$.  In the next paragraph, we will show that
\begin{equation}
\label{eqn:want}
\Xi(e_0) \in \CA^\geq
\end{equation}
Once we have done this, we will have shown that 
\begin{equation}
\label{eqn:first inclusion}
\Xi(\UUaffg) \subseteq \CA^\geq
\end{equation}
By analogy, one can also prove that
\begin{equation}
\label{eqn:second inclusion}
\Xi(\UUaffl) \subseteq \CA^\leq
\end{equation}
where $\UUaffl \subset \UUaff$ is generated by $\{f_i,\kappa_i^{-1}\}_{i \in \wI}$. However, the facts that
$$
\UUaff = \UUaffg \otimes \UUaffl \quad \text{and} \quad \CS = \CA^\geq \otimes \CA^\leq
$$
(the former being the well-known triangular decomposition of quantum affine algebras, and the latter being \eqref{eqn:a triangular}) together with the fact that $\Xi$ is an isomorphism, would imply that the inclusions in \eqref{eqn:first inclusion} and \eqref{eqn:second inclusion} are actually equalities.

\medskip

\noindent It remains to prove \eqref{eqn:want}. Because of \eqref{eqn:iterated commutator e}, this boils down to showing that
\begin{equation}
\label{eqn:anto}
R \in \CS^-_{<0} \quad \Rightarrow \quad \left[ f, R \right]_q \in \CS^-_{<0}
\end{equation}
where $f = (z_{i1}^0)^-$ for any given $i \in I$. Formula \eqref{eqn:naive slope 2} implies that
$$
\Delta(R) = 1 \otimes R + (\text{vdeg} > 0) \woo (\text{anything})
$$
while either \eqref{eqn:coproduct f} or \eqref{eqn:coproduct shuffle minus} imply that
$$
\Delta(f) \in 1 \otimes f + f \otimes \kappa_i^{-1} + (\text{vdeg} > 0) \woo (\text{anything})
$$
Therefore, we conclude that
$$
\Delta \left( f R- q^{(\hdeg R, -\bs^i)} R f \right) \in 1 \otimes \left( f R- q^{(\hdeg R, -\bs^i)} R f \right) + 
$$
$$
+ f \otimes \Big[ \kappa_i^{-1}R - q^{(\hdeg R, -\bs^i)}  R\kappa_i^{-1}\Big] + (\text{vdeg} > 0) \woo (\text{anything})
$$
By \eqref{eqn:k shuffle}, the term in square brackets vanishes. Then \eqref{eqn:naive slope 2} implies that 
$$
\left[ f, R \right]_q = f R - q^{(\hdeg R, -\bs^i)} R f
$$
lies in $\CS^-_{<0}$, as required.

\end{proof}

\medskip

\noindent Note that Proposition \ref{prop:slope affine} proves Conjecture 2.25 of \cite{NT2}. 

\medskip

\subsection{Subalgebras and generators}
\label{sub:gen}

Let us consider the subalgebra
\begin{equation}
\label{eqn:the inclusion}
\mathring{\CS}_{\geq 0}^+ \subseteq \CS_{\geq 0}^+
\end{equation}
generated by $\{z_{i1}^d\}_{i \in I, d \geq 0}$, and its graded summands $\mathring{\CS}_{\geq 0|\bn} \subseteq \CS_{\geq 0|\bn}$.

\medskip

\begin{proposition}
\label{prop:non-negative}

For $\fg$ of finite type, the inclusion \eqref{eqn:the inclusion} is an equality.

\end{proposition}

\medskip

\begin{proof}

It is well-known that $\Phi(\UUaffg) \cap \UUp$ is generated by $\{e_{i,d}\}_{i \in I, d \geq 0}$ (see \cite[Section 2.3]{HJ}). Passing this statement under the isomorphism $\Upsilon$ implies 
$$
\CA^\geq \cap \CS^+ \stackrel{\eqref{eqn:a geq}}= \CS_{\geq 0}^+
$$
is generated by $\{z_{i1}^d\}_{i \in I, d \geq 0}$, thus implying that $\mathring{\CS}_{\geq 0}^+ = \CS_{\geq 0}^+$. 

\end{proof}

\medskip

\noindent The inclusion \eqref{eqn:the inclusion} is strict even for $\fg$ of affine type (for example, imaginary root vectors in the horizontal subalgebra of $U_q(L \widehat{\fsl}_n)$ lie in $\CB_0^+$ but not in $\oCS^+_{\geq 0}$, see \cite{N Toroidal}). 

\medskip

\begin{proposition}
\label{prop:weaker}

For any Kac-Moody Lie algebra $\fg$ and any $\bn \in \nn$, we have
\begin{equation}
\label{eqn:span of positive}
\CS_{\geq 0|\bn} \subseteq \emph{span } \Big\{z_{i_11}^{d_1} * \dots * z_{i_n1}^{d_n}\Big\}_{\bs^{i_1}+\dots+\bs^{i_n} = \bn}^{d_1,\dots,d_n \geq -N}
\end{equation}
for large enough $N$ (which may depend on $\bn$).

\end{proposition}

\medskip

\begin{proof}

 Assume $\br \in \BZ_{>0}^I$. By the definition of $\CS^+$ in Theorem \ref{thm:quantum to shuffle}, it is generated as an algebra by $\{z_{i1}^d\}_{i \in I,d \in \BZ}$. In horizontal degrees up to $\bn$, only finitely many of the subalgebras $\{\CB_{\mu}^+\}_{\mu \in [0,1)}$ are non-zero, and the non-zero ones are finite-dimensional by the first statement of Proposition \ref{prop:easy}. Therefore, we can choose $N$ large enough (depending on $\bn$) so that any shuffle element generated by $\{\CB_{\mu}^+\}_{\mu \in [0,1)}$ up to horizontal degree $\bn$ lies in the span in the right-hand side of \eqref{eqn:span of positive}. By the second statement of Proposition \ref{prop:easy}, the same is true for any shuffle element generated by $\{\CB_{\mu}^+\}_{\mu \geq 0}$. As such shuffle elements span the whole of $\CS_{\geq 0}^+$ by \eqref{eqn:factorization 3}, we are done.

\end{proof}

\medskip

\section{Representation theory via shuffle algebras}
\label{sec:representation}

\medskip

\subsection{Highest $\ell$-weight modules}
\label{sub:highest weight general}

In the present Section, we will study representations of $\CA^\geq$ (which generalizes the Borel subalgebra $\UUaffg$ of quantum affine algebras for finite type $\fg$, as we saw in Proposition \ref{prop:slope affine}). Most of the basic notions here are direct adaptations of those of Subsections \ref{sub:representations} and \ref{sub:q-characters affine}.

\medskip

\begin{definition}
\label{def:ell weight general}

An $\ell$-\textbf{weight} is a collection of invertible power series
\begin{equation}
\label{eqn:ell weight general}
\bpsi = \left(\psi_i(z) = \sum_{d=0}^{\infty} \frac {\psi_{i,d}}{z^d} \in \BC[[z^{-1}]]^* \right)_{i \in I}
\end{equation}
Such a $\bpsi$ is called 

\medskip

\begin{itemize}[leftmargin=*]

\item \textbf{rational} if every $\psi_i(z)$ is the expansion of a rational function

\medskip

\item \textbf{regular} if every $\psi_i(z)$ is the expansion of a rational function which is regular at $z = 0$ (this rational function is already regular and non-zero at $z = \infty$ by \eqref{eqn:ell weight general})

\medskip

\item \textbf{polynomial} if every $\psi_i(z)$ is a polynomial in $z^{-1}$

\medskip

\item \textbf{constant} if there exists $\bom \in \cc$ such that $\psi_i(z) = \psi_{i,0} = q^{(\bom,\bs^i)}$ for all $i \in I$.

\end{itemize}

\medskip

\end{definition}

\medskip 

\noindent Consider any $\ell$-weight $\bpsi$ and any $\bom \in \cc$. If we have
\begin{equation}
\label{eqn:psi to gamma}
\psi_{i,0} = q^{(\bom,\bs^i)}, \quad \forall i \in I
\end{equation}
then we will write $\bom = \lead(\bpsi)$.  Above and henceforth, we identify with each other those direct summands $V_{\bom}$ whose subscripts differ by an element in the kernel of the symmetric pairing \eqref{eqn:symmetric pairing}. One defines $\ell$-weight decompositions of representations of $\CA^\geq$ by direct analogy with \eqref{eqn:ell weight decomposition}-\eqref{eqn:ell weight}. For $\bom \in \cc$, the symbol $[\bom]$ will denote the corresponding constant $\ell$-weight, as in Definition \ref{def:ell weight general}.

\medskip

\begin{definition}
\label{def:verma module}

For any $\ell$-weight $\bpsi$, consider the so-called Verma-like module
\begin{equation}
\CA^\geq \curvearrowright W(\bpsi)
\end{equation}
generated by a single vector $\vac$ modulo the relations
\begin{equation}
\ph_i^+(z) \cdot \vac = \psi_i(z) \vac
\end{equation}
for all $i \in I$, and $\CS_{\geq 0|\bn} \cdot \vac = 0$ for any $\bn > \b0$.

\end{definition}

\medskip

\noindent Because of the triangular decomposition \eqref{eqn:slope zero plus}, we have a vector space isomorphism
\begin{equation}
\label{eqn:iso verma}
W({\bpsi}) \cong \CS^-_{<0}
\end{equation}
Thus, $W(\bpsi)$ inherits a horizontal grading from $\CS^-_{<0}$, which we will shift as
$$
\text{hdeg} \left( F\vac \right) = \bom - \bn
$$
for any $F \in \CS_{<0|-\bn}$, where $\bom = \lead(\bpsi)$. We make this choice so that the weight of $F\vac$ matches the horizontal degree defined above, i.e. $\kappa_iF\vac = q^{(\bom - \bn,\bs^i)} F\vac$.

\medskip

\subsection{Simple modules}
\label{sub:simple}

As is often the case in representation theory, simple modules arise as quotients of Verma-like modules such as the ones of Definition \ref{def:verma module}.

\medskip

\begin{proposition}
\label{prop:ideal} 

For any $\ell$-weight $\bpsi$, consider the linear subspace
\begin{equation}
	\label{eqn:graded j}
J(\bpsi) = \bigoplus_{\bn \in \nn}	J(\bpsi)_{\bn} \subseteq  \bigoplus_{\bn \in \nn} \CS_{<0|-\bn} = \CS^-_{<0} 
\end{equation}
consisting of those shuffle elements $F(z_{i1},\dots,z_{in_i})_{i \in I} \in \CS_{<0|-\bn}$ such that \footnote{Note that the left-hand side of \eqref{eqn:psi pairing} is well-defined for all $E$ and $F$, because the pairing is trivial unless the total vertical degree of its arguments is 0. Thus, only finitely many terms of the power series $\psi_i(z)$ actually contribute to the left-hand side of \eqref{eqn:psi pairing} for any given $E$ and $F$.}
\begin{equation}
\label{eqn:psi pairing}
\left \langle E(z_{i1},\dots,z_{in_i}) \prod_{i \in I} \prod_{a=1}^{n_i} \psi_i(z_{ia}) , S \left( F(z_{i1},\dots,z_{in_i}) \right) \right \rangle = 0
\end{equation}
$\forall E \in \CS_{\geq 0|\bn}$. Then $J(\bpsi) \vac$ is the unique maximal graded $\CA^\geq$ submodule of $W(\bpsi)$. 

\end{proposition}

\medskip

\begin{example} 
\label{ex:2}

For $\fg = \fsl_2$, $J(\psi)$ is the set of $F \in \CS_{<0|-n}$ such that
\begin{equation}
\label{eqn:formula ex 2}
\int_{|z_1| \ll \dots \ll |z_n|} \frac {z_1^{d_1} \dots z_n^{d_n} F(z_1,\dots,z_n)}{\prod_{1\leq a < b \leq n} \frac {z_b - z_a q^{-2}}{z_b-z_a}} \psi(z_1) \dots \psi (z_n) = 0
\end{equation}
$\forall d_1,\dots,d_n \geq 0$. This is a consequence of \eqref{eqn:antipode pairing shuffle} and Proposition \ref{prop:non-negative}.

\end{example}

\medskip 

\begin{proof} We need to show that the subspace $J(\bpsi) \vac \subseteq W(\bpsi)$ is preserved by 

\medskip

\begin{enumerate}[leftmargin=*]

\item left multiplication with $\CS^-_{<0}$

\medskip

\item left multiplication with the $\ph_{i,d}^+$'s, or equivalently, with $\{\kappa_j,p_{j,u}\}_{j \in I, u > 0}$

\medskip

\item left multiplication with $\CS^+_{\geq 0}$

\end{enumerate}

\medskip

\noindent To prove (1), let us consider any $F' \in \CS^-_{<0}$, $F'' \in  J(\bpsi)$ and $E \in \CS^+_{\geq 0}$. The fact that the antipode map is an antiautomorphism and \eqref{eqn:bialgebra 1} imply that
\begin{equation}
\label{eqn:*}
\left \langle E \prod \psi , S(F' * F'') \right \rangle = \left \langle \Delta(E \prod \psi) , S(F'') \otimes S(F') \right \rangle
\end{equation}
where $E\prod \psi$ is shorthand for $E(z_{i1},\dots,z_{in_i})\prod_{i \in I} \prod_{a=1}^{n_i} \psi_i(z_{ia})$. The following Claim is an easy consequence of \eqref{eqn:bialgebra 2} and \eqref{eqn:antipode pairing}, which we leave as an exercise to the reader. It is the reason we used $S(F)$ instead of $F$ in \eqref{eqn:psi pairing}.

\medskip

\begin{claim}
\label{claim:pairing}

If \eqref{eqn:psi pairing} holds for all $E \in \CS_{\geq 0}^+$, then it also holds for all $E \in \CS_{\geq 0}^{\geq}$.

\end{claim}

\medskip

\noindent By Claim \ref{claim:preserved} and \eqref{eqn:slope 1}, we have $\Delta(E \prod \psi) \in \CS^\geq_{\geq 0} \otimes \CS^+$. Then Claim \ref{claim:pairing} and the fact that $F'' \in J(\bpsi)$ imply that the first tensor factors in the RHS of \eqref{eqn:*} have pairing 0. Thus, the whole pairing in \eqref{eqn:*} is 0, hence $F' * F'' \in J(\bpsi)$, as required.

\medskip

\noindent To prove (2), recall from \eqref{eqn:k shuffle}-\eqref{eqn:p shuffle} that commuting $F$ with $\kappa_j$ and $p_{j,u}$ amounts to multiplying $F$ by either a scalar or a color-symmetric polynomial $\rho = \rho(z_{ia})_{i \in I, 1 \leq a \leq n_i}$. As can be readily seen from formulas \eqref{eqn:pairing shuffle plus}-\eqref{eqn:pairing shuffle minus}, 
$$
\Big \langle E \rho, F \Big \rangle = \Big \langle E, F \rho \Big \rangle, \quad \forall  E \in \CS^+, \ F \in \CS^-, \ \rho \text{ Laurent polynomial}
$$
Therefore, replacing $F \leadsto F \rho$ in \eqref{eqn:psi pairing} has the same effect as replacing $E \leadsto E \rho$. By Claim \ref{claim:preserved}, we have $E \in \CS_{\geq 0}^+ \Rightarrow E\rho \in \CS_{\geq 0}^+$ for any color-symmetric polynomial $\rho$. Thus, the fact that relation \eqref{eqn:psi pairing} holds $\forall E \in \CS_{\geq 0}^+$ and given $F$ implies that \eqref{eqn:psi pairing} holds $\forall E \in \CS_{\geq 0}^+$ and given $F\rho$; this establishes (2).

\medskip

\noindent For statement (3), we invoke \eqref{eqn:drinfeld double shuffle reversed} for any $E \in \CS_{\geq 0|\bn}$ and $F \in \CS^-_{<0}$:
\begin{equation}
\label{eqn:base}
EF = \Big \langle E_1, F_1 \Big \rangle F_2E_2 \Big \langle E_3, S(F_3) \Big \rangle = F_1E_1 \Big \langle E_2, S(F_2) \Big \rangle
\end{equation}
The second equality is due to the fact that $E_1$ and $F_1$ have slopes $\geq 0$ and $<0$ (respectively, see \eqref{eqn:delta 2 e}-\eqref{eqn:delta 2 f}), and formula \eqref{eqn:orthogonal}. Applying \eqref{eqn:base} to $\vac$ gives us
\begin{equation}
\label{eqn:r on vac}
EF \vac = F_1 \vac \cdot \left \langle E(z_{i1},\dots,z_{in_i}) \prod_{i \in I} \prod_{a=1}^{n_i} \psi_i(z_{ia}) , S ( F_2 ) \right \rangle 
\end{equation}
because $\Delta(E) = \prod_{i\in I} \prod_{a=1}^{n_i} \ph_i^+(z_{ia}) \otimes E$ plus terms whose first tensor factor has $\text{hdeg} > \b0$ (see \eqref{eqn:leading delta plus}), and thus annihilates $\vac$. Recall from \eqref{eqn:slope 2} that all $F_1$ which appear in the formula above lie in $\CS^-_{<0}$. To show that the right-hand side of \eqref{eqn:r on vac} lies in $J(\bpsi) \vac$, it suffices to show that for any $E' \in \CS_{\geq 0|\bm}$ we have 
$$
0 = \left \langle E'(z_{i1},\dots,z_{im_i}) \prod_{i \in I} \prod_{a=1}^{m_i} \psi_i(z_{ia}) , S ( F_1 ) \right \rangle   \left \langle E(z_{i1},\dots,z_{in_i}) \prod_{i \in I} \prod_{a=1}^{n_i} \psi_i(z_{ia}) , S ( F_2 ) \right \rangle 
$$
$$
\stackrel{\eqref{eqn:bialgebra 2}}= \left \langle (E' * E)(z_{i1},\dots,z_{i,m_i+n_i}) \prod_{i \in I} \prod_{a=1}^{m_i+n_i} \psi_i(z_{ia}) , S ( F ) \right \rangle
$$
However, the pairing above is 0 because $\CS^+_{\geq 0}$ is closed under $*$, and $F \in J(\bpsi)$.

\medskip

\noindent Having showed that $J(\bpsi) \vac$ is a graded $\CA^\geq$ submodule of $W(\bpsi)$, it remains to show that it is the unique such maximal graded submodule. To this end, choose any $F \in \CS_{<0|-\bn} \backslash J(\bpsi)_{\bn}$, which means that there exists $E \in \CS_{\geq 0|\bn}$ such that
$$
\left \langle E(z_{i1},\dots,z_{in_i}) \prod_{i \in I} \prod_{a=1}^{n_i} \psi_i(z_{ia}) , S ( F(z_{i1},\dots,z_{in_i}) ) \right \rangle =: \alpha \neq 0
$$
Because $\Delta(F) = 1 \otimes F$ plus tensors whose second factor has hdeg $> \hdeg F$, formula \eqref{eqn:r on vac} reads precisely $EF \vac = \alpha \vac$. This implies that any graded submodule of $W(\bpsi)$ which strictly contains $J(\bpsi) \vac$ must contain the highest weight vector $\vac$, and thus must be the whole of $W(\bpsi)$. 

\medskip

\end{proof}

\begin{corollary}
\label{cor:simple module}

For any $\ell$-weight $\bpsi = (\psi_i(z))_{i \in I}$, the quotient
\begin{equation}
\label{eqn:simple module}
L(\bpsi) = W(\bpsi) \Big/ J(\bpsi) \vac 
\end{equation}
is the unique (up to isomorphism) simple graded $\CA^\geq$ module generated by a single vector $v$ subject to
$$
\ph_i^+(z) \cdot v  = \psi_i(z) v
$$
for all $i \in I$, and $\CS_{\geq 0|\bn} \cdot v = 0$ for all $\bn > \b0$.  

\end{corollary}

\medskip

\noindent For finite type $\fg$, the (type 1) simple finite-dimensional $\UUaff$ modules were shown in \cite{CP} to be isomorphic to $L(\bpsi)$ for 
\begin{equation}
\label{eqn:finite dimensional psi}
\bpsi = \left( \psi_i(z) =  \prod_{\alpha \in \BC^*} \frac {z q_i - \alpha q^{-1}_i }{z - \alpha}\right)_{i \in I}
\end{equation}
(with finitely many $\alpha$'s in the product above, not necessarily distinct). In particular, such $\bpsi$ are regular $\ell$-weights, in the sense of Definition \ref{def:ell weight general}.

\medskip

\subsection{Category $\CO$}
\label{sub:category O}

The following is the natural generalization of Definition \ref{def:category o affine}.

\medskip

\begin{definition}
\label{def:category o general}

A complex representation $\CA^\geq \curvearrowright V$ is said to be in \textbf{category $\CO$} if
\begin{equation}
\label{eqn:direct sum category O}
V = \bigoplus_{\bom \in \cup_{s=1}^t (\bom^s - \nn) } V_{\bom}
\end{equation}
for finitely many $\bom^1,\dots,\bom^t \in \cc$, such that every
\begin{equation}
\label{eqn:weight general}
V_{\bom} = \Big\{v \in V \text{ s.t. } \kappa_i \cdot v = q^{(\bom, \bs^i)}v, \ \forall i\in I \Big\}
\end{equation}
is finite-dimensional.

\end{definition}

\medskip

\noindent The following is the natural generalization of the last sentence in Theorem \ref{thm:simple}.

\medskip

\begin{theorem}
\label{thm:is rational} 

The simple $\CA^\geq$-modules in category $\CO$ are precisely $L(\bpsi)$ for rational $\ell$-weights $\bpsi$ (in the sense of Definition \ref{def:ell weight general}).

\end{theorem}

\medskip

\begin{proof} As a vector space, the simple module \eqref{eqn:simple module} is isomorphic to the quotient
\begin{equation}
\label{eqn:simple}
L(\bpsi) \cong \CS^-_{<0} \Big / J(\bpsi)
\end{equation}
It is easy to see that its weight spaces are given by (let $\bom = \lead(\bpsi)$)
\begin{equation}
\label{eqn:weight simple}
L(\bpsi)_{\bom - \bn} \cong \CS_{<0|-\bn} \Big / J(\bpsi)_{\bn}
\end{equation}
as $\bn$ varies over $\nn$. Thus, $L(\bpsi)$ does indeed decompose as a direct sum \eqref{eqn:direct sum category O} with $t = 1$ and $\bom^1 = \bom$. Therefore, it remains to show that the weight spaces \eqref{eqn:weight simple} are finite-dimensional if and only if $\bpsi$ is rational. We start with the ``only if" statement: for any $i \in I$, the vector space $\CS_{<0|-\bs^i} = z_{i1}\BC[z_{i1}]$ is infinite dimensional. If $L(\bpsi)_{\bom - \bs^i}$ is to be finite dimensional, then there must exist a polynomial
$$
0 \neq Q(z_{i1}) \in J(\bpsi)_{\bs^i}
$$
The defining property \eqref{eqn:psi pairing} for $\bn = \bs^i$ implies that for all $d \geq 0$, we have
$$
0 = \left \langle z_{i1}^{d} \psi_i(z_{i1}) , S(Q(z_{i1})) \right \rangle = \text{constant term of } \Big[ - z_{i1}^{d} \psi_i(z_{i1}) Q(z_{i1}) \Big]
$$
The equality above can hold for all $d \geq 0$ only if $\psi_i(z_{i1}) Q(z_{i1})$ equals a polynomial $P(z_{i1})$, which would imply that $\psi_i(z)$ is the expansion of a rational function.

\medskip

\noindent For the ``if" statement, we must prove that for any $\bn \in \nn$, finitely many linear conditions on $F \in \CS_{<0|-\bn}$ will ensure that \eqref{eqn:psi pairing} holds for all $E \in \CS_{\geq 0|\bn}$. By \eqref{eqn:span of positive}, it suffices to show that finitely many linear conditions on $F \in \CS_{<0|-\bn}$ ensure that
\begin{multline}
\label{eqn:the computation}
0 = \left\langle e_{i_1,d_1} * \dots * e_{i_n,d_n}  \prod \psi, S(F) \right \rangle \stackrel{\eqref{eqn:antipode pairing shuffle}}= \\ (-1)^n \int_{1 \ll |z_1| \ll \dots \ll |z_n|} \frac {z_1^{d_1} \dots z_n^{d_n} F (z_1,\dots,z_n)}{\prod_{1\leq a < b \leq n} \zeta_{i_bi_a} \left(\frac {z_b}{z_a} \right)} \prod_{a=1}^n \psi_{i_a}(z_a)
\end{multline}
for all $d_1,\dots,d_n \geq -N$ (we may henceforth fix any one of the finitely many orderings $i_1,\dots,i_n$ of $\bn$, see Definition \ref{def:i ordering}). The symbol $1 \ll$ in the subscript of $\int$ means that the contours of integration of the variables $z_1,\dots,z_n$ must be very large relative to the finitely many poles of the rational functions $\psi_{i_a}(z_a)$. Therefore, we may compute \eqref{eqn:the computation} by sending the variable $z_1$ to 0, then sending $z_2$ to 0, $\dots$, then sending $z_n$ to $0$. At step $b$ of this process, we must account for the residues when

\medskip

\begin{itemize}[leftmargin=*]

\item $z_b$ is one of the poles of $\psi_{i_b}$ 

\medskip

\item $z_b = z_aq^{-d_{i_ai_b}}$ for some $a<b$

\medskip 

\item $z_b = 0$

\end{itemize}

\medskip 

\noindent In all three cases above, finitely many linear conditions on $F$ ensure the vanishing of the iterated residue in the variables $z_1,\dots,z_n$ (for example, in the third bullet we require the vanishing of $F$ up to and including order $N+r_{i_b}$ at $z_b = 0$, where $r_{i_b}$ is the order of the pole of $\psi_{i_b}(z_b)$ at $z_b = 0$). This concludes our proof.

\end{proof}

\medskip

\subsection{The polynomial case}
\label{sub:integral}

An important instance of simple modules are the ones with polynomial highest $\ell$-weight, namely
$$
\btau = (\tau_i(z))_{i \in I} \in \left( \BC^* + z^{-1} \BC[z^{-1}] \right)^I
$$
The following shows that as a graded vector space, $L(\btau)$ only depends on the $I$-tuple $\br = \ord \btau$ of degrees of the $\tau_i$ in $z^{-1}$ (up to a grading shift by $\bom = \lead(\btau)$).

\medskip

\begin{proposition}
\label{prop:integral}

For any polynomial $\ell$-weight $\btau$ as above, $J(\btau)$ coincides with \begin{equation}
\label{eqn:j integral}
J^{\br} := \left\{F \in \CS_{<0}^- \text{ s.t. } \left \langle E (z_{ia}) \prod_{i,a} z_{ia}^{-r_i}, S(F(z_{ia})) \right \rangle = 0, \forall E \in \CS_{\geq 0}^+ \right\} 
\end{equation}
and so inherits an extra grading by $d = \emph{vdeg }F$ from $J^{\br}$. In other words, if we let
\begin{equation}
	\label{eqn:l integral}
	L^{\br} = \bigoplus_{\bn \in \nn} \bigoplus_{d \in \BN} L^{\br}_{-\bn,d} :=  \bigoplus_{\bn \in \nn} \bigoplus_{d \in \BN} \CS_{<0|-\bn,d} \Big/ J^{\br}_{\bn,d} =  \CS^-_{<0} \Big/ J^{\br} 
\end{equation}
(which is well-defined for any $\br \in \zz$), then we may define
\begin{equation}
\label{eqn:grading ol}
L(\btau) = \bigoplus_{\bn \in \nn} \bigoplus_{d \in \BN} L(\btau)_{\bom-\bn,d}
\end{equation}
where $L(\btau)_{\bom-\bn,d}$ matches $L^{\br}_{-\bn,d}$ under the equality of vector spaces $L(\btau) = L^\br$.

\end{proposition}

\medskip

\begin{example}
\label{ex:3}

When $\fg = \fsl_2$, \eqref{eqn:formula ex 2} shows that $J^r$ is the set of $F$ such that
\begin{equation}
\label{eqn:formula ex 3}
\int_{|z_1| \ll \dots \ll |z_n|} \frac {z_1^{d_1} \dots z_n^{d_n} F(z_1,\dots,z_n)}{\prod_{1\leq a < b \leq n} \frac {z_b - z_a q^{-2}}{z_b-z_a}} = 0
\end{equation}
$\forall d_1,\dots,d_n \geq -r$, which is equivalent to $F \in (z_1\dots z_n)^{r+1} \BC[z_1,\dots,z_n]^{\emph{sym}}$. Comparing this with the fact $F \in \CS_{<0|-n} = z_1\dots z_n \BC[z_1,\dots,z_n]^{\emph{sym}}$, we have
$$
J^r_n = \begin{cases} \CS_{<0|-n}  = z_1\dots z_n \BC[z_1,\dots,z_n]^{\emph{sym}}& \text{if }r <0 \\ \CS_{<-r|-n} = (z_1\dots z_n)^{r+1} \BC[z_1,\dots,z_n]^{\emph{sym}} & \text{if }r \geq 0 \end{cases}
$$
for any $n > 0$. The $d$-grading of $J^r_n$ is given by homogeneous degree.

\end{example}

\medskip

\begin{proof} Let us write $\tau_i(z) = z^{-r_i} P_i(z)$, where $P_i$ is a polynomial in $z$ with non-zero constant term. By \eqref{eqn:psi pairing}, a shuffle element $F \in \CS_{<0|-\bn}$ lies in $J(\btau)$ if and only if  
\begin{equation}
\label{eqn:psi pairing 3}
\left \langle E(z_{i1},\dots,z_{in_i}) \prod_{i \in I} \prod_{a=1}^{n_i} \Big(z_{ia}^{-r_i} P_i(z_{ia}) \Big), S(F(z_{i1},\dots,z_{in_i})) \right \rangle = 0
\end{equation}
for all $E \in \CS_{\geq 0|\bn}$. Since the pairing is trivial unless the total vertical degree of its arguments is 0, the LHS of \eqref{eqn:psi pairing 3} is automatically 0 if the homogeneous degree of $E$ is large enough (for fixed $F$). Since one can always choose polynomials $Q_i$ such that $P_i(z)Q_i(z) = 1 + O(z^N)$ for large enough $N$, and $\CS_{\geq 0|\bn}$ is closed under multiplication with color-symmetric polynomials (Claim \ref{claim:preserved}), then \eqref{eqn:psi pairing 3} is equivalent to 
$$
\left \langle E(z_{i1},\dots,z_{in_i}) \prod_{i \in I} \prod_{a=1}^{n_i} z_{ia}^{-r_i} , S(F(z_{i1},\dots,z_{in_i})) \right \rangle = 0.
$$

\end{proof}

\medskip

\noindent Even though as a vector space, $L(\btau)$ only depends on $\br = \ord \btau$, as a $\CA^{\geq}$ module the whole of $\btau$ is important. This is reflected in the fact that in the RHS of \eqref{eqn:act 3 intro}, we cannot pinpoint exactly how the action of $E \in \CS_{\geq 0}$ affects the $d$ grading of $L(\btau)$.

\medskip

\begin{proposition}
\label{prop:bigrading}

The $\CA^{\geq}$ action affects the grading of $L(\btau)$ as in \eqref{eqn:act 1 intro}-\eqref{eqn:act 3 intro}.

\end{proposition}

\medskip

\begin{proof} The action of $F \in \CS_{<0}^-$ on $F'\vac \in L(\btau)$ is simply given by shuffle multiplying $F$ with $F'$, which leads to formula \eqref{eqn:act 1 intro}. Similarly, the fact that
\begin{equation}
\label{eqn:action p}
[p_{j,u}, F]\vac \stackrel{\eqref{eqn:p shuffle}}= \sum_{i \in I} \text{constant}_{i,j} \cdot F(z_{i1}^u+\dots+z_{in_i}^u)\vac
\end{equation}
implies that the coefficient $\left[\frac {\ph^+_j(z)}{\tau_j(z)}\right]_{z^{-u}}$ sends $F\vac$ to $F\rho\vac$ for some degree $u$ polynomial $\rho$, which leads to \eqref{eqn:act 2 intro}. Finally, the action of $E \in \CS_{\geq 0|\bn}$ on $L(\btau)$ is
\begin{equation}
	\label{eqn:action e}
EF \vac \stackrel{\eqref{eqn:r on vac}}= F_1 \vac \cdot \left \langle E(z_{i1},\dots,z_{in_i}) \prod_{i \in I} \prod_{a=1}^{n_i} \tau(z_{ia}) , S ( F_2 ) \right \rangle 
\end{equation}
so the effect of $E$ on $F\vac$ is to send it to a multiple of $F_1\vac$. In terms of vertical degree, this corresponds to a decrease by $\vdeg F_2$. Since the pairing is non-zero only if the total vertical degree of its arguments is 0, then the only terms which appear in the right-hand side of \eqref{eqn:action e} are those for which $\vdeg F_2 \in [ - \vdeg E, \br \cdot \bn - \vdeg E]$. This precisely leads to formula \eqref{eqn:act 3 intro}. 
	
\end{proof}

\subsection{The positive case}
\label{sub:positive}

When $\br \in \BZ_{>0}^I$, we will construct a factorization of
$$
L^{\br} = \CS_{<0}^- \Big/ J^{\br}
$$
starting from the factorizations \eqref{eqn:factorization 1}-\eqref{eqn:factorization 4}, in order to prove the  Mukhin-Young conjecture \eqref{eqn:my i}. Consider the slope subalgebras $\CB_\mu^\pm$ defined with respect to $\br$ in \eqref{eqn:the r}. The following Lemma is the crucial step toward obtaining Theorem \ref{thm:my refined}.

\medskip

\begin{lemma} 
\label{lem:composition}

We have an isomorphism of $(-\nn) \times \BN$ graded vector spaces 
\begin{equation}
\label{eqn:composition}
L^{\br} = \bigotimes^\leftarrow_{\mu \in [-1,0)} \CB_\mu^-
\end{equation}

\end{lemma}

\begin{proof} Let $\CS_{< - 1}^{-,\text{max}} = \bigoplus_{\bn > \b0} \CS_{<-1|-\bn}$. Formulas \eqref{eqn:factorization 1}-\eqref{eqn:factorization 4} imply that
$$
\CS^- =  \CS_{< -1}^- \cdot \CS_{ \geq -1}^- = \CS_{\geq -1}^- \oplus \CS_{< - 1}^{-,\text{max}} \cdot \CS_{\geq -1}^-
$$
while formulas \eqref{eqn:orthogonal} and \eqref{eqn:j integral} imply that some $F \in \CS^-_{<0}$ lies in $J^{\br}$ if and only 
$$
S(F) \in \CS_{< - 1}^{-,\text{max}} \cdot \CS_{\geq -1}^- \quad \Leftrightarrow \quad F \in S^{-1} ( \CS_{\geq -1}^- ) \cdot S^{-1} \left(\CS_{< - 1}^{-,\text{max}}\right)
$$
However, property \eqref{eqn:slope 2} and the defining property of the antipode implies that
$$
S^{-1} \left(\CS_{< - 1}^{-,\text{max}}\right) \subset \CS^{\leq} \cdot \CS_{< - 1}^{-,\text{max}}
$$
so we conclude that
\begin{equation}
\label{eqn:compare 1}
F \in J^{\br} \quad \Leftrightarrow \quad  F \in \CS^- \cdot \CS_{< - 1}^{-,\text{max}} \stackrel{\eqref{eqn:factorization 1}}= \CS^-_{\geq -1} \cdot \CS_{< - 1}^{-,\text{max}}
\end{equation}
However, we recall that 
\begin{equation}
\label{eqn:compare 2}
F \in \CS_{<0}^- \stackrel{\eqref{eqn:factorization 4}}=  \left( \bigotimes^\leftarrow_{\mu \in [-1,0)} \CB_\mu^- \right) \cdot \CS_{<-1}^-
\end{equation}
so we may replace $\CS_{\geq -1}^-$ in the right-hand side of \eqref{eqn:compare 1} by $\otimes^\leftarrow_{\mu \in [-1,0)} \CB_\mu^-$. Once we make this change, factoring \eqref{eqn:compare 2} by \eqref{eqn:compare 1} precisely implies \eqref{eqn:composition}.

\end{proof}

\subsection{A variant}
\label{sub:variant}

The following will be an important player in the next Section.

\medskip

\begin{definition}
	\label{def:variant}
	
	For any rational $\ell$-weight $\bpsi$, define \footnote{We will only consider $\oL(\bpsi)$ as a graded vector space, though it is naturally a module for the algebra $\mathring{\CA}^{\geq} = \CS_{<0}^- \otimes \CB^+_\infty \otimes \oCS^+_{\geq 0}$, as per the analogue of Propositions \ref{prop:a are algebras} and \ref{prop:ideal}.}
	\begin{equation}
		\label{eqn:ol}
		\oL(\bpsi) = \CS_{<0}^- \Big/ \oJ(\bpsi)
	\end{equation}
	where $\oJ(\bpsi) = \bigoplus_{\bn \in \nn} \oJ(\bpsi)_{\bn}$ consists of those $F \in \CS_{<0|-\bn}$ such that
	\begin{equation}
		\label{eqn:psi pairing 2}
		\left \langle E(z_{i1},\dots,z_{in_i}) \prod_{i \in I} \prod_{a=1}^{n_i} \psi_i(z_{ia}) , S(F(z_{i1},\dots,z_{in_i})) \right \rangle = 0
	\end{equation}
	for all $E \in \oCS_{\geq 0|\bn}$ (see Subsection \ref{sub:gen}).
	
\end{definition}

\medskip

\noindent When $\fg$ is of finite type, Proposition \ref{prop:non-negative} implies that for all $\ell$-weights $\bpsi$, we have
\begin{equation}
	\label{eqn:are equal}
\oJ(\bpsi) = J(\bpsi) \quad \Rightarrow \quad	\oL(\bpsi) = L(\bpsi)
\end{equation}
For a general Kac-Moody $\fg$, the fact that $\oCS_{\geq 0} \subseteq \CS_{\geq 0}$ implies that $\oJ(\bpsi) \supseteq J(\bpsi)$ for any $\ell$-weight $\psi$, and so there is a surjective map $L(\bpsi) \twoheadrightarrow \oL(\bpsi)$. We write
\begin{equation}
\label{eqn:graded ol}
\oL(\bpsi) = \bigoplus_{\bn \in \nn} \oL(\bpsi)_{\bom-\bn}
\end{equation}
where $\bom = \lead(\bpsi)$. Moreover, the natural analogues of Subsection \ref{sub:integral} hold. If
\begin{equation}
	\label{eqn:j integral variant}
	\oJ^{\br} := \left\{F \in \CS_{<0}^- \text{ s.t. } \left \langle E (z_{ia}) \prod_{i,a} z_{ia}^{-r_i}, S(F(z_{ia})) \right \rangle = 0, \forall E \in \oCS_{\geq 0}^+ \right\} 
\end{equation}
\begin{equation}
	\label{eqn:l integral variant}
	\oL^{\br} = \bigoplus_{\bn \in \nn} \bigoplus_{d \in \BN} \oL^{\br}_{-\bn,d} :=  \bigoplus_{\bn \in \nn} \bigoplus_{d \in \BN} \CS_{<0|-\bn,d} \Big/ \oJ^{\br}_{\bn,d} =  \CS^-_{<0} \Big/ \oJ^{\br} 
\end{equation}
for any $\br \in \zz$, then we may define for any polynomial $\ell$-weight $\btau$ 
\begin{equation}
	\label{eqn:grading ol variant}
	\oL(\btau) = \bigoplus_{\bn \in \nn} \bigoplus_{d \in \BN} \oL(\btau)_{\bom-\bn,d}
\end{equation}
where $\oL(\btau)_{\bom-\bn,d}$ matches $\oL^{\ord \btau}_{-\bn,d}$ under the equality of vector spaces $\oL(\btau) = \oL^{\ord \btau}$.

\bigskip

\section{Characters and residues}
\label{sec:characters}

\medskip

\subsection{$q$-characters}
\label{sub:q-characters general}

In \eqref{eqn:simple}, we gave a description of $L(\bpsi)$ as a graded vector space, which we will now use in order to compute its character and $q$-character. Because of Theorem \ref{thm:is rational}, these quantities are well-defined for any rational $\ell$-weight $\bpsi$, see \eqref{eqn:q-character definition}. In more detail, formula \eqref{eqn:weight simple} implies that the character is given by
\begin{equation}
\label{eqn:character general}
\chi(L(\bpsi)) = \sum_{\bn \in \nn} \dim_{\BC} \left(\CS_{<0|-\bn} \Big/ J(\bpsi)_{\bn} \right) [\bom - \bn]
\end{equation}
where $\bom=\lead(\bpsi)$. Similarly, for the variant in Subsection \ref{sub:variant}, we have
\begin{equation}
	\label{eqn:character general variant}
	\chi(\oL(\bpsi)) = \sum_{\bn \in \nn} \dim_{\BC} \left(\CS_{<0|-\bn} \Big/ \oJ(\bpsi)_{\bn} \right) [\bom - \bn]
\end{equation}
because of formulas \eqref{eqn:ol} and \eqref{eqn:graded ol}.

\medskip

\begin{lemma}
\label{lem:is module}

For any $\bn \in \nn$ and any $\ell$-weight $\bpsi$, the vector spaces
\begin{equation}
\label{eqn:is module}
\CS_{<0|-\bn} \Big/ J(\bpsi)_{\bn} \quad \text{and} \quad \CS_{<0|-\bn} \Big/ \oJ(\bpsi)_{\bn}
\end{equation}
are modules over $\CP_{\bn} = \BC[z_{i1},\dots,z_{in_i}]^{\emph{sym}}_{i \in I}$.

\end{lemma}

\medskip

\begin{proof} We must show that the quotients in \eqref{eqn:is module} are preserved by multiplication by color-symmetric polynomials. The fact that $\CS_{<0|-\bn}$ is preserved was showed in Claim \ref{claim:preserved}. The fact that $J(\bpsi)_{\bn}$ and $\oJ(\bpsi)_{\bn}$ are preserved is because they are defined via the duality property \eqref{eqn:psi pairing} with respect to the sets $\CS_{\geq 0|\bn}$ and $\oCS_{\geq 0|\bn}$ (respectively) which are themselves preserved by multiplication with color-symmetric polynomials: the first one because of Claim \ref{claim:preserved}, and the second one because multiplication with the color-symmetric polynomial $\rho = \sum_a z_{ia}^u$ is a derivation (i.e. $\rho (E*E') = (\rho E) * E' + E * (\rho E')$), which clearly sends the generator $z_{i1}^d$ to $z_{i1}^{d+u}$ and $z_{j1}^d$ to $z_{j1}^d$ for all $j \neq i$.

\end{proof}

\noindent As a consequence of Lemma \ref{lem:is module}, we see that (recall that $\BC^{\bn} = \prod_{i \in I} \BC^{n_i}/S_{n_i}$)
\begin{align}
&\CS_{<0|-\bn} \Big/ J(\bpsi)_{\bn} = \bigoplus_{\bx \in \BC^{\bn}} \left( \CS_{<0|-\bn} \Big/ J(\bpsi)_{\bn} \right)_{\bx} \label{eqn:eigenspace decomposition} \\
&\CS_{<0|-\bn} \Big/ \oJ(\bpsi)_{\bn} = \bigoplus_{\bx \in \BC^{\bn}} \left( \CS_{<0|-\bn} \Big/ \oJ(\bpsi)_{\bn} \right)_{\bx} \label{eqn:eigenspace decomposition variant}
\end{align} 
where the direct summands in the RHS corresponding to $\bx = (x_{ia})_{i \in I, 1 \leq a \leq n_i}$ are the generalized eigenspaces on which $z_{i1}^u+\dots+z^u_{in_i}$ acts by $x_{i1}^u+\dots+x_{in_i}^u$ for all $i \in I$ and all $u \geq 1$. By \eqref{eqn:k shuffle}, \eqref{eqn:p shuffle} and the formula
\begin{equation}
\label{eqn:power series}
q^{-d_{ij}} \exp \left(\sum_{u=1}^{\infty} \frac {x_{ia}^u(q^{-ud_{ij}} - q^{ud_{ij}})}{uz^u} \right) = \frac {z - x_{ia}q^{d_{ij}}}{zq^{d_{ij}}-x_{ia}}
\end{equation}
for all $i,j \in I$ and $a \in \{1,\dots,n_i\}$, we conclude that the $q$-character is given by \footnote{In the traditional notation \eqref{eqn:old a}, the $\ell$-weights that appear in the formulas below are $$\left[ \left( \prod_{i \in I} \prod_{a=1}^{n_i} \frac {z - x_{ia}q^{d_{ij}}}{zq^{d_{ij}}-x_{ia}} \right)_{j \in I} \right] = \prod_{i \in I} \prod_{a=1}^{n_i} A_{i,x_{ia}}^{-1}$$ Formulas \eqref{eqn:q-character general}-\eqref{eqn:q-character general variant} are proved by observing that the commuting operators $\kappa_j, p_{j,u}$ act on the generalized eigenspaces in the RHS of \eqref{eqn:eigenspace decomposition}-\eqref{eqn:eigenspace decomposition variant} with eigenvalue prescribed by \eqref{eqn:power series}.}
\begin{align}
&\chi_q(L(\bpsi)) = [\bpsi] \sum_{\bn \in \nn} \sum_{\bx \in \BC^\bn} \mu_{\bx}^{\bpsi}  \left[\left( \prod_{i \in I} \prod_{a=1}^{n_i} \frac {z - x_{ia}q^{d_{ij}}}{zq^{d_{ij}}-x_{ia}} \right)_{j \in I} \right] \label{eqn:q-character general} \\
&\chi_q(\oL(\bpsi)) = [\bpsi] \sum_{\bn \in \nn} \sum_{\bx \in \BC^\bn} \omu_{\bx}^{\bpsi}  \left[\left( \prod_{i \in I} \prod_{a=1}^{n_i} \frac {z - x_{ia}q^{d_{ij}}}{zq^{d_{ij}}-x_{ia}} \right)_{j \in I} \right] \label{eqn:q-character general variant}
\end{align}
(in the formulas above, the product of symbols $[\bpsi]$ is defined as in \eqref{eqn:product}), where
\begin{align}
&\mu_{\bx}^{\bpsi} =  \dim_{\BC} \left(\CS_{<0|-\bn} \Big/ J(\bpsi)_{\bn} \right)_{\bx} \label{eqn:q-character general multiplicity} \\
&\omu_{\bx}^{\bpsi} =  \dim_{\BC} \left(\CS_{<0|-\bn} \Big/ \oJ(\bpsi)_{\bn} \right)_{\bx} \label{eqn:q-character general variant multiplicity} 
\end{align}
The problem of calculating $q$-characters boils down to calculating the multiplicities $\mu_{\bx}^{\bpsi}$  and $\omu_{\bx}^{\bpsi}$ above. We will see that the latter multiplicities are easier to compute than the former, although for finite type $\fg$ they are equal due to \eqref{eqn:are equal}.
 
\medskip

\subsection{The first ideal}
\label{sub:ideal one}

Our next goal is to prove Theorem \ref{thm:main} on the modified $q$-characters $\chi_q(\oL(\bpsi))$, but before we do so, we will need to introduce some notation. 

\medskip

\begin{definition}
\label{def:ideal one}

For any $i_1,\dots,i_n \in I$ and any $x_1,\dots,x_n \in \BC^*$, we define
\begin{equation}
\label{eqn:ideal m}
\fm_{i_1,\dots,i_n|x_1,\dots,x_n}^{\bpsi} \subseteq \frac {\BC[z_1,\dots,z_n]}{\prod_{1\leq s < t \leq n, i_s \neq i_t} (z_s-z_t)}
\end{equation}
to be the set of $F(z_1,\dots,z_n)$ such that
\begin{equation}
\label{eqn:residue x}
\underset{z_n=x_n}{\emph{Res}} \dots \underset{z_1=x_1}{\emph{Res}}  \frac {F (z_1,\dots,z_n)G(z_1,\dots,z_n)}{\prod_{1\leq a < b \leq n} \zeta_{i_bi_a} \left(\frac {z_b}{z_a} \right)} \prod_{a=1}^n \frac {\psi_{i_a}(z_a)}{z_a} = 0
\end{equation}
for any $G \in \BC[z_1,\dots,z_n]$. 

\end{definition}

\noindent While the condition above may seem complicated, it boils down to $F$ lying in the kernel of a finite family of differential operators, followed by specialization at $(z_1,\dots,z_n) = (x_1,\dots,x_n)$. These differential operators may be recursively computed quite explicitly, but more important for us will be the following remarks:

\medskip

\begin{itemize}[leftmargin=*]

\item The quotient corresponding to the inclusion \eqref{eqn:ideal m} is a finite-dimensional $\BC[z_1,\dots,z_n]$-module supported at $(x_1,\dots,x_n) \in (\BC^*)^n$, because multiplying $F$ by a sufficiently high power of any $z_a-x_a$ would annihilate the residue \eqref{eqn:residue x}. 

\medskip

\item For fixed $i_1,\dots,i_n \in I$, there are only finitely many $x_1,\dots,x_n \in \BC^*$ for which the inclusion \eqref{eqn:ideal m} fails to be an equality. This is because in order to obtain a non-zero residue in \eqref{eqn:residue x}, each complex number $x_b$ must either be a pole of $\psi_{i_b}(z)$ or equal to $x_a q^{-d_{i_ai_b}}$ for some $a < b$. This phenomenon is the reason for the criterion in the last paragraph of Subsection \ref{sub:q-char intro}.

\end{itemize}

\medskip

\noindent Given an ordering $i_1,\dots,i_n$ of $\bn$ (see Definition \ref{def:i ordering}), we say that $x_1,\dots,x_n \in \BC^*$ is an ordering of 
$$
\bx = (x_{ia})_{i \in I, 1\leq a \leq n_i} \in (\BC^*)^{\bn} = \prod_{i \in I} (\BC^*)^{n_i}/S_{n_i}
$$
if the multisets $\{x_a | i_a = j\}$ and $\{x_{j1},\dots,x_{jn_j}\}$ are equal for all $j \in I$. Then we let
\begin{equation}
\label{eqn:ideal m sym}
\fm^{\bpsi}_{\bx} = \CV_{\bn}  \mathop{\bigcap_{i_1,\dots,i_n \text{ an ordering of }\bn}}_{x_1,\dots,x_n \text{ an ordering of }\bx} \fm_{i_1,\dots,i_n|x_1,\dots,x_n}^{\bpsi}
\end{equation}
In other words, $F \in \CV_{\bn}$ lies in $\fm^{\bpsi}_{\bx}$ if \eqref{eqn:residue x} vanishes for all orderings of $\bn$ and of $\bx$. By the first bullet above, the $\CP_{\bn}$-module $\CV_{\bn}/\fm^{\bpsi}_{\bx}$ is supported at $\bx \in (\BC^*)^{\bn}$.

\medskip

\begin{example}
\label{ex:4}

Let us consider $\fg = \fsl_2$ and $\psi(z) = \frac {zq^k-q^{-k}}{z-1}$ for $k \geq 0$, i.e. the highest $\ell$-weight of the dimension $k+1$ representation. The expression
$$
\underset{z_n=x_n}{\emph{Res}} \dots \underset{z_1=x_1}{\emph{Res}} F (z_1,\dots,z_n)G(z_1,\dots,z_n) \prod_{1\leq a < b \leq n} \frac {z_b-z_a}{z_b-z_aq^{-2}} \prod_{a=1}^n \frac {z_aq^k-q^{-k}}{z_a(z_a-1)}
$$
has a unique pole at $z_1=1$, so the residue can be non-zero only if $x_1=1$. The factor $z_2-z_1$ in the numerator annihilates the apparent pole at $z_2 = 1$, so there is a unique pole at $z_2 = q^{-2}$ and so the residue can be non-zero only if $x_2 = q^{-2}$. Repeating this argument forces $x_3 = q^{-4}$, $x_4=q^{-6}$, $...$ in order to have a non-zero residue, but as soon as we encounter the variable $z_{k+1}$ all poles are gone because of $z_{k+1}q^k - q^{-k}$ in the numerator. We conclude
\begin{equation}
\label{eqn:formula ex 4}
\fm_{x_1,\dots,x_n}^\psi = \begin{cases} (z_1-x_1,\dots,z_n-x_n) &\text{if } n \leq k \text{ and } x_a = q^{-2(a-1)} \\ \BC[z_1,\dots,z_n] &\text{otherwise} \end{cases}
\end{equation}

\end{example}

\medskip

\subsection{The second ideal}
\label{sub:ideal two}

Having dealt with non-zero complex numbers $x_1,\dots,x_n$ in the previous Subsection, we will now deal with the completely opposite case. 

\medskip

\begin{definition}
\label{def:ideal two}

For any $i_1,\dots,i_n \in I$ and any $\br = (r_i)_{i \in I} \in \zz$, we define
\begin{equation}
\label{eqn:ideal n}
\fn_{i_1,\dots,i_n}^{\br} \subseteq \frac {\BC[z_1,\dots,z_n]}{\prod_{1\leq s < t \leq n, i_s \neq i_t} (z_s-z_t)}
\end{equation}
to be the set of $F(z_1,\dots,z_n)$ such that
\begin{equation}
\label{eqn:residue 0}
\underset{z_n=0}{\emph{Res}} \dots \underset{z_1=0}{\emph{Res}} \frac {F (z_1,\dots,z_n)G(z_1,\dots,z_n)}{\prod_{1\leq a < b \leq n} \zeta_{i_bi_a} \left(\frac {z_b}{z_a} \right)} \prod_{a=1}^n \frac {\psi_{i_a}(z_a)}{z_a}  = 0
\end{equation}
for any $G \in \BC[z_1,\dots,z_n]$, where $\psi_i(z)$ is any power series in $z^{-r_i} (\BC^* + z\BC[[z]])$.

\end{definition}

\noindent Let us explain why the set of $F$ satisfying the condition \eqref{eqn:residue 0} depends only on $\br$ and not on $\bpsi$, which is implicit in Definition \ref{def:ideal two} (compare with Subsection \ref{sub:variant}). The vanishing of the residue at $z_1 = \dots = z_n = 0$ above is equivalent to the fact that the power series expansion (as $|z_1| \ll \dots \ll |z_n| \ll 1$) of the rational function
$$
\frac {F (z_1,\dots,z_n)}{\prod_{1\leq a < b \leq n} \zeta_{i_bi_a} \left(\frac {z_b}{z_a} \right)} \prod_{a=1}^n \psi_{i_a}(z_a)  
$$
does not contain any monomials $z_1^{k_1} \dots z_n^{k_n}$ with $k_1,\dots,k_n \leq 0$. However, if this condition is to be violated for some $(k_1,\dots,k_n)$, let us assume $k_1+\dots+k_n$ is minimal with respect to this property. Thus, any violation can be considered to arise from the lowest order terms $z_a^{-r_{i_a}}$ of the power series $\psi_{i_a}(z_a)$, hence
\begin{equation}
\label{eqn:ideal zero}
F \in \fn_{i_1,\dots,i_n}^{\br} \quad \Leftrightarrow \quad \left[ \frac {F (z_1,\dots,z_n) \prod_{a=1}^n z_a^{-r_{i_a}}}{\prod_{1\leq a < b \leq n} \zeta_{i_bi_a} \left(\frac {z_b}{z_a} \right)} \right]_{|z_1| \ll \dots \ll |z_n|}
\end{equation}
does not contain any monomials $z_1^{k_1} \dots z_n^{k_n}$ with $k_1,\dots,k_n \leq 0$. Having shown that indeed the set \eqref{eqn:ideal n} only depends on $\br$ and not on $\bpsi$, let us define
\begin{equation}
\label{eqn:ideal n sym}
\fn^{\br}_{\bn} = \CV_{\bn} \bigcap_{i_1,\dots,i_n \text{ an ordering of }\bn} \fn_{i_1,\dots,i_n}^\br
\end{equation}
In other words, $F \in \CV_{\bn}$ lies in $\fn^{\br}_{\bn}$ if \eqref{eqn:residue 0} vanishes for all orderings of $\bn$.
It is clear that the $\CP_{\bn}$-module $\CV_{\bn}/\fn_{\bn}^{\br}$ is supported at $\b0_{\bn} = (0,\dots,0)$. 

\medskip

\begin{example}
\label{ex:5}

When $\fg = \fsl_2$, by a logic akin to Example \ref{ex:3}, we have
\begin{equation}
\label{eqn:formula ex 5}
\fn^r = \begin{cases} z_1 \dots z_n \BC[z_1,\dots,z_n] &\text{if }r < 0 \\ (z_1 \dots z_n)^{r+1} \BC[z_1,\dots,z_n] &\text{if }r \geq 0 \end{cases}
\end{equation}

\end{example}

\medskip

\subsection{The main theorem}
\label{sub:main theorem}

We are now poised to prove Theorem \ref{thm:main} on the modified $q$-characters \eqref{eqn:q-character general variant} for any rational $\ell$-weight $\bpsi$. Recall that $\ord \bpsi$ denotes the $I$-tuple of the orders of the poles at $z=0$ of the rational functions $\psi_i(z)$. We will prove that the multiplicities in formula \eqref{eqn:q-character general variant} satisfy 
\begin{equation}
\label{eqn:factor}
\omu_{\bx}^{\bpsi} = \omu_{\by}^{\bpsi} \onu^{\ord\bpsi}_{\bn - \bm}
\end{equation}
for any $\bx = (\by,\b0_{\bn-\bm})$ with $\by \in (\BC^*)^{\bm}$, where
\begin{equation}
\label{eqn:multiplicity mu}
\omu^{\bpsi}_{\by} = \dim_{\BC} \left(\CS_{<0|-\bm} \Big/ \CS_{<0|-\bm} \cap \fm_{\by}^{\bpsi} \right)
\end{equation}
and
\begin{equation}
\label{eqn:multiplicity nu}
\onu^{\br}_{\bp} = \dim_{\BC} \left(\CS_{<0|-\bp} \Big/ \CS_{<0|-\bp} \cap \fn_{\bp}^{\br} \right) \ 
\end{equation}
for all $\bp \in \nn$ and $\br \in \zz$. 

\medskip

\begin{example}
\label{ex:6}

Consider $\fg = \fsl_2$ and $\psi(z) = \frac {zq^k-q^{-k}}{z-1}$ for $k \geq 0$ (note that $\oJ(\psi) = J(\psi)$ due to Proposition \ref{prop:non-negative}). For any $\by \in (\BC^*)^{(m)}$, \eqref{eqn:formula ex 4}  implies 
\begin{equation}
\label{eqn:formula ex 6}
\omu_{\by}^\psi = \begin{cases} 1 &\text{if } m \leq k \text{ and } \by = (1,\dots,q^{-2(m-1)}) \\ 0 &\text{otherwise} \end{cases}
\end{equation}
and \eqref{eqn:formula ex 5} implies that $\onu^0_p = \delta_{p0}$. Plugging the above formulas into \eqref{eqn:q-character general} or \eqref{eqn:q-character general variant} recovers the well-known formula for the $q$-character of $\oL(\psi) = L(\psi)$:
\begin{equation}
\label{eqn:formula ex 6 bis}
\chi_q(L(\psi)) = [\psi] \sum_{n=0}^k A_1^{-1} A_{q^{-2}}^{-1} \dots A_{q^{-2(n-1)}}^{-1}
\end{equation}

\end{example}

\medskip

\begin{proof} \emph{of Theorem \ref{thm:main}:} By analogy with \eqref{eqn:the computation}, for any $F \in \CS_{<0|-\bn}$ we have
\begin{equation}
\label{eqn:the equivalence}
F \in \mathring{J}(\bpsi)_{\bn} \quad \Leftrightarrow \quad 0 = \int_{1 \ll |z_1| \ll \dots \ll |z_n|} \frac {z_1^{d_1} \dots z_n^{d_n} F (z_1,\dots,z_n)}{\prod_{1\leq a < b \leq n} \zeta_{i_bi_a} \left(\frac {z_b}{z_a} \right)} \prod_{a=1}^n \psi_{i_a}(z_a)
\end{equation}
for all orderings $i_1,\dots,i_n$ of $\bn$ and for all $d_1,\dots,d_n \geq 0$. By moving the contours of integration toward 0, the integral above picks up two kinds of residues: those at non-zero complex numbers $x \in \BC^*$ and those at 0. Putting this together, we have
\begin{equation}
\label{eqn:residue}
F \in \oJ(\bpsi)_{\bn} \quad \Leftrightarrow \quad 0 = \sum_{\{1,\dots,n\} = \{s_1<\dots<s_m\} \sqcup \{t_1< \dots < t_{n-m}\}}
\end{equation}
$$
\sum_{y_1,\dots,y_m \in \BC^*} \underset{z_{s_m}=y_m}{\text{Res}} \dots \underset{z_{s_1}=y_1}{\text{Res}} \underset{z_{t_{n-m}}=0}{\text{Res}} \dots \underset{z_{t_{1}}=0}{\text{Res}}   \frac {z_1^{d_1} \dots z_n^{d_n} F (z_1,\dots,z_n)}{\prod_{1\leq a < b \leq n} \zeta_{i_bi_a} \left(\frac {z_b}{z_a} \right)} \prod_{a=1}^n \frac {\psi_{i_a}(z_a)}{z_a}
$$
for all orderings $i_1,\dots,i_n$ of $\bn$ and all $d_1,\dots,d_n \geq 0$ (the denominator $\frac 1{z_a}$ in the RHS is there due to our convention on integrals in Subsection \ref{sub:pairing shuffle}).

\medskip

\begin{claim}
\label{claim:residue}

The vanishing condition in the right-hand side of \eqref{eqn:residue} holds for all orderings $i_1,\dots,i_n$ of $\bn$ and all $d_1,\dots,d_n \geq 0$ if and only if for any partition
$$
\{1,\dots,n\} = \{s_1 < \dots < s_m\} \sqcup \{t_1 < \dots < t_{n-m} \}
$$
and any $y_1,\dots,y_m \in \BC^*$, the partial power series expansion 
\begin{equation}
\label{eqn:expansion}
\Big[ F(z_1,\dots,z_n) \Big]_{|z_{s_1}| ,\dots,|z_{s_m}| \gg |z_{t_1}|, \dots , |z_{t_{n-m}}|}
\end{equation}
is an infinite sum of terms in
\begin{equation}
\label{eqn:tensor ideal}
\fm^{\bpsi}_{i_{s_1},\dots,i_{s_m} | y_1,\dots,y_m} \boxtimes \fn_{i_{t_1},\dots,i_{t_{n-m}}}^{\eord \bpsi}
\end{equation}
where 
\begin{align}
&\fm^{\bpsi}_{i_{s_1},\dots,i_{s_m}|y_1,\dots,y_m} \subseteq \frac {\BC[z_{s_1},\dots,z_{s_m}]}{\prod_{1 \leq u < v \leq m, i_{s_u} \neq i_{s_v}} (z_{s_u} - z_{s_v})} \label{eqn:ideal m mod}\\
&\fn_{i_{t_1},\dots,i_{t_{n-m}}}^{\eord \bpsi} \subseteq \frac {\BC[z_{t_1},\dots,z_{t_{n-m}}]}{\prod_{1 \leq u < v \leq n-m, i_{t_u} \neq i_{t_v}} (z_{t_u} - z_{t_v})} \label{eqn:ideal n mod}
\end{align}
are the sets \eqref{eqn:ideal m} and \eqref{eqn:ideal n}, respectively, with the obvious changes of variables.

\end{claim}

\noindent In \eqref{eqn:expansion}, we expand all the denominators $z_{s_a} - z_{t_b}$, but we leave the denominators $z_{s_a}-z_{s_b}$ and $z_{t_a}-z_{t_b}$ untouched. The terms of the expansion have progressively lower degree in the variables $z_{s_1},\dots,z_{s_m}$ and higher degree in $z_{t_1}, \dots,z_{t_{n-m}}$. Since a term of high enough degree in $z_{t_1}, \dots,z_{t_{n-m}}$ will automatically land in the set $\fn$ from \eqref{eqn:ideal n mod}, the condition that the expansion \eqref{eqn:expansion} is an infinite sum of tensors \eqref{eqn:tensor ideal} actually boils down to checking finitely many degrees. Moreover, this condition is equivalent to the expansion
\begin{equation}
\label{eqn:expansion alternative}
\left[ \frac {F(z_1,\dots,z_n)}{\prod_{s_a < t_b} \zeta_{i_{t_b}i_{s_a}} \left(\frac {z_{t_b}}{z_{s_a}} \right) \prod_{t_b < s_a} \zeta_{i_{s_a}i_{t_b}} \left(\frac {z_{s_a}}{z_{t_b}} \right)} \right]_{|z_{s_1}| ,\dots,|z_{s_m}| \gg |z_{t_1}|, \dots , |z_{t_{n-m}}|} 
\end{equation}
lying in \eqref{eqn:tensor ideal}, because the functions $\zeta_{ij}(x)$ are regular and non-zero at 0 and $\infty$.

\medskip

\noindent Let us prove Claim \ref{claim:residue}. The ``if" statement is obvious, because it implies that the sum of residues in \eqref{eqn:residue} vanishes termwise. For the ``only if" statement, let us observe that the rational function
$$
\frac {z_1^{d_1} \dots z_n^{d_n} F (z_1,\dots,z_n)}{\prod_{1\leq a < b \leq n} \zeta_{i_bi_a} \left(\frac {z_b}{z_a} \right)} \prod_{a=1}^n \frac {\psi_{i_a}(z_a)}{z_a}
$$
has finitely many poles $(x_{1,v}, \dots, x_{n,v})$, where $v$ runs from 1 to some $p$. Fix any $u \in \{1,\dots,p\}$. Since we are allowed to choose the integers $d_1,\dots,d_n \geq 0$ arbitrarily, then the vanishing of the sum in \eqref{eqn:residue} implies the vanishing of the same sum with
$$
F(z_1,\dots,z_n) \quad \text{replaced by} \quad F(z_1,\dots,z_n) \prod_{s=1}^n  \prod_{\alpha \in \{x_{s,1},\dots,x_{s,p}\} \backslash \{x_{s,u}\}} (z_s - \alpha)^N
$$
$\forall N \in \BN$. If $N$ is chosen high enough, one of the factors $(z_s - x_{s,v})^N$ will cancel out the pole $(x_{1,v}, \dots, x_{n,v})$, for any $v \neq u$. Therefore, \eqref{eqn:residue} implies the vanishing of the residue at the $n$-tuple $(x_{1,u}, \dots, x_{n,u})$. If we separate the 0 coordinates from the non-zero coordinates of the aforementioned $n$-tuple, this is equivalent to 
$$
\underset{z_{s_m}=y_m}{\text{Res}} \dots \underset{z_{s_1}=y_1}{\text{Res}} \underset{z_{t_{n-m}}=0}{\text{Res}} \dots \underset{z_{t_{1}}=0}{\text{Res}}   \frac {z_1^{d_1} \dots z_n^{d_n} F (z_1,\dots,z_n) (\text{linear factors})}{\prod_{1\leq a < b \leq n} \zeta_{i_bi_a} \left(\frac {z_b}{z_a} \right)} \prod_{a=1}^n \frac {\psi_{i_a}(z_a)}{z_a} = 0
$$
for some fixed $\{1,\dots,n\} = \{s_1<\dots<s_m\} \sqcup \{t_1< \dots < t_{n-m}\}$, fixed $y_1,\dots,y_m \in \mathbb{C}^*$ and all $d_1,\dots,d_n\geq 0$. The ``linear factors" in the numerator do not vanish at the iterated residue in the formula above, so they may simply be ignored. Recalling the definition of the ideals \eqref{eqn:ideal m} and \eqref{eqn:ideal n}, the condition above is precisely equivalent to the fact that the expansion \eqref{eqn:expansion alternative} belongs to \eqref{eqn:tensor ideal}. As we already mentioned, this is equivalent to \eqref{eqn:expansion} lying in \eqref{eqn:tensor ideal}, so this concludes the proof of Claim \ref{claim:residue}.

\medskip

\noindent Consider any $F \in \CS_{<0|-\bn}$. For any $\bx = (\by,\b0_{\bn-\bm}) \in \BC^{\bn}$ with $\by \in (\BC^*)^{\bm}$, the fact that the power series expansion \eqref{eqn:expansion} lies in \eqref{eqn:tensor ideal} for all orderings $i_{s_1}, \dots, i_{s_m}$ of $\bm$, $i_{t_1},\dots,i_{t_{n-m}}$ of $\bn - \bm$ and $y_1,\dots,y_m$ of $\by$ boils down to
\begin{equation}
\label{eqn:tensor ideal symmetric}
\Big[ F(z_{i1},\dots,z_{in_i}) \Big]_{|z_{i1}|,\dots,|z_{im_i}| \gg |z_{i,m_i+1}|,\dots,|z_{in_i}|} \in \fm^{\bpsi}_{\by} \boxtimes \fn^{\ord \bpsi}_{\bn-\bm}
\end{equation}
with the sets $\fm^{\bpsi}_{\by}$ and $\fn^{\ord \bpsi}_{\bn-\bm}$ of \eqref{eqn:ideal m sym} and \eqref{eqn:ideal n sym} being understood to consist of functions in the variables $z_{i1},\dots,z_{im_i}$ and $z_{i,m_i+1},\dots,z_{in_i}$, respectively. Therefore, Claim \ref{claim:residue} implies that
$$
\oJ(\bpsi)_{\bn} = \bigcap_{\bm = \b0}^{\bn} \bigcap_{\by \in (\BC^*)^{\bm}} \Big(F \in \CS_{<0|-\bn} \text{ s.t. \eqref{eqn:tensor ideal symmetric} holds} \Big)
$$
By taking the appropriate quotient, we infer that
$$
\CS_{<0|-\bn} \Big/ \mathring{J}(\bpsi)_{\bn} = \CS_{<0|-\bn} \Big/ \bigcap_{\bm = \b0}^{\bn} \bigcap_{\by \in (\BC^*)^{\bm}} \Big(F \in \CS_{<0|-\bn} \text{ s.t. \eqref{eqn:tensor ideal symmetric} holds} \Big) = 
$$
\begin{equation}
\label{eqn:quotients quotients}
= \bigoplus_{\bm = \b0}^{\bn} \bigoplus_{\by \in (\BC^*)^{\bm}} \CS_{<0|-\bn} \Big/ \Big(F \in \CS_{<0|-\bn} \text{ s.t. \eqref{eqn:tensor ideal symmetric} holds} \Big)
\end{equation}
The last equality holds because for any $\bm \leq \bn$ and $\by \in (\BC^*)^{\bm}$, the set $\{F \text{ satisfying \eqref{eqn:tensor ideal symmetric}}\}$ contains any function which vanishes to high enough order at $(\{z_{i1},\dots,z_{in_i}\} = \{y_{i1},\dots,y_{im_i},0,\dots,0\})_{i \in I}$, and so the direct summands in \eqref{eqn:quotients quotients} are finite-dimensional $\CP_{\bn}$-modules supported at different points $(\by,\b0_{\bn-\bm}) \in \BC^{\bn}$. To prove \eqref{eqn:factor}, thus concluding the proof of Theorem \ref{thm:main}, we must show the following two claims for any $\bm \leq \bn$, $\by \in (\BC^*)^{\bm}$:
$$
\CS_{<0|-\bn} \Big/ \Big(F \in \CS_{<0|-\bn} \text{ s.t. \eqref{eqn:tensor ideal symmetric} holds} \Big) \xrightarrow{\text{expansion as }|z_{i1}|,\dots,|z_{im_i}| \gg |z_{i,m_i+1}|,\dots,|z_{in_i}|}
$$
\begin{equation}
\label{eqn:claim 1}
\Big( \CS_{-\bm} \Big/ \CS_{-\bm} \cap \fm_{\by}^{\bpsi} \Big) \boxtimes \Big( \CS_{<0|-\bn + \bm} \Big/ \CS_{<0|-\bn + \bm} \cap \fn_{\bn - \bm}^{\ord \bpsi} \Big)
\end{equation}
is an isomorphism, and the natural inclusion map
\begin{equation}
\label{eqn:claim 2}
 \Big( \CS_{<0|-\bm} \Big/ \CS_{<0|-\bm} \cap \fm_{\by}^{\bpsi} \Big) \rightarrow \Big( \CS_{-\bm} \Big/ \CS_{-\bm} \cap \fm_{\by}^{\bpsi} \Big)
\end{equation}
is an isomorphism. Indeed, the latter claim implies that the dimensions of the left and right-hand sides of \eqref{eqn:claim 1} are precisely the left and right-hand sides of \eqref{eqn:factor}.

\medskip

\noindent First of all, let us explain why the expansion map in \eqref{eqn:claim 1} is well-defined: as we already mentioned, the expansion of any $F(z_{i1},\dots,z_{in_i})$ as $|z_{i1}|,\dots,|z_{im_i}| \gg |z_{i,m_i+1}|,\dots,|z_{in_i}|$ is given by terms of progressively lower degree in $\{z_{ia}\}_{a \leq m_i}$ and progressively higher degree in $\{z_{jb}\}_{b > m_j}$. As the degree in the latter variables becomes sufficiently large, the corresponding terms will automatically land in the set $\fn_{\bn - \bm}^{\ord \bpsi}$. Therefore, while the expansion in \eqref{eqn:claim 1} consists of infinitely many terms, all but finitely many of them have second tensor factor 0. Moreover, any term in the expansion has the property that the $\{z_{jb}\}_{b > m_j}$ variables determine a shuffle element of slope $<0$, which is simply a restatement of \eqref{eqn:slope 2}. 

\medskip

\noindent It is clear that both arrows in \eqref{eqn:claim 1} and \eqref{eqn:claim 2} are injective. To prove that \eqref{eqn:claim 2} is surjective, let us take a shuffle element $G\in \CS_{-\bm}$. For $M$ large enough, we have by \eqref{eqn:factorization 1} and the second statement of Proposition \ref{prop:easy}
$$
G(z_{i1},\dots,z_{im_i}) \prod_{i \in I} \prod_{a=1}^{m_i} z_{ia}^M \in \CS_{<0|-\bm}
$$
However, for $N$ large enough, we also have
\begin{equation}
\label{eqn:big power}
\left( \prod_{i \in I} \prod_{a=1}^{m_i} z_{ia}^M - \prod_{i \in I} \prod_{a=1}^{m_i} y_{ia}^M \right)^N G(z_{i1},\dots,z_{im_i})  \in \fm_{\by}^{\bpsi}
\end{equation}
because the set $\fm_{\by}^{\bpsi}$ is associated to the maximal ideal corresponding to the point $(\{z_{i1},\dots,z_{im_i}\} = \{y_{i1},\dots,y_{im_i}\})_{i \in I}$. Since $y_{ia} \neq 0$ for all $i,a$, the formulas above prove that $G$ is equal to a linear combination of elements in $\CS_{<0|-\bm}$ modulo the ideal $\fm_{\by}^{\bpsi}$. This precisely establishes the surjectivity of \eqref{eqn:claim 2}.

\medskip

\noindent Let us now prove that \eqref{eqn:claim 1} is surjective. It suffices to show that for all homogeneous $G \in \CS_{-\bm}$ and $H \in \CS_{<0|-\bn+\bm}$, there exists some $F \in \CS_{<0|-\bn}$ such that
$$
\Big[ F(z_{i1},\dots,z_{in_i}) \Big]_{|z_{ia}|_{a\leq m_i} \gg |z_{jb}|_{b > m_j}} = G(z_{i1},\dots,z_{im_i}) \boxtimes H(z_{i,m_i+1},\dots,z_{in_i}) + \dots
$$
where the ellipsis denotes terms with second tensor factor of homogeneous degree higher than that of $H$, or terms where either the first or the second tensor factors lie in $\fm_{\by}^{\bpsi}$ or $\fn_{\bn-\bm}^{\ord \bpsi}$, respectively. For large enough natural numbers $M,N$, we choose
$$
F = \left\{ \left[1 - \left(1 - \prod_{i \in I} \prod_{a=1}^{m_i} \frac {z_{ia}^M}{y_{ia}^M} \right)^N \right] G \right\} * H 
$$
If $M$ is large enough, the shuffle element in the curly brackets is $G$ times a multiple of $\prod_{i,a} z_{ia}^M$, and thus has slope $<0$ by the second statement of Proposition \ref{prop:easy}. Therefore, $F$ as defined above also has slope $<0$. As we expand the shuffle product $\{ \dots  G\} * H$ when $\bm$ of the variables are much larger than the other $\bn - \bm$ variables, the summand in the shuffle product which has the lowest possible degree in the latter variables is 
\begin{multline*}
\left\{ \left[1 - \left(1 - \prod_{i \in I} \prod_{a=1}^{m_i} \frac {z_{ia}^M}{y_{ia}^M} \right)^N \right] G (\{z_{ia}\}_{a \leq m_i}) \right\} \boxtimes H (\{z_{jb}\}_{b > m_j}) \stackrel{\eqref{eqn:big power}}\equiv \\ \equiv (G (\{z_{ia}\}_{a \leq m_i}) \text{ mod }\fm_{\by}^{\bpsi}) \boxtimes H (\{z_{jb}\}_{b > m_j})
\end{multline*}
This is because if one of the ``small" variables is chosen among the variables of $G$, then it will come with an exponent $\geq M$, and the number $M$ can be chosen to be much larger than the homogeneous degree of $H$.

\end{proof}

\medskip

\noindent We note that formulas \eqref{eqn:multiplicity mu} and \eqref{eqn:multiplicity nu} are much easier to compute than \eqref{eqn:q-character general variant multiplicity}, as they remove $\oJ(\bpsi)_{\bn}$ from the picture and only ask to intersect shuffle algebras with the combinatorially explicit sets $\fm^{\bpsi}_{\by}$ and $\fn^{\ord \bpsi}_{\bp}$. In particular, for $\fg$ of finite type they can in principle be calculated on a computer using the explicit description of the shuffle algebra in \eqref{eqn:e}. Together with \eqref{eqn:are equal}, this gives a shuffle algebra approach for the computation of $q$-characters of simple representations in category $\CO$ for quantum affine algebras.

\medskip

\subsection{$q$-characters for polynomial $\ell$-weights}
\label{sub:first corollary}

Formulas \eqref{eqn:intro q-character polynomial variant} and \eqref{eqn:intro q-character polynomial} say that the $q$-character of simple modules $L(\btau)$ (and their variants $\oL(\btau)$) associated to a polynomial $\ell$-weight $\btau$ encodes the same information as the usual character. By \eqref{eqn:q-character general} and \eqref{eqn:q-character general variant}, this boils down to the fact that the $\CP_{\bn}$-modules
\begin{equation}
\label{eqn:two modules}
\CS_{<0|-\bn} \Big/ J(\btau)_{\bn} \quad \text{and} \quad \CS_{<0|-\bn} \Big/ \oJ(\btau)_{\bn}
\end{equation}
are supported at the origin $\b0_{\bn} \in \BC^{\bn}$. In turn, this is due to the second bullet in Subsection \ref{sub:ideal one}, since a Laurent polynomial has no non-zero poles; alternatively, the modules \eqref{eqn:two modules} are finite-dimensional and graded by vertical degree (see Subsection \ref{sub:integral} for the module on the left, and Subsection \ref{sub:variant} for the module on the right). If we let
\begin{align} 
&\chi^{\br} = \sum_{\bn \in \nn} \dim_{\BC} \left( \CS_{<0|-\bn} \Big/J^{\br}_{\bn} \right) [-\bn] \label{eqn:chi r} \\
&\ochi^{\br} = \sum_{\bn \in \nn} \dim_{\BC} \left( \CS_{<0|-\bn} \Big/\oJ^{\br}_{\bn} \right) [-\bn] \label{eqn:chi r variant}
\end{align} 
for any $\br \in \zz$, then Proposition \ref{prop:integral} implies formula \eqref{eqn:intro q-character polynomial}, while its analogue in Subsection \ref{sub:variant} implies \eqref{eqn:intro q-character polynomial variant}. Note that the dimensions which appear in the RHS of \eqref{eqn:chi r variant} are equal to the numbers $\onu^{\br}_{\bn}$ of \eqref{eqn:multiplicity nu} for all $\bn \in \nn$ and all $\br \in \zz$, which is simply because the condition on $F$ in \eqref{eqn:j integral variant} is equivalent via Lemma \ref{lem:antipode pairing shuffle} to 
$$
\underset{z_n=0}{\text{Res}} \dots \underset{z_1=0}{\text{Res}} \frac {F (z_1,\dots,z_n)G(z_1,\dots,z_n)}{\prod_{1\leq a < b \leq n} \zeta_{i_bi_a} \left(\frac {z_b}{z_a} \right)} \prod_{a=1}^n z_a^{-r_{i_a}-1}  = 0
$$
for all orderings $i_1,\dots,i_n$ of $\bn$ and all polynomials $G$.

\medskip

\begin{proof} \emph{of Proposition \ref{prop:no pole}:} Since $\chi^{\br}$ measures the character of \eqref{eqn:l integral}, the fact that $\chi^{\br} = 1$ if $\br \in - \nn$ is an immediate consequence of the fact that the pairing in \eqref{eqn:j integral} is trivially 0 (as $\vdeg E \geq 0$, $-r_i \geq 0$ and $\vdeg F > 0$). $\ochi^{\br} = 1$ is analogous.

\end{proof}

\begin{proof} \emph{of Corollary \ref{cor:regular}:} Immediate from \eqref{eqn:intro factor series} and Proposition \ref{prop:no pole}.
	
\end{proof}

\medskip

\subsection{Refined characters}
\label{sub:refined characters}

Since the vector spaces in \eqref{eqn:two modules} are graded by vertical degree (as explained in Subsections \ref{sub:integral} and \ref{sub:variant}), we can therefore express the refined characters in \eqref{eqn:refined intro q-character polynomial} and \eqref{eqn:refined intro q-character polynomial variant} for any polynomial $\ell$-weight $\btau$ as
\begin{align}
	&\chi_{q}^{\text{ref}}(L(\btau)) = [\btau] \sum_{\bn \in \nn} \sum_{d=0}^{\infty} \dim_{\BC} \left(\CS_{<0|-\bn,d} \Big/ J(\btau)_{\bn,d} \right) [-\bn] v^d \label{eqn:refined} \\
	&\chi_{q}^{\text{ref}}(\oL(\btau)) = [\btau] \sum_{\bn \in \nn} \sum_{d=0}^{\infty} \dim_{\BC} \left(\CS_{<0|-\bn,d} \Big/ \oJ(\btau)_{\bn,d} \right) [-\bn] v^d \label{eqn:refined variant}
\end{align} 
We are now ready to prove Theorem \ref{thm:my refined}, which states that for $\fg$ of finite type,
\begin{equation}
	\label{eqn:my refined final}
	\chi_{q}^{\text{ref}}(L(\btau)) = [\btau] \prod_{\balpha \in \Delta^+} \prod_{d=1}^{\br \cdot \balpha} \frac 1{1-[-\balpha]v^d}
\end{equation}
for any polynomial $\ell$-weight $\btau$ such that $\br = \ord \btau \in \BZ_{>0}^I$. 

\medskip 

\begin{proof} \emph{of Theorem \ref{thm:my refined}:} Proposition \ref{prop:slope affine} and the decomposition \eqref{eqn:a geq} imply that
	$$
\Upsilon \Big( \Phi \left(\UUaffg\right) \cap \UUm \Big) = \CS^-_{<0}
	$$
	The well-known PBW basis for the algebra on the left (see \cite[Subsection 5.22]{NT}, based on \cite{B0}) gives us the following formula for its graded character
	\begin{equation}
		\label{eqn:pbw affine}
		\chi^{\text{ref}}(\binfty) := \sum_{\bn \in \nn} \sum_{d=0}^{\infty} \dim_{\BC}\left(\CS_{<0|-\bn,d} \right)[-\bn]v^d  = \prod_{\balpha \in \Delta^+} \prod_{d=1}^{\infty} \frac 1{1-[-\balpha]v^d}
	\end{equation}
	On the other hand, by Lemma \ref{lem:composition}, we have
$$
\frac {\chi_{q}^{\text{ref}}(L(\btau))}{[\btau]} =: \chi^{\text{ref}}(\br) = \sum_{\bn \in \nn} \sum_{d=0}^{\infty} \dim_{\BC} \left(\otimes^\leftarrow_{\mu \in [-1,0)} \CB_\mu^- \right)_{-\bn,d} [-\bn]v^d
$$
Proposition \ref{prop:easy} therefore implies that for all $k \in \BN$
	$$
	\sigma^k \left(\chi^{\text{ref}}(\br)\right) =\sum_{\bn \in \nn} \sum_{d=k(\br \cdot \bn)}^{\infty} \dim_{\BC} \left(\otimes^\leftarrow_{\mu \in [-k-1,-k)} \CB_\mu^- \right)_{-\bn,d} [-\bn]v^d
	$$
	where we make the convention that $\sigma([-\bn]v^d) = [-\bn]v^{d+\br\cdot \bn}$. Therefore, the factorization \eqref{eqn:factorization 4} implies that
	$$
	\chi^{\text{ref}}(\binfty) = \prod_{k=0}^{\infty} \sigma^k \left(\chi^{\text{ref}}(\br)\right)
	$$
We conclude that
	$$
	\chi^{\text{ref}}(\br) = \frac {\chi^{\text{ref}}(\binfty)}{\sigma(\chi^{\text{ref}}(\binfty))} \stackrel{\eqref{eqn:pbw affine}}= \frac {\prod_{\balpha \in \Delta^+} \prod_{d=1}^{\infty} (1-[-\balpha]v^d)^{-1}}{\prod_{\balpha \in \Delta^+} \prod_{d=1}^{\infty} (1-[-\balpha]v^{d+\br\cdot \balpha})^{-1}}
	$$
	which precisely implies \eqref{eqn:my refined final}. 
	
\end{proof} 

\begin{remark}
\label{rem:general}

The only place we needed $\fg$ to be of finite type is \eqref{eqn:pbw affine}, where we used the well-known PBW basis of quantum affine algebras to calculate the graded dimension of $\CS_{<0}^-$. For a general Kac-Moody Lie algebra $\fg$, we conjecture that
\begin{equation}
\label{eqn:conj}
\chi^{\emph{ref}}(\binfty) \stackrel{?}= \prod_{\bn \in \nn \backslash \b0} \prod_{d=1}^{\infty} \left( \frac 1{1-[-\bn] v^d} \right)^{a_{\fg,\bn}}
\end{equation}
where the non-negative integers $\{a_{\fg,\bn}\}_{\bn \in \nn \backslash \b0}$ are defined by the formula
\begin{equation}
\label{eqn:conj hor}
\sum_{\bn \in \nn} \dim_{\BC}(\CB_{0|\pm \bn})[-\bn] = \prod_{\bn \in \nn \backslash \b0} \left( \frac 1{1-[-\bn] } \right)^{a_{\fg,\bn}}
\end{equation}
(the formulas above are inspired by \cite{Dav, Mo, N R-matrix, OS}, which pertain to a related setting that includes that of simply laced $\fg$). Once formula \eqref{eqn:conj} would be established, the proof of Theorem \ref{thm:my refined} runs through and establishes the following analogue of \eqref{eqn:my refined final}
\begin{equation}
	\label{eqn:my refined final general}
\chi_{q}^{\emph{ref}}(L(\btau)) = [\btau] \prod_{\bn \in \nn \backslash \b0} \prod_{d=1}^{\br \cdot \bn} \left( \frac 1{1-[-\bn]v^d} \right)^{a_{\fg,\bn}}
\end{equation}

\end{remark}

\medskip

\subsection{The last corollary}
\label{sub:last corollary}

We will now prove Corollary \ref{cor:main}. Fix a regular $\ell$-weight $\bpsi$, and note that $\chi_q(\oL(\bpsi))$ only has contributions from $\bx \in (\BC^*)^{\bn}$, due to Corollary \ref{cor:regular}. 

\medskip

\begin{definition}
\label{def:monochrome}

We will say that a polynomial $\ell$-weight $\btau$ is $\bpsi$-\textbf{monochrome} if every $\bx \in (\BC^*)^{\bn}$ is either
\begin{align}
&(\bpsi,\btau)\text{-black, i.e. } \prod_{i \in I} \prod_{a = 1}^{n_i} \tau_i(x_{ia}) \neq 0, \text{ or} \label{eqn:black} \\
&(\bpsi,\btau)\text{-white, i.e. } (\emph{any element of }\CS_{-\bn})\prod_{i \in I} \prod_{a = 1}^{n_i} \tau_i(z_{ia}) \in \fm^{\bpsi}_{\bx} \label{eqn:white}
\end{align}

\end{definition}

\medskip

\noindent The meaning of the condition above is that multiplying a shuffle element $F \in \CS_{<0|-\bn}$ by $\prod_{i \in I} \prod_{a=1}^{n_i} \tau_i(z_{ia})$ either has no effect on the pole of $F$ at $\bx$ (in the case of option \eqref{eqn:black}) or it annihilates the pole completely (in the case of option \eqref{eqn:white}). If $\btau$ is $\bpsi$-monochrome, we define the truncated $q$-character of $\oL(\bpsi)$ as
\begin{equation}
\label{eqn:truncated q-character finite}
\chi^\btau_q(\oL(\bpsi)) =  [\bpsi] \sum_{\bn \in \nn} \mathop{\sum_{\bx \in (\BC^*)^{\bn}}}_{(\bpsi,\btau)\text{-black}} \omu^\bpsi_{\bx} \left[\left( \prod_{i \in I} \prod_{a=1}^{n_i} \frac {z - x_{ia}q^{d_{ij}}}{zq^{d_{ij}}-x_{ia}}\right)_{j \in I} \right] 
\end{equation}
In other words, we remove all the $(\bpsi,\btau)$-white $\bx$'s from \eqref{eqn:q-character general variant}. Our goal is to prove
\begin{equation}
\label{eqn:conjecture}
\chi_q(\oL(\bpsi \btau)) = \chi^{\btau}_q(\oL(\bpsi)) \cdot [\btau] \ochi^{\ord \bpsi \btau}
\end{equation}
under the hypothesis that $\btau$ is $\bpsi$-monochrome.

\medskip

\begin{example}
\label{ex:7}

Let us consider $\fg = \fsl_2$, $\psi(z) = \frac {zq^k - q^{-k}}{z-1}$, $\tau(z) = 1 - \frac {q^{-2k'}}z$ for $0 \leq k' < k$. Because of the description of $\fm_{x_1,\dots,x_n}^\psi$ in \eqref{eqn:formula ex 4}, we have
\begin{equation}
\label{eqn:formula ex 7 bis}
(x_1,\dots,x_n) \text{ is } \begin{cases} (\psi,\tau)\text{-black if } q^{-2k'} \notin \{x_1,\dots,x_n\} \\ (\psi,\tau)\text{-white if } q^{-2k'} \in \{x_1,\dots,x_n\} \end{cases}
\end{equation}
Thus, $\tau$ is $\psi$-monochrome. Since the non-trivial summands in $\chi_q(\oL(\psi))$ correspond to $(x_1,\dots,x_n) = (1,\dots,q^{-2(n-1)})$ as per \eqref{eqn:formula ex 6 bis}, we have
\begin{equation}
\label{eqn:formula ex 7 bis}
\chi_q^{\tau}(\oL(\psi)) = [\psi] \sum_{n = 0}^{k'} A_1^{-1} A_{q^{-2}}^{-1} \dots A_{q^{-2(n-1)}}^{-1}
\end{equation}

\end{example}

\medskip

\begin{proof} \emph{of Corollary \ref{cor:main}:} Consider any $\bx = (\by,\b0_{\bn-\bm}) \in \BC^{\bn}$ with $\by \in (\BC^*)^{\bm}$. By \eqref{eqn:factor}, we have
\begin{equation}
\label{eqn:factor proof}
\omu_{\bx}^{\bpsi \btau} = \omu_{\by}^{\bpsi \btau} \onu_{\bn-\bm}^{\ord \bpsi\btau}
\end{equation}
However, the assumption that $\btau$ is $\bpsi$-monochrome implies that
\begin{equation}
\label{eqn:cases one}
\CS_{<0|-\bm} \cap \fm_{\by}^{\bpsi \btau} = \begin{cases} \CS_{<0|-\bm} \cap \fm_{\by}^{\bpsi} &\text{if }\by \text{ is }(\bpsi,\btau)\text{-black} \\ \CS_{<0|-\bm} &\text{if }\by \text{ is }(\bpsi,\btau)\text{-white} \end{cases}
\end{equation}
(the first option is because multiplication by a Laurent polynomial which does not vanish at $\by$ has no bearing on the pole of a rational function at $\by$, and the second option is because the residue 
$$
\underset{z_m=y_m}{\text{Res}} \dots \underset{z_1=y_1}{\text{Res}}  \frac {\text{any element of }\CS_{-\bm}}{\prod_{1\leq a < b \leq m} \zeta_{i_bi_a} \left(\frac {z_b}{z_a} \right)} \prod_{a=1}^m \frac {\psi_{i_a}(z_a)\tau_{i_a}(z_a)}{z_a}
$$
is automatically 0 if condition \eqref{eqn:white} holds). Taking the codimensions of the left and right-hand sides of \eqref{eqn:cases one} in $\CS_{<0|-\bm}$ implies that (using the notation in \eqref{eqn:multiplicity mu})
\begin{equation}
\label{eqn:cases two}
\omu_{\by}^{\bpsi \btau} = \begin{cases} \omu_{\by}^{\bpsi} &\text{if }\by \text{ is }(\bpsi,\btau)\text{-black} \\ 0 &\text{if }\by \text{ is }(\bpsi,\btau)\text{-white} \end{cases}
\end{equation}
Plugging this formula into the right-hand side of \eqref{eqn:factor proof} and summing over all $\bx = (\by ,\b0_{\bn - \bm}) \in \BC^{\bn}$ gives us \eqref{eqn:conjecture}, as required.

\end{proof}

\subsection{An important consequence}
\label{sub:conjecture}

We will now show that as a special case, Corollary \ref{cor:main} implies \cite[Conjecture 7.15]{HL Borel}. To see this, let us recall the setup of the aforementioned conjecture: we let $\fg$ be of finite type, fix $R \in \BZ$ and pick
\begin{equation}
\label{eqn:hl psi}
\bpsi = \left( \psi_i(z) =  \prod_{r = -\infty}^{R+d_i} \left( \frac {z q^{d_i} - q^{r-d_i}}{z - q^r} \right)^{t_{i,r}} \right)_{ i \in I}
\end{equation}
where $d_i = \frac {d_{ii}}2$ and $t_{i,r} \geq 0$ are certain integers, almost all of which are 0. In particular, $\bpsi$ is regular (it actually corresponds to a finite-dimensional representation, see \eqref{eqn:finite dimensional psi})  and so its $q$-character \eqref{eqn:intro q-character simple} only involves $\bx = (x_{ia})_{i \in I, a \in \{1,\dots,n_i\}} \in (\BC^*)^{\bn}$. Moreover, as explained in the criterion in the last paragraph of Subsection \ref{sub:q-char intro}, we further have $x_{ia} \in q^{\BZ}, \forall i,a$. This allowed \loccit to define the truncated $q$-character
\begin{equation}
\label{eqn:truncated q-character}
\chi^R_q(L(\bpsi)) = [\bpsi] \sum_{\bn \in \nn} \sum_{\by = (y_{ia} \in q^{\leq R-d_i})_{i \in I, a \leq n_i}} \mu_{\by}^{\bpsi} \left[\left( \prod_{i \in I} \prod_{a=1}^{n_i} \frac {z - y_{ia}q^{d_{ij}}}{zq^{d_{ij}}-y_{ia}} \right)_{j \in I} \right] 
\end{equation}
by retaining only those summands in \eqref{eqn:intro q-character simple} whose coordinates satisfy the condition underneath the summation sign. Consider any polynomial $\ell$-weight $\btau$ such that
\begin{align}
&\tau_i(z) \text{ is divisible by } \prod_{r = R-d_i+1}^{R+d_i} \left(1 - \frac {q^r}z \right)^{u^{\by}_{i,r}}, \ \forall i \in I, \forall \by \text{ as in \eqref{eqn:truncated q-character}} \label{eqn:tau particular 1} \\
&\tau_i(z) \text{ is not divisible by } \left(1 - \frac {q^r}z \right), \ \forall i \in I, r \leq R-d_i \label{eqn:tau particular 2}
\end{align}
where the collection of non-negative integers $\{u^{\by}_{i,r}\}_{i \in I, r \in (R-d_i,R+d_i]}$ is defined for every summand $\by$ in \eqref{eqn:truncated q-character} which has non-zero multiplicity, as follows
\begin{multline}
\label{eqn:exponent u}
u^{\by}_{i,r} = t_{i,r}  - \left| \Big\{ b \in \{1,\dots,n_i\} \Big| y_{ib} = q^{r-2d_i} \Big\} \right| \\ + \sum_{I \ni j \neq i}  \left| \Big\{ a \in \{1,\dots,n_j\} \Big| y_{ja} = q^{r+d_{ij}} \Big\} \right| 
\end{multline}
Thus, the polynomial $\ell$-weight $\btau$ must satisfy one condition \eqref{eqn:tau particular 1} for every non-trivial summand in \eqref{eqn:truncated q-character}. In particular, for the summand corresponding to $\bn = \b0$ we have $u_{i,r} = t_{i,r}$ for all $i \in I$ and $r \in (R-d_i,R+d_i]$. Under these circumstances, Conjecture 7.15 of \cite{HL Borel} (in the formulation of Remark 7.16 of \emph{loc. cit.}) states that
\begin{multline}
\label{eqn:hl conjecture}
\chi_q(L(\bpsi \btau)) = \chi_q^R(L(\bpsi)) \chi_q(L(\btau)) \\ \stackrel{\eqref{eqn:intro q-character polynomial}}= \chi_q^R(L(\bpsi)) [\btau] \chi^{\ord \btau} \stackrel{\eqref{eqn:hl psi}}= \chi_q^R(L(\bpsi)) [\btau] \chi^{\ord \bpsi \btau}
\end{multline}
(note that \loccit state the above conjecture for a specific $\btau$ that depends on $R$, called $\boldsymbol{\Psi}_R$ therein, but the proof below holds for any $\btau$ satisfying \eqref{eqn:tau particular 1}-\eqref{eqn:tau particular 2}). However, the formula above is precisely the same as \eqref{eqn:intro conjecture} (recall that for $\fg$ of finite type, we have $L(\bpsi) = \oL(\bpsi)$, $\forall \bpsi$) once we prove the following.

\medskip

\begin{claim}
\label{claim:conj}

Under the above circumstances, we have $\forall \bx = (x_{ia} \in q^{\BZ})_{i \in I, a \leq n_i}$
$$
\bx \text{ is }(\bpsi,\btau)\text{-black} \quad \Leftrightarrow \quad x_{ia} \in q^{\BZ_{\leq R-d_i}}, \forall i,a
$$
and $\bx$ is $(\bpsi,\btau)$-white otherwise (thus implying that $\btau$ is $\bpsi$-monochrome).

\end{claim}

\medskip

\begin{proof} The implication $\Leftarrow$ is obvious, because if all coordinates of some $\bx$ satisfy $x_{ia} \in q^{\BZ_{\leq R-d_i}}$, then $\tau_i(x_{ia}) \neq 0$ due to \eqref{eqn:tau particular 2}. For the implication $\Rightarrow$, we will prove the contrapositive: assume that $\bx$ is such that 
$x_{ia} \in q^{\BZ_{> R-d_i}}$ for some $i \in I$ and some $a \in \{1,\dots,n_i\}$. We wish to prove that in this case, the residue
\begin{equation}
\label{eqn:residue ultimate} 
\underset{z_n=x_n}{\text{Res}} \dots \underset{z_1=x_1}{\text{Res}}  \frac {\text{any element of }\CS_{-\bn}}{\prod_{1\leq a < b \leq n} \zeta_{i_bi_a} \left(\frac {z_b}{z_a} \right)} \prod_{a=1}^n \frac {\psi_{i_a}(z_a) \tau_{i_a}(z_a)}{z_a} 
\end{equation}
is 0 (with respect to any given orderings $i_1,\dots,i_n$ of $\bn$ and $x_1,\dots,x_n$ of $\bx$), because that would imply that $\bx$ is $(\bpsi,\btau)$-white. Our assumption is that there exists $b \in \{1,\dots,n\}$ such that
\begin{equation}
\label{eqn:assumption}
x_b = q^r \text{ with }r > R-d_{i_b}
\end{equation}
and let us take $b$ minimal with this property. Recall that the residue \eqref{eqn:residue ultimate} can be non-zero only if $x_b$ is either equal to a pole of $\psi_{i_b} (z) \tau_{i_b} (z)$ or equal to $x_a q^{-d_{i_ai_b}}$ for some $a < b$. The first option cannot hold because of assumption \eqref{eqn:tau particular 1} with $\by \in (\BC^*)^{\b0}$, in which case we already mentioned that $u_{i,r}^{\by} = t_{i,r}$ and thus $\tau_{i_b}(z)$ cancels out the pole at $z = q^r$ of $\psi_{i_b}(z)$. As for the second option, it implies the existence of a collection of indices
\begin{equation}
\label{eqn:collection of indices}
a = s_0 < s_1 < \dots < s_{m-1} < s_m = b
\end{equation}
such that $x_a$ is a pole of $\psi_{i_a} (z) \tau_{i_a} (z)$ and 
\begin{equation}
\label{eqn:k<m}
x_{s_k} = x_{s_{k-1}} q^{-d_{i_{s_{k-1}}i_{s_k}}}, \quad \forall k \in \{1,\dots,m\}
\end{equation}
We will assume that the collection of indices \eqref{eqn:collection of indices} is minimal with the properties above, and let $\by = (x_a,x_{s_1},\dots,s_{m-1})$. We will now estimate the order of vanishing of the rational function in \eqref{eqn:residue ultimate} at $z_b = x_b$: for every $k$ satisfying \eqref{eqn:k<m}, we have a contribution of $-1$ from the $\zeta$ factors in the denominators of \eqref{eqn:residue ultimate}, which precisely balances out a contribution of $+1$ from the term on the second line of \eqref{eqn:exponent u}. However, the latter contribution is decreased by 1 if there exists some $l < m$ such that $i_{s_l} = i_b$ and $x_{s_l} = x_b q^{-2d_i}$ (such an $l$ must be unique due to the minimality of \eqref{eqn:collection of indices}). However, in this case a subset of the variables $(x_a,x_{s_1},\dots,s_{m-1})$ determines a wheel together with $x_b = q^r$ (in the terminology of \cite{FO}, see \cite{N Arbitrary}). Since elements of $\CS_{-\bn}$ vanish at wheels (see \eqref{eqn:wheel} for an example of this phenomenon), this implies that there is no actual pole at $z_b = x_b$ and thus the residue \eqref{eqn:residue ultimate} is 0.

\end{proof}

\medskip 

\section{Quantum toroidal $\fgl_1$}
\label{sec:toroidal}

\medskip

\subsection{A related setting}
\label{sub:toroidal}

We will now redevelop the theory in the preceding Subsections in the setting where quantum loop algebras are replaced by quantum toroidal $\fgl_1$. Although strictly speaking ill-defined, the latter algebra morally corresponds to the ``set of simple roots" $I = \{i\}$ and ``Cartan matrix" given by the $1 \times 1$ matrix $(0)$. Moreover, quantum toroidal $\fgl_1$ has two parameters $q_1,q_2$ instead of the single parameter $q$. Thus, in what follows, we will assume $q_1,q_2$ are non-zero complex numbers which do not satisfy $q_1^aq_2^b = 1$ for any $(a,b) \in \BZ^2 \backslash (0,0)$. For the remainder of the present paper, we will replace the zeta functions \eqref{eqn:zeta} by
\begin{equation}
\label{eqn:zeta toroidal}
  \zeta (x) = \frac {(xq_1 - 1)(x q_2 - 1)}{(x - 1)(xq_1q_2 - 1)}
\end{equation}
Compare the following with Definitions \ref{def:pre quantum loop} and \ref{def:quantum loop} (note that we will not include the index $i$ in the notation of the series \eqref{eqn:formal series}, as $I$ is a one-element set). It is often called the Ding-Iohara-Miki algebra, see \cite{DI, M}.

\medskip

\begin{definition}
\label{def:toroidal}

Quantum toroidal $\fgl_1$ (with one central element set equal to 1) is
$$
  \UUi = \BC \Big \langle e_{d}, f_{d}, \ph_{d'}^\pm \Big \rangle_{d \in \BZ, d' \geq 0} \Big/
  \text{relations \eqref{eqn:rel 0 tor}-\eqref{eqn:rel 3 tor}}
$$
where we impose the following relations for all $\pm,\pm' \in \{+,-\}$:
\begin{equation}
\label{eqn:rel 0 tor}
[[e_{d-1},e_{d+1}],e_d] = 0, \quad \forall d \in \BZ
\end{equation}
\begin{equation}
\label{eqn:rel 1 tor}
  e(x) e(y) \zeta \left( \frac yx \right) =\, e(y) e(x) \zeta \left(\frac xy \right)
\end{equation}
\begin{equation}
\label{eqn:rel 2 tor}
  \ph^\pm(y) e(x) \zeta \left(\frac xy \right) = e(x) \ph^\pm(y) \zeta \left( \frac yx \right)
\end{equation}
\begin{equation}
\label{eqn:rel 3 tor}
  \ph^{\pm}(x) \ph^{\pm'}(y) = \ph^{\pm'}(y) \ph^{\pm}(x), \quad
  \ph_{0}^+ \ph_{0}^- = 1
\end{equation}
as well as the opposite relations with $e$'s replaced by $f$'s, and finally the relation
\begin{equation}
\label{eqn:rel 4 tor}
  \left[ e(x), f(y) \right] =
  \frac {\delta \left(\frac xy \right) \Big( \ph^+(x) - \ph^-(y) \Big)}{(q_1-1)(q_2-1)(q_1^{-1}q_2^{-1}-1)}  
\end{equation}

\end{definition}

\medskip

\noindent The algebra $\UUi$ is graded by $\BZ \times \BZ$, with
\begin{equation}
\deg e_{d} = (1,d), \qquad \deg f_{d} = (-1,d), \qquad \deg \ph^\pm_{d'} = (0,\pm d')
\end{equation}
Let $\UUip, \UUim, \UUig, \UUil$ be the subalgebras of $\UUi$ generated by $\{e_d\}_{d\in \BZ}$, $\{f_d\}_{d\in \BZ}$, $\{e_d,\ph_{d'}^+\}_{d\in \BZ, d' \geq 0}$,  $\{f_d,\ph_{d'}^-\}_{d\in \BZ,d' \geq 0}$, respectively. The latter two subalgebras are topological Hopf algebras, using the formulas of Subsection \ref{sub:hopf pre-quantum}. With this in mind, quantum toroidal $\fgl_1$ is a Drinfeld double
\begin{equation}
\label{eqn:toroidal drinfeld double}
\UUi = \UUig \otimes \UUil
\end{equation}
with respect to the natural analogue of the Hopf pairing \eqref{eqn:pairing pre-quantum}
\begin{equation}
\label{eqn:toroidal pairing}
\UUig \otimes \UUil \xrightarrow{\langle \cdot, \cdot \rangle} \BC
\end{equation}
(in order to ensure that relation \eqref{eqn:rel 4 tor} holds, we must renormalize the pairing \eqref{eqn:pairing pre-quantum} so that the denominator of \eqref{eqn:pairing ef} is replaced by minus the denominator of \eqref{eqn:rel 4 tor}).

\medskip

\begin{remark}

There also exists an algebra called quantum toroidal $\fgl_n$ (\cite{GKV}), which is a two-parameter version of the quantum loop algebra of type $\widehat{A}_{n-1}$, endowed with an extra series of Cartan elements. All the constructions in the present Section generalize to quantum toroidal $\fgl_n$, see \cite{N Toroidal} for the respective shuffle algebra.

\end{remark}

\medskip

\subsection{The shuffle algebra}
\label{sub:shuffle gl1}

The quantum toroidal $\fgl_1$ version of the shuffle algebra was studied in \cite{FHHSY, N Shuffle}, to which we refer for proofs of all the results summarized in the present Subsection. The big shuffle algebra
\begin{equation}
\label{eqn:big shuffle toroidal}
\CV = \bigoplus_{n \in \BN} \CV_{n}, \quad \text{where} \quad \CV_{n} = \frac {\BC[z_{1}^{\pm 1},\dots,z_{n}^{\pm 1}]^{\text{sym}}}{\prod_{1 \leq a \neq b \leq n} (z_a q_1 q_2 - z_b)}
\end{equation}
is endowed with the shuffle product
\begin{multline}
\label{eqn:mult toroidal}
E( z_{1}, \dots, z_{n}) * E'(z_{1}, \dots,z_{n'}) = \frac 1{n!n'!} \cdot \\ \textrm{Sym} \left[ E(z_{1}, \dots, z_{n}) E'(z_{n+1}, \dots, z_{n+n'})  \prod_{1 \leq a \leq n < b \leq n+n'} \zeta \left( \frac {z_{a}}{z_{b}} \right) \right]
\end{multline}
Then one defines the (small) shuffle algebra
\begin{equation}
\label{eqn:e toroidal}
\CS^+ = \left\{  \frac {\rho(z_{1},\dots,z_{n})}{\prod_{1 \leq a \neq b \leq n} (z_a q_1 q_2 - z_b)} \right\}
\end{equation}
where $\rho$ goes over the set of symmetric Laurent polynomials that satisfy the so-called wheel conditions (\cite{FHHSY}):
\begin{equation}
\label{eqn:wheel toroidal}
\rho(w, wq_1,wq_1q_2,z_4,\dots,z_n) = \rho(w, wq_2,wq_1q_2,z_4,\dots,z_n) = 0
\end{equation}
It was shown in \cite{N Shuffle} that
\begin{equation}
\label{eqn:upsilon plus toroidal}
\Upsilon^+ : \UUip \stackrel{\sim}{\rightarrow} \CS^+, \qquad e_d \mapsto z_1^d \in \CV_1, \quad \forall d \in \BZ
\end{equation}
is an isomorphism. Set $\CS^- = \CS^{+,\text{op}}$, and define the double shuffle algebra
\begin{equation}
\label{eqn:double shuffle toroidal}
\CS = \CS^+ \otimes \frac {\BC[\ph_d^\pm]_{d \geq 0}}{\ph_0^+ \ph_0^- - 1} \otimes \CS^-
\end{equation} 
by the natural analogues of relations \eqref{eqn:shuffle plus commute}, \eqref{eqn:shuffle minus commute}, \eqref{eqn:shuffle plus minus commute}. In particular, if we write
\begin{equation}
\label{eqn:k and p toroidal}
\ph^\pm(y) = \kappa^{\pm 1} \exp \left(\sum_{u=1}^\infty \frac {p_{\pm u}}{uy^{\pm u}} \right)
\end{equation} 
then $\kappa$ is central, while for any $X \in \CS_{\pm n}$ we have
\begin{equation}
\label{eqn:p shuffle toroidal}
[p_u, X] = \pm X (z_1^u+\dots+z_n^u)(q_1^u-1)(q_2^u-1)(q_1^{-u} q_2^{-u}-1)
\end{equation}
With this in mind, the isomorphism \eqref{eqn:upsilon plus toroidal}, together with its opposite when $+$ is replaced by $-$, combine into an isomorphism
\begin{equation}
\label{eqn:upsilon toroidal}
\Upsilon : \UUi  \stackrel{\sim}{\rightarrow} \CS
\end{equation}
of $\BZ \times \BZ$ graded algebras, where we grade $\CS$ by
\begin{equation}
\label{eqn:grading shuffle toroidal}
\deg X = (\pm n, d)
\end{equation}
for any element $X(z_1,\dots,z_n) \in \CS^\pm$ of homogeneous degree $d$. We will call horizontal degree (denoted by ``hdeg") and vertical degree (denoted by ``vdeg") the two components of the grading \eqref{eqn:grading shuffle toroidal}, and denote the graded summands by
\begin{equation}
	\label{eqn:grading shuffle toroidal summands}
\CS^\pm = \bigoplus_{n \in \BN} \CS_{\pm n} = \bigoplus_{n \in \BN} \bigoplus_{d \in \BZ} \CS_{\pm n,d}
\end{equation}

\medskip

\subsection{Slope subalgebras}
\label{sub:slope toroidal}

Slope subalgebras for the shuffle algebra \eqref{eqn:e toroidal} predated (and served as inspiration for) those of Subsection \ref{sub:slope}. They were introduced in the context at hand in \cite{N Shuffle}, in order to give a shuffle algebra incarnation of the isomorphism (\cite{S}) between $\UUi$ and the elliptic Hall algebra of \cite{BS}. We will not repeat the definition of the subalgebras
\begin{equation}
\label{eqn:subalgebras toroidal}
\Big\{ \CS_{\geq \mu}^\pm, \ \CS_{\leq \mu}^\pm \text{ and } \CB_\mu^\pm \Big\}_{\mu \in \BQ}
\end{equation}
(nor the versions where $\geq,\leq$ are replaced by $>,<$) as they are word for word adaptations of the analogous notions in \eqref{eqn:shuffle slope geq}, \eqref{eqn:shuffle slope leq}, \eqref{eqn:slope subalgebra} for $\br = (1)$. In light of the isomorphism between $\CS$ and the elliptic Hall algebra, we have 
\begin{equation}
\label{eqn:slope toroidal}
\CB^\pm_{\frac dn} \cong \BC[p_{\pm n,\pm d}, p_{\pm 2n, \pm 2d}, p_{\pm 3n, \pm 3d}, \dots]
\end{equation}
(\cite[Theorem 4.3]{N Shuffle}) for any coprime integers $(n,d) \in \BZ_{>0} \times \BZ$, with $\deg p_{\pm nk, \pm dk} = (\pm nk, \pm dk)$ for all $k \geq 1$. The subalgebras $\CS_{\geq \mu}^\pm$ and $\CS_{\leq \mu}^\pm$ may be reconstructed from the slope subalgebras \eqref{eqn:slope toroidal} by the analogues of the factorizations \eqref{eqn:factorization 3}-\eqref{eqn:factorization 4}
\begin{align}
\bigotimes^{ \rightarrow}_{\mu \in [\nu, \infty)} \CB^\pm_{\mu} \xrightarrow{\sim} \ &\CS^\pm_{\geq \nu}  \xleftarrow{\sim} \bigotimes^{ \leftarrow}_{\mu \in [\nu, \infty)} \CB^\pm_{\mu} \label{eqn:factorization toroidal 1} \\ 
\bigotimes^{\rightarrow}_{\mu \in (-\infty,\nu]} \CB_\mu^\pm \xrightarrow{\sim} \ &\CS^\pm_{\leq \nu} \xleftarrow{\sim} \bigotimes^{\leftarrow}_{\mu \in (-\infty,\nu]} \CB_\mu^\pm \label{eqn:factorization toroidal 2} 
\end{align}
(as well as the analogous formulas with $\geq,\leq$ replaced by $>,<$ and the half-open intervals replaced by open intervals). By analogy with \eqref{eqn:factorization 1}, the algebras $\CS^\pm$ also have factorizations as above, but with $\mu$ going over $\BQ$. Thus, \eqref{eqn:slope toroidal} implies that
\begin{equation}
\label{eqn:pbw}
\CS^\pm = \bigoplus_{\text{convex path }v} \BC \cdot p_v^\pm
\end{equation}
Let us explain the notation in the right-hand side of the formula above. A sequence $v = \{(n_1,d_1), \dots, (n_k,d_k)\}$ of vectors in $\BZ_{>0} \times \BZ$ is called a \textbf{convex path} if
\begin{equation}
	\label{eqn:convex path}
\frac {d_1}{n_1} \leq \dots \leq \frac {d_k}{n_k} \quad \text{and} \quad n_a \leq n_{a+1} \text{ if } \frac {d_a}{n_a} = \frac {d_{a+1}}{n_{a+1}}
\end{equation}
and we write in \eqref{eqn:pbw}
\begin{equation}
	\label{eqn:pv}
p_{v}^\pm = p_{\pm n_1, \pm d_1} \dots p_{\pm n_k, \pm d_k}
\end{equation}
The elements \eqref{eqn:pv} are orthogonal, with pairing given by (\cite[Proposition 5.4]{N Shuffle})
\begin{equation}
	\label{eqn:pairing toroidal}
	\Big \langle p_v^+,p_{v'}^- \Big \rangle = \delta_{vv'} \frac {\prod_{(n,d) \in \BZ_{>0} \times \BZ} \# \{a\text{ s.t. } (n_a,d_a) = (n,d)\}!}{\prod_{a=1}^k (1-q_1^{g_a})(1-q_2^{g_a})(1-q_1^{-g_a} q_2^{-g_a})}
\end{equation}
where $g_a = \gcd(n_a,d_a)$ for all $a \in \{1,\dots,n\}$. The algebras $\CS^\pm$ have bases analogous to \eqref{eqn:pbw} indexed by concave paths, i.e. flipping the inequalities in \eqref{eqn:convex path}. Alternatively, we can retain the convention that $v = \{(n_1,d_1), \dots, (n_k,d_k)\}$ denotes a convex path, but consider the basis of elements
\begin{equation}
	\label{eqn:reversed path}
p_{\text{rev}(v)}^\pm = p_{\pm n_k, \pm d_k} \dots p_{\pm n_1,\pm d_1}
\end{equation}
of $\CS^\pm$ instead of \eqref{eqn:pv}. However, as opposed from the $p_v^\pm$'s, the $p_{\text{rev}(v)}^\pm$'s are not orthogonal with respect to the pairing.

\medskip

\noindent Using \eqref{eqn:factorization toroidal 1}-\eqref{eqn:factorization toroidal 2}, we may construct the following analogues of \eqref{eqn:a geq} and \eqref{eqn:a leq}
\begin{align}
&\CA^{\geq}_\nu = \CS^-_{<\nu}  \otimes \CB_\infty^+ \otimes \CS^+_{\geq \nu}=  \bigotimes^{\leftarrow}_{\mu \in (-\infty,\nu)} \CB_\mu^- \otimes \bigotimes^{\leftarrow}_{\mu \in [\nu,\infty]} \CB_\mu^+ \label{eqn:eqref 1} \\
&\CA^{\leq}_\nu = \CS^+_{<\nu}  \otimes \CB_\infty^- \otimes \CS^-_{\geq \nu}= \bigotimes^{\leftarrow}_{\mu \in (-\infty,\nu)} \CB_\mu^+ \otimes \bigotimes^{\leftarrow}_{\mu \in [\nu,\infty]} \CB_\mu^- \label{eqn:eqref 2}
\end{align}
which are subalgebras of $\CS$ by the natural analogue of Proposition \ref{prop:a are algebras}. The analogue of Proposition \ref{prop:a triangular} also holds, in that multiplication induces isomorphisms 
\begin{equation}
\label{eqn:triangular toroidal}
\CA^{\geq}_\nu \otimes \CA^{\leq}_\nu \xrightarrow{\sim} \CS \xleftarrow{\sim} \CA^{\leq}_\nu \otimes \CA^{\geq}_\nu
\end{equation}
for all $\nu \in \BQ$. It is known that the subalgebras $\CA^{\geq}_{\nu}$ and $\CA^{\leq}_{\nu}$ are actually topological Hopf algebras, and \eqref{eqn:triangular toroidal} is a Drinfeld double type decomposition. 

\medskip

\subsection{Category $\CO$ for quantum toroidal $\fgl_1$} 
\label{sub:category O toroidal}

The analogue of category $\CO$ for quantum toroidal $\fgl_1$ was studied in \cite{FJMM} in connection with the Bethe ansatz in the theory of integrable systems. In a nutshell, one considers representations
\begin{equation}
\label{eqn:a geq toroidal}
\CA^{\geq} = \CA^{\geq}_0 \curvearrowright V
\end{equation}
which have abstract weight decompositions
\begin{equation}
	\label{eqn:weight toroidal}
V = \bigoplus_{n \in \BC} V_n
\end{equation}
where the $V_n$ are finite-dimensional, and non-zero only for $n$ in a finite number of translates of $-\BN$. The difference between the case at hand and that of quantum loop algebras in Definitions \ref{def:category o affine} and \ref{def:category o general} is that the weight grading above cannot be deduced from the Cartan subalgebra of $\CA^\geq$. Instead, one must either enlarge quantum toroidal $\fgl_1$ by introducing an extra element $D^\perp$ which keeps track of the horizontal degree (as was done in \cite{FJMM}), or equivalently, one must assume that the weight decompositions \eqref{eqn:weight toroidal} have the property that
\begin{equation}
\label{eqn:weight commute toroidal}
x \cdot V_n \subseteq V_{n + \hdeg x}, \quad \forall x \in \CA^\geq
\end{equation}
We assume \eqref{eqn:weight commute toroidal} in the present paper. By analogy with Subsections \ref{sub:representations} and \ref{sub:q-characters affine}, simple graded $\CA^\geq$ modules are in one-to-one correspondence with highest $\ell$-weights
$$
\bpsi = (\psi(z),m) \in \BC[[z^{-1}]]^* \times \BC
$$
where $m$ indicates the highest weight. The corresponding simple module $L(\bpsi)$ is in category $\CO$ if and only if $\psi(z)$ is the expansion of a rational function. The $q$-character of a representation in category $\CO$ is defined in \cite{FJMM} as
\begin{equation}
\label{eqn:q-character toroidal}
\chi_q(V) = \sum_{\bpsi \in \BC[[z^{-1}]]^* \times \BC} \dim_{\BC}( V_{\bpsi})[\bpsi]
\end{equation}
where $V_{(\psi(z),m)}$ is the generalized eigenspace of $V$ for the series $\ph^+(z) \in \CA^{\geq}[[z^{-1}]]$, corresponding to the eigenvalue $\psi(z)$, intersected with the weight subspace $V_m$. 

\medskip

\subsection{Simple modules}
\label{sub:simple toroidal}

For any $\ell$-weight $\bpsi = (\psi(z),m)$, we have a representation
\begin{equation}
\CA^\geq \curvearrowright W(\bpsi)
\end{equation}
generated by a vector $\vac \in V_m$ modulo the relations $\ph^+(z)\cdot \vac = \psi(z)\vac$ and $E \cdot \vac = 0$ for any $E \in \CS_{\geq 0|n}$ with $n > 0$. By analogy with Subsection \ref{sub:simple}, the simple module with highest $\ell$-weight $\bpsi$ arises as the quotient
\begin{equation}
L(\bpsi) = W(\bpsi) \Big/ J(\bpsi)\vac
\end{equation}
where $J(\bpsi) = \bigoplus_{n \in \BN} J(\bpsi)_n$ is defined as follows: let $J(\bpsi)_n$ be the set of $F \in \CS_{<0|-n} $ such that
\begin{equation}
\label{eqn:psi pairing toroidal}
\left \langle E(z_1,\dots,z_n) \prod_{a=1}^n \psi(z_a), S(F(z_1,\dots,z_n)) \right \rangle = 0, \quad \forall E \in \CS_{\geq 0|n}
\end{equation}
By analogy with \eqref{eqn:q-character general}, we therefore have the following formula for the $q$-character
\begin{equation}
\label{eqn:q-character general toroidal}
\chi_q(L(\bpsi)) = [\bpsi] \cdot
\end{equation}
$$
\cdot \sum_{n \in \BN} \sum_{\bx = (x_1,\dots,x_n)\in \BC^n/S_n}  \mu_{\bx}^{\bpsi} \left[ \left( \prod_{a=1}^n \frac {(z - x_a q_1)(z - x_a q_2)(z q_1 q_2 - x_a)}{(z q_1 - x_a)(z q_2 - x_a)(z - x_a q_1 q_2)} , -n \right)\right]
$$
where the multiplicities in the formula above are given by
\begin{equation}
\label{eqn:multiplicity toroidal}
\mu_{\bx}^\bpsi = \dim_{\BC} \left(\CS_{<0|-n} \Big / J(\bpsi)_{n} \right)_{\bx}
\end{equation}
The vector space in the right-hand side of \eqref{eqn:multiplicity toroidal} is the fiber at $\bx \in \BC^n/S_n$ of the $\BC[z_1,\dots,z_n]^{\text{sym}}$-module $\CS_{<0|-n} / J(\bpsi)_{n}$. In formula \eqref{eqn:q-character general toroidal} and henceforth, the product of symbols $[\bpsi]$ is defined to be multiplicative in $\psi(z)$ and additive in $m$.

\medskip

\subsection{Explicit computations}
\label{sub:explicit toroidal}

As an illustration of formulas \eqref{eqn:q-character general toroidal} and \eqref{eqn:multiplicity toroidal}, let us calculate the $q$-character corresponding to a polynomial $\ell$-weight, i.e.
\begin{equation}
\label{eqn:integral tau toroidal}
\btau = \Big( \tau(z) = a_0 + a_1z^{-1} + \dots + a_{r-1} z^{-r+1} + a_rz^{-r}, m\Big)
\end{equation}
for various complex numbers $a_0, a_1, \dots, a_r,m$ such that $a_0,a_r \neq 0$. Just like in Proposition \ref{prop:integral}, a shuffle element $F \in \CS_{<0|-n}$ lies in $J(\btau)_n$ if and only if
\begin{equation}
\label{eqn:condition equivalent 1}
\left \langle E(z_1,\dots,z_n) \prod_{a=1}^n z_a^{-r}, S(F(z_1,\dots,z_n)) \right \rangle = 0, \qquad \forall E \in \CS_{\geq 0|n}
\end{equation}
Since the map $E(z_1,\dots,z_n)\mapsto E(z_1,\dots,z_n) \prod_{a=1}^n z_a^{-r}$ yields an isomorphism $\CS_{\geq 0} \xrightarrow{\sim} \CS_{\geq -r}$ (by the analogue of Proposition \ref{prop:easy}), then \eqref{eqn:condition equivalent 1} is equivalent to
\begin{equation}
	\label{eqn:condition equivalent 2}
\Big \langle E(z_1,\dots,z_n), S(F(z_1,\dots,z_n)) \Big \rangle = 0, \qquad \forall E \in \CS_{\geq -r|n}
\end{equation}

\medskip

\begin{lemma}
\label{lem:concave}

The basis elements in \eqref{eqn:pv} and \eqref{eqn:reversed path} satisfy the equality
\begin{equation}
\label{eqn:antipode convex}
S(p_{\emph{rev}(v)}^-) \in (-1)^{l(v)} p_v^- \kappa^{|v|} + \sum_{v' \prec v} \BC \cdot p_{v'}^-
\end{equation}
where for any convex path $v = \{(n_1,d_1),\dots,(n_k,d_k)\}$, we write $l(v) = k$ and $|v| = n_1+\dots+n_k$, and $v' \preceq v$ means that the convex paths obtained by stringing together the vectors of $v'$ and $v$ have the same start and end points, but at every $x$-coordinate the former has $y$-coordinate less than or equal to the latter (we also write $v' \prec v$ if $v' \preceq v$ and $v' \neq v$). In the right-hand side of \eqref{eqn:antipode convex}, we allow convex paths to contain vectors of infinite slope, corresponding to elements of $\CB_\infty^-$.

\end{lemma}

\medskip

\begin{proof} Formula \eqref{eqn:antipode convex} is a consequence of the properties of the antipode in the elliptic Hall algebra (\cite{BS}), but we will prove it by induction on $|v|$. The base case will be when all the constituent vectors of the convex path $v$ have the same slope $\mu$.  By analogy with \eqref{eqn:slope 2}-\eqref{eqn:slope 4}, we have
$$	
\Delta(p_{\text{rev}(v)}^-) \in \Delta_{\mu}(p_{\text{rev}(v)}^-) + \CS^-_{<\mu} \otimes \CS^{\leq}_{>\mu}
$$
where $\Delta_{\mu} : \CB_{\mu}^- \rightarrow \CB_{\mu}^- \otimes \CB_{\mu}^-[\kappa^{-1}]$ corresponds to the Hall coproduct on the ring of symmetric polynomials under the isomorphism \eqref{eqn:slope toroidal}. If we apply $\text{Id} \otimes S$ to the above equality and then multiply the tensor factors, we conclude that $S|_{\CB^-_\mu}$ matches the Hall antipode on the ring of symmetric polynomials, modulo elements of 
$$
\CS^-_{\leq \mu} \cdot \CS^{\leq}_{\geq \mu}
$$
which correspond to convex paths that go below the line of slope $\mu$. Since the Hall antipode is determined by $p_{-n,-d} \mapsto -p_{-n,-d}\kappa^n$ for all $(n,d) \in \BZ_{>0} \times \BZ$ of slope $\mu$, we conclude precisely \eqref{eqn:antipode convex}. 
	
\medskip
	
\noindent Let us now prove the induction step, for which we may assume that the convex path $v$ contains vectors of different slopes. We may then express $v$ as the concatenation of convex paths $v'$ and $v''$, where all the constituent vectors of $v'$ have strictly smaller slope that all the constituent vectors of $v''$. Therefore, we have
\begin{multline*}
S(p_{\text{rev}(v)}^-)  =  S(p_{\text{rev}(v')}^-) S(p_{\text{rev}(v'')}^-) = \\ = (-1)^{l(v')+l(v'')} p_{v'}^- p_{v''}^- \kappa^{|v'|+|v''|}+ \dots = (-1)^{l(v)} p_{v}^- \kappa^{|v|} + \dots
\end{multline*}
where the ellipsis denotes $p$'s corresponding to the concatenation of a convex path $\preceq v'$ with the concatenation of a convex path $\preceq v''$ (other than the concatenation of $v'$ with $v''$, which is just $v$). By the ``straightening" argument in \cite[Lemma 5.6]{BS}, these ellipsis terms may be expressed as linear combinations of $p_{\tilde{v}}^-$ for $\tilde{v} \prec v$.

\end{proof}

\medskip

\begin{proposition}
\label{prop:q-character integral toroidal}

The $q$-character of $L(\btau)$ for a polynomial $\ell$-weight \eqref{eqn:integral tau toroidal} is
\begin{equation}
\label{eqn:q-character integral toroidal}
\chi_q(L(\btau)) = [\btau] \prod_{n=1}^{\infty} \left( \frac 1{1-h^n} \right)^{rn}
\end{equation}
where $h = [(1,-1)]$ is the $\ell$-weight whose first component is the constant rational function 1, and whose second component reflects a grading shift of $-1$.

\end{proposition}

\medskip

\begin{proof} We will argue as in Lemma \ref{lem:composition}. By formula \eqref{eqn:antipode convex}, for any convex path $v = \{(n_1,d_1),\dots,(n_k,d_k)\}$ satisfying $\frac {d_1}{n_1} \leq \dots \leq \frac {d_k}{n_k} < 0$, the shuffle element
$$
F = p_{\text{rev}(v)}^- 
$$
enjoys the following with respect to any convex path $v' = \{(n_1',d_1'),\dots,(n'_{k'},d_{k'}')\}$ satisfying $-r \leq \frac {d_1'}{n_1'} \leq \dots \leq \frac {d'_{k'}}{n'_{k'}}$
$$
\Big \langle p_{v'}^+, S(F) \Big \rangle \stackrel{\eqref{eqn:pairing toroidal}}= \begin{cases} 0 &\text{if }-r> \frac {d_1}{n_1} \\ 0 &\text{if }-r \leq \frac {d_1}{n_1} \text{ and } v \prec v' \\ \neq 0 &\text{if } -r \leq \frac {d_1}{n_1} \text{ and } v = v' \end{cases}
$$
Then \eqref{eqn:condition equivalent 2} and (the concave path version of) \eqref{eqn:pbw} imply that 
\begin{equation}
\label{eqn:pbw simple}
\CS_{<0}^- \Big / J(\btau) = \bigoplus_{-r \leq \frac {d_1}{n_1} \leq \dots \leq \frac {d_k}{n_k} < 0} \BC \cdot p_{\text{rev}(v)}^-
\end{equation}
Since each $p_{-n,-d}$ contributes a factor of $h^n$ to the $q$-character, this implies \eqref{eqn:q-character integral toroidal}.

\end{proof}

\medskip

\subsection{An extra grading}
\label{sub:extra}

Formula \eqref{eqn:refined} applies equally well to quantum toroidal $\fgl_1$: for any polynomial $\ell$-weight $\btau$ as in \eqref{eqn:integral tau toroidal}, the set $J(\btau)$ is graded with respect to vertical degree, thus allowing us to define the refined $q$-character 
\begin{equation}
\label{eqn:refined toroidal}
\chi_q^{\text{ref}}(L(\btau)) = [\btau] \sum_{n \in \BN} \sum_{d=0}^{\infty} \dim_{\BC}  \left(\CS_{<0|-n,d} \Big/ J(\btau)_{n,d} \right) h^n v^d
\end{equation}
We will now prove Theorem \ref{thm:toroidal}, which states that the refined $q$-character matches the notion considered in \cite[Section 4.2]{FJMM}, and proves the explicit formula in \cite[Conjecture 4.20]{FJMM}. 

\medskip

\begin{proof} \emph{of Theorem \ref{thm:toroidal}:} The explicit formula 
	\begin{equation}
		\label{eqn:explicit formula}
		\chi_q^{\text{ref}}(L(\btau)) = [\btau] \prod_{n =1}^{\infty} \prod_{d=1}^{rn} \frac 1{1-h^nv^d}
	\end{equation}	
is an immediate consequence of \eqref{eqn:pbw simple} and the fact that each generator $p_{-n,-d}$ contributes a factor of $h^nv^d$ to the refined $q$-character. The action of $\CA^\geq$ on $L(\btau)$ interacts with the horizontal and vertical gradings by
\begin{equation}
	\label{eqn:act tor 1}
	F \cdot L(\btau)_{m-n,d} \subseteq L(\btau)_{m-n+\text{hdeg } F, d+\text{vdeg } F}
\end{equation}
\begin{equation}
	\label{eqn:act tor 2}
	\left[\frac {\ph^+(z)}{\tau(z)} \right]_{z^{-u}} \cdot L(\btau)_{m-n,d} \subseteq L(\btau)_{m-n,d+u}
\end{equation}
\begin{equation}
	\label{eqn:act tor 3}
	E \cdot L(\btau)_{m-n,d} \subseteq \bigoplus_{\bullet = 0}^{r(\text{hdeg } E)} L(\btau)_{m-n+\text{hdeg }E, d+\text{vdeg } E - \bullet}
\end{equation}
for any $F,\ph^+(z),E$ in the creating, diagonal, annihilating part of $\CA^{\geq}$, respectively, see \eqref{eqn:eqref 1}. These formulas are proved by analogy with Proposition \ref{prop:bigrading}.
	
\end{proof}

\end{document}